\documentclass[11pt]{amsart}

\usepackage{amssymb,mathrsfs,graphicx,extpfeil}
\usepackage{amsmath,amsfonts,amssymb,amscd,amsthm,bbm}

\usepackage{epsfig}
\usepackage{indentfirst, latexsym, amssymb, enumerate,amsmath,graphicx}
\usepackage{float}
\usepackage{comment}
\usepackage{caption}
\usepackage{subcaption}
\usepackage{extpfeil}
\usepackage{graphicx,colortbl}
\usepackage{epsfig}
\usepackage{caption}
\usepackage{subcaption}
\usepackage{tcolorbox}
\usepackage{bm}
\usepackage{tikz}
\usepackage{tikz-cd}
\usepackage{soul}
\usetikzlibrary{matrix,shapes,arrows,positioning,chains}

\usepackage{amsmath}
\newcommand{\sgn}{\text{sgn}}
\title[The Winfree model with a high-order uncertain couplings]{Measure-valued death state and local sensitivity analysis for  Winfree models with uncertain high-order couplings}

\author[Ha]{Seung-Yeal Ha}
\address[Seung-Yeal Ha]{\newline Department of Mathematical Sciences and Research Institute of Mathematics \newline Seoul National University, Seoul 08826, Republic of Korea}
\email{syha@snu.ac.kr}

\author[Kang]{Myeongju Kang}
\address[Myeongju Kang]{\newline Department of Finance and Big Data \newline Gachon University, Seongnam 13120, Republic of Korea}
\email{mathemjkang@gachon.ac.kr}

\author[Yoon]{Jaeyoung Yoon}
\address[Jaeyoung Yoon]{\newline {\color{black}Department of Mathematics, School of Computation, Information and Technology,\newline Technical University of Munich, Boltzmannstraße 3, 85748 Garching bei M\"unchen, Germany}}
\email{\color{black}wodud1516@gmail.com}

\author[Zanella]{Mattia Zanella}
\address[Mattia Zanella]{\newline Department of Mathematics\newline University of Pavia, Pavia 27100, Italy}
\email{mattia.zanella@unipv.it}

\newtheorem{theorem}{Theorem}[section]
\newtheorem{lemma}{Lemma}[section]
\newtheorem{corollary}{Corollary}[section]
\newtheorem{proposition}{Proposition}[section]
\newtheorem{remark}{Remark}[section]

\newtheorem{definition}{Definition}[section]

\newcommand{\bbr}{\mathbb R}

\newcommand{\bbt} {\mathbb T}

\def\charf {\mbox{{\text 1}\kern-.30em {\text l}}}

\calclayout

\makeatletter
\newcommand{\newparallel}{\mathrel{\mathpalette\new@parallel\relax}}
\newcommand{\new@parallel}[2]{%
\begingroup
\resizebox{!}{\htz@}{\raisebox{\depth}{$\m@th#1/\mkern-5mu/$}}%
\endgroup
}
\makeatother

\begin{document}
\tikzstyle{block} = [rectangle, draw, 
text width=15em, text centered, rounded corners, minimum height=3em]
\tikzstyle{line} = [draw, -latex']

\date{\today}
\subjclass[2021]{39A30, 34E10, 65L05}
\keywords{Local sensitivity analysis, oscillator death, phase-coupled oscillators, synchronization, uncertainty quantification, Winfree model}

\thanks{Acknowledgment: The work of S.-Y. Ha is partially supported by National Research Foundation of (NRF-2020R1A2C3A01003881). {\color{black}J. Yoon acknowledges the support of the Alexander von Humboldt foundation.} M. Zanella acknowledges the support of the GNFM group of INdAM (National Institute of High Mathematics) and of the ICSC – Centro Nazionale di Ricerca in High Performance Computing, Big Data and Quantum Computing, funded by European Union – NextGeneration EU. M. Zanella wishes to acknowledge partial support by MUR-PRIN2022PNRR Project No. P2022Z7ZAJ. {\color{black}All authors contributed equally to this work.}}

\begin{abstract}
We study the measure-valued death state and local sensitivity analysis of the Winfree model and its mean-field counterpart with uncertain high-order couplings. The Winfree model is the first mathematical model for synchronization, and it can cast as the effective approximation of the pulse-coupled model for synchronization, and it exhibits diverse asymptotic patterns depending on system parameters and initial data. For the proposed models, we present several frameworks leading to oscillator death in terms of system parameters and initial data, and the propagation of regularity in random space. We also present several numerical tests and compare them with analytical results.
\end{abstract}

\maketitle

\centerline{\date}



\section{Introduction} \label{sec:1}
\setcounter{equation}{0}
Collective behaviors of interacting particle systems are often observed in nature. Among them, we are interested in synchronization in which oscillators exhibit the same rhythms due to weak interactions between constituent oscillators, e.g., flashing of fireflies and firing of neurons \cite{B-B,Ku2,Pe,P-R}, etc. Despite its ubiquity, mathematical modeling for synchronization was done by Arthur Winfree \cite{Wi2, Wi1} only in a half century ago. After Winfree's pioneering works, studies on emergent behavior \cite{GMP,HJM,HKPR,KR} and several variants have been proposed in literature, e.g. the Kuramoto model \cite{Ku2}.  In this work, we are mainly interested in the Winfree model. To set up the stage, we begin with a brief description.

Let $\theta_i = \theta_i(t)$ be the phase of the $i$-th Winfree oscillator on  the unit circle with the natural frequency $\nu_i$. Then, the Winfree model reads as
\begin{align}\label{A-1}
	\dot\theta_i = \nu_i +\frac{\kappa}{N}\sum_{j=1}^N S(\theta_i) I(\theta_j), \quad \forall~i\in[N] := \{1,\cdots, N\},
\end{align}
where $\kappa$ is the nonnegative coupling strength, and $S, I:\bbr~\to~\bbr$ are $2\pi$-periodic functions which are called sensitivity and influence functions, respectively, and they will be specified explicitly later.

In fact, the Winfree model \eqref{A-1} belongs to the class of \emph{phase-coupled model} for synchronization, which assumes that time-varying behavior of oscillators depends only on the current phase of oscillators by assuming the variation of amplitude to be unity. In contrast, there is another class of synchronization models, namely \emph{pulse-coupled model} taking into account the firing of oscillators. To name a few, integrate-and-fire model and the Peskin model correspond to the class of pulse-coupled model. These types of synchronization models are widely used to describe the synchronization of real neurons \cite{AS, B-B, Pe, P-R, Wi1}. However, it is very difficult to analyze a pulse-coupled model mathematically, since the influence function is usually described by Dirac-delta function $\delta_0$ so that the right-hand side of \eqref{A-1} provides a non-smooth vector field. To overcome this technical difficulty, authors in \cite{AS} approximated the pulse-coupled mechanism using the Winfree model \eqref{A-1} by employing following \emph{pulse-like} influence functions:
\begin{align} \label{A-2}
I_n(\theta) = a_n (1+\cos\theta)^n, \quad n \ge 1, \quad a_n := \frac{(2n)!!}{2^n(2n-1)!!},
\end{align}
where the coefficient $a_n$ is chosen to satisfy a normalization condition $\displaystyle \int_0^{2\pi}I_n(\theta)d\theta = 2\pi$. It is well-known that on $[-\pi, \pi]$, $I_n$ converges to $2\pi\delta_0$ in distribution as $n \to \infty$. This implies that $I_n$ approximates the impulse fired when oscillator touches zero. In addition, for the sensitivity function, they employed
\begin{align} \label{A-4}
	S(\theta) = -\sin\theta
\end{align}
to recast the Winfree model as a generalized Adler equation (see Section \ref{sec:2.1} for detail).  Now, we combine \eqref{A-2} and \eqref{A-4} to propose the Winfree model (\cite{AS}) with a high-order coupling:
\begin{align} \label{A-5}
	\dot{\theta}_i=\nu_i-\frac{\kappa}{N}\sin\theta_i\sum_{j=1}^NI_n(\theta_j), \quad \forall~i\in[N].
\end{align}
When we consider the real applications of \eqref{A-5},  the determination of the order $n$ is based on the phenomenology or modeler’s free will. Hence, the uncertainty (\cite{APZ,HJR}) in the order $n$ is intrinsic, and it is natural to ask how uncertainty in the order $n$ can affect the collective dynamics of system \eqref{A-5}. In this paper, we are interested in the Winfree model with uncertain high-order couplings:
\begin{align}
\begin{cases} \label{A-6}
\displaystyle \partial_t \theta_i(t,z)=\nu_i-\frac{a(z)\kappa}{N}\sin\theta_i(t,z)\sum_{j=1}^N\big(1+\cos\theta_j(t,z)\big)^z, \quad \forall~t>0, \\
\displaystyle \theta_i(0,z) = \theta^{\mathrm{in}}_i{\color{black}(z)}, \quad \forall~i\in[N],
	\end{cases}
\end{align}
where $z= z(\omega)$ is a $[1,\infty)$-valued random variable on a given probability space $(\Omega,\mathcal{F},\mathbb{P})$ and the normalizing factor $a(z)$ is given to satisfy the normalization condition:
\begin{align}\label{def_az}
	a(z) := \frac{2^z\left(\Gamma(z+1)\right)^2}{\Gamma(2z+1)}\quad\mbox{so that}\quad\int_0^{2\pi}a(z)(1+\cos\theta)^z d\theta=2\pi.
\end{align}
{\color{black}To briefly highlight the challenges in analysis, high-order couplings ($z>1$) disrupt the gradient flow ($z=1$), leading to a lack of uniform boundedness. Furthermore, the variation of $\theta$-dynamics in $z$ cannot be explicitly computed due to the complexity in the structure of $a(z)$. In detail, $\partial_za(z)$ is considered when we study the propagation of regularity in random space (see Section \ref{sec:5} for details). To address this difficulty, we adopt random variables $z$ in $L^\infty(\Omega)$ (See Definition \ref{D2.2}).} From now on, we call the system \eqref{A-6} as the {\it random Winfree model}.\newline
\indent {\color{black}Next, we} consider the kinetic counterpart of \eqref{A-6} which can be obtained as a mean-field approximation of \eqref{A-6}.  Let $f = f(t, \theta, \nu, z)$ be the probability density function of random Winfree oscillators at time $t\in[0, \infty)$, phase $\theta\in\bbt$, natural frequency $\nu\in\bbr$, and random variable $z\in[1,\infty)$. Then, the formal BBGKY hierarchy yields the following Vlasov-type equation \cite{HPZ}:
\begin{align}
\begin{cases} \label{A-7}
\displaystyle \partial_t f +\partial_\theta(f L[f]) = 0, \quad (t,\theta, \nu, z) \in \bbr_+ \times {\mathbb T} \times \bbr \times [1, \infty), \\
\displaystyle L[f](t, \theta, \nu, z) = \nu-\sigma(t,z) \sin\theta,\\
\displaystyle \sigma(t,z) = \kappa a(z) \int_{\mathbb T \times \mathbb R} (1+\cos\theta_*)^z f(t, \theta_*, \nu_*, z) g(\nu_*) d\nu_* d\theta_*,
\end{cases}
\end{align}
where $g: \bbr\to\bbr$ is the probability density function for a natural frequency $\nu$. We call the equation \eqref{A-7} as the {\it random kinetic Winfree model}. In this paper, we address the following two issues:
\vspace{0.5cm}
\begin{itemize}
\item
(Q1):~How does randomness characterized by random variable $z: \Omega \to [1, \infty)$ affect in the propagation of regularity in random space?
\vspace{0.2cm}
\item
(Q2):~Can we find an explicit measure-valued stationary death state (see Section \ref{sec:4})?
\end{itemize}
\vspace{0.4cm}

In this paper, we provide quantitative estimates and numeric tests for the above two questions (Q1) - (Q2). More precisely, our main results in this work can be categorized as follows.

The first set of results is concerned with the existence of oscillator death state for \eqref{A-6} and \eqref{A-7} (see Section \ref{sec:3.1} and Section \ref{sec:4}). For the Winfree model \eqref{A-6}, it suffices to check that there exists a bounded trapping set in which oscillators are attracted to it in finite or infinite time. In this case, all the rotation numbers $\{ \rho_i \}_{i\in[N]}$ (see \eqref{B-5-1} for definition) are zero:
\[ \rho_i \equiv 0, \quad \forall~i \in [N]. \]
For the construction of such bounded trapping set, we assume that initial data and coupling strength satisfy the following conditions:
\begin{align} \label{A-7-1}
\max_{i \in [N]}  |\theta_i^{\mathrm{in}}| < c < \pi, \quad \forall~i \in [N] \quad \mbox{and} \quad  \kappa>\frac{\|\mathcal{V}\|_\infty}{a(z)\sin c(1+\cos c)^z}, \qquad \mbox{almost surely}.
\end{align}
Under these conditions, all the phases will be trapped in a bounded region in $\bbr^N$ (see Proposition \ref{P2.1}), i.e.,
\[  \sup_{0 \leq t < \infty}  |\theta_i(t,z)| \leq c, \quad  \mbox{almost surely}. \]
Hence the random phase configuration $\Theta(t, z)$ approaches to death state asymptotically almost surely. Under the conditions \eqref{A-7-1}, we show that system \eqref{A-6} admits the existence of a unique equilibrium and that a dynamic solution converges to it exponentially fast almost surely  (see Proposition \ref{P3.1} and Theorem \ref{T3.1}). For the kinetic model \eqref{A-7}, the measure-valued stationary death state can be constructed explicitly based on the solvability of the following integral equation:
\begin{align} \label{A-8}
x = \kappa a(z) \int_{-\infty}^{\infty} \left( 1+\sqrt{1-\nu_*^2/x^2} \right)^z g(\nu_*) d\nu_*, \quad \forall~z \in [1, \infty).
\end{align}
Once the integral equation \eqref{A-8} has a solution, say $\tilde{\sigma}$, then, the state $f^{\infty}(\theta, \nu, z) = \delta(\theta-\theta^*(\nu,z))$ with $\sin\theta^*(\nu,z) = \frac{\nu}{{\tilde \sigma}(z)}$ is a measure-valued stationary death state to \eqref{A-7}. For details, we refer to Proposition \ref{P4.1}.

The second set of analytical results deal with the propagation of regularity in random space for \eqref{A-6} and \eqref{A-7}.  First, we show that $H^1_z$-norm of $\Theta$ is uniformly bounded (see Theorem \ref{T3.2}):
\[
\sup_{0 \leq t < \infty} \Big( \int_{[1,\infty)} |\Theta(t,z)|^2 \rho(z)dz + \int_{[1,\infty)} |\partial_z\Theta(t, z)|^2 \rho(z)dz \Big) < \infty,
\]
where $|\cdot |$ is the standard $\ell^2$-norm in $\bbr^N$. On the other hand, for the kinetic Winfree model \eqref{A-7}, we assume that coupling function $I = I(\theta,z)$ and initial data $f^{\mathrm{in}}$ satisfy
\[
 k\geq 1, \quad  p>1, \quad T > 0, \quad \max_{0\leq l\leq k} \|\partial^l_zI\|_{L^\infty_{\theta,z}} <\infty, \quad  \partial_\theta^j f^{\mathrm{\mathrm{in}}}\in L^\infty_{\theta, \nu, z}, \quad 0\leq j\leq k,
\]
Then, for a global solution $f$ to \eqref{A-7}, there exists $\Lambda_{1,k} = \Lambda_{1,k}(T, C, \kappa, f^{\mathrm{\mathrm{in}}})$ such that
\[ \|f\|_{W^{k,p}_{\theta, \nu, z}} \leq {\mathcal O}(1)  \Big(  \|f^{\mathrm{\mathrm{in}}}\|_{L^p_{\theta, \nu, z}} + 1 \Big) e^{\Lambda_{1,k}t}, \quad t \geq 0. \]
We refer to Theorem \ref{T5.1} for details. \newline

\indent In numeric section, we first provide how accurately the kinetic description \eqref{A-7} approximates the particle random Winfree model \eqref{A-6}. For simulations, we use generalized polynomial chaos expansion and stochastic Galerkin particle methods in \cite{PZ_SP,XK,Zan}. Hence, we also present the accuracy of the stochastic Galerkin method in producing results based on the degree of the expansion under two regimes for random variable $z$: (i) uniform distribution, (ii) Gaussian distribution (See Figure \ref{Fig_rv}).\newline
\begin{figure}
	\centering
	\mbox{\includegraphics[width=0.5\textwidth]{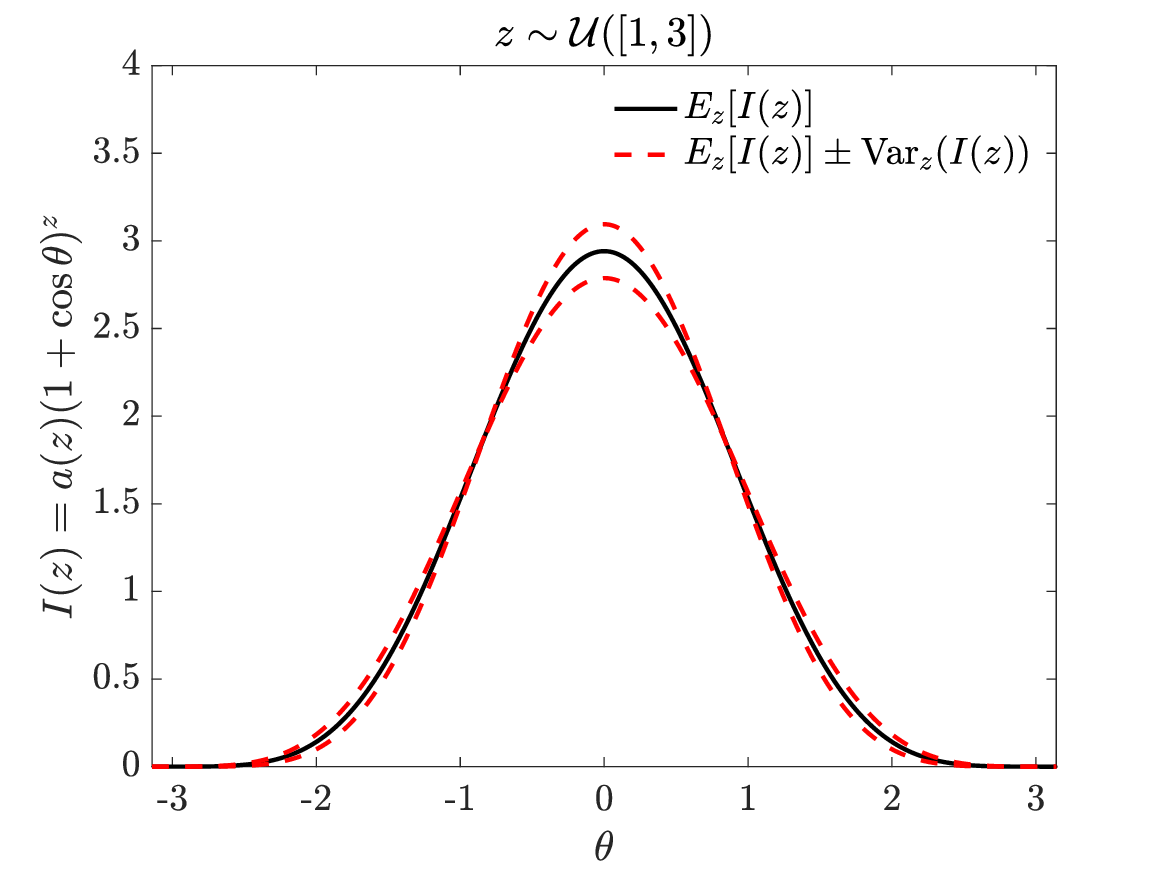}
		\hspace{-.0cm}
		\includegraphics[width=0.5\textwidth]{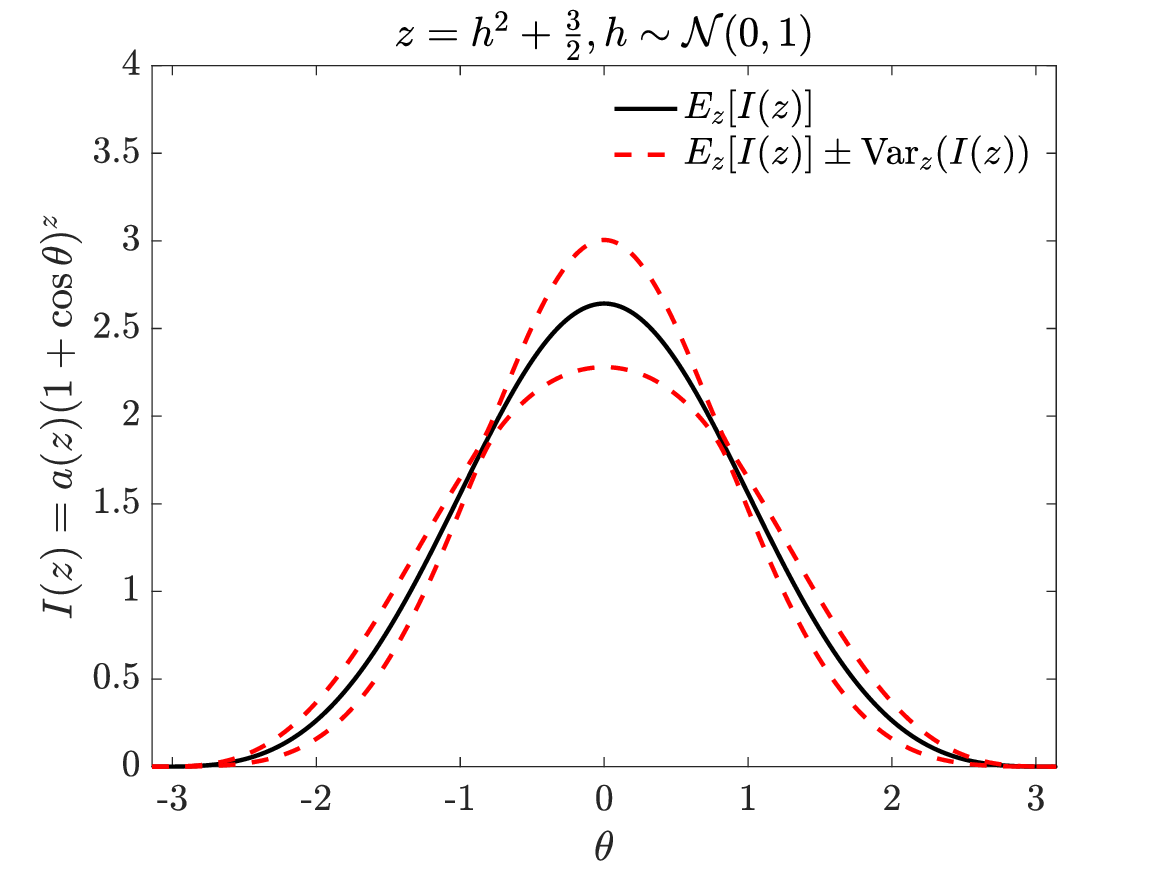}
	}
	\caption{Concentration of $I(z)$ for uniform/Gaussian random variables}
	\label{Fig_rv}
\end{figure}

\indent The rest of this paper is organized as follows. In Section \ref{sec:2}, we first discuss how the Winfree model \eqref{A-6} with high-order couplings can cast as the generalized  Adler equation, and review the previous results on the emergent behaviors in \cite{KHY} and we present the existence of a bounded trapping set for the Winfree model. In Section \ref{sec:3}, we provide sufficient frameworks for the emergence of death in the random Winfree model  \eqref{A-6} and perform a local sensitivity analysis of \eqref{A-6} with respect to uncertainty $z$. In Section \ref{sec:4}, we study the existence of the measure-valued stationary death state to the random kinetic Winfree model \eqref{A-7}. In Section \ref{sec:5}, we study the local sensitivity of the kinetic model \eqref{A-7} with respect to uncertainty $z$. In Section \ref{sec:6}, we provide several numerical simulations for \eqref{A-7} and compare them with analytical results. Finally, Section \ref{sec:7} is devoted to a brief summary of our main results and discussion of some remaining issues for a future work. \newline

\noindent{\bf Notation}:~Throughout the paper, we use the following handy notation:
\begin{gather*}
\Theta = (\theta_1, \cdots, \theta_N), \quad \Theta^{\mathrm{in}} = (\theta^{\mathrm{in}}_1, \cdots, \theta^{\mathrm{in}}_N), \quad \mathcal V = (\nu_1, \cdots, \nu_N), \\
{\mathcal D}(\Theta) := \max_{i,j\in[N]} |\theta_i - \theta_j|, \quad   {\mathcal D}({\mathcal V}) := \max_{i,j\in[N]} |\nu_i - \nu_j|, \quad  \|\mathcal V\|_\infty := \max_{i\in[N]}|\nu_i|.
\end{gather*}
\section{Preliminaries} \label{sec:2}
\setcounter{equation}{0}
In this section, we show how the Winfree model can cast as a generalized Adler's equation and review the previous results on the deterministic Winfree model with high-order couplings, and then we study several basic a priori estimates on the random Winfree model.
\subsection{The generalized Adler equation} \label{sec:2.1}
In this subsection, we study how the Winfree model can cast as a generalized Adler's equation for the special pair of sensitivity and influence functions. First, we recall the Adler equation \cite{Ad}:
\begin{align} \label{B-1}
	\dot\theta = \nu-\kappa\sin\theta, \quad \theta \in \bbr,
\end{align}
where $\nu$ and $\kappa$ denote a natural frequency and a nonnegative coupling strength, respectively.  Since \eqref{B-1} is separable first-order ODE, we can find an explicit solution in \cite{CHJK} for \eqref{B-1}, but it will not be useful in our discussion. Next, we show that the Adler equation \eqref{B-1} can arise from the Kuramoto model with $N=2$. More precisely, consider the Kuramoto model for a two-oscillator system:
\begin{align} \label{B-2}
\begin{dcases}
\displaystyle {\dot \theta}_1 = \nu_1 + \frac{\kappa}{2} \sin(\theta_2 - \theta_1), \\
\displaystyle  {\dot \theta}_2 = \nu_2 + \frac{\kappa}{2} \sin(\theta_1 - \theta_2).
\end{dcases}
\end{align}
We set the relative states:
\begin{align*}
\theta := \theta_1 - \theta_2 \quad \mbox{and} \quad  \nu := \nu_1 - \nu_2.
\end{align*}
Then, it is easy to see that the system \eqref{B-2} reduces to \eqref{B-1} for $(\theta, \nu)$. A generalized Adler equation can be obtained by replacing the constant coupling strength $\kappa$ in \eqref{B-1} with the state-dependent coupling strength $\tilde\kappa = \tilde\kappa(\theta)$:
\begin{align} \label{B-4}
	\dot\theta = \nu- \tilde\kappa(\theta) \sin\theta, \quad \theta \in \bbr,
\end{align}
From now on, we call \eqref{B-4} as a generalized Adler's equation. Now, we return to the Winfree model \eqref{A-1} with the special sensitivity function \eqref{A-4}. Then, it can be rewritten as
\begin{align}\label{B-5}
	\dot\theta_i = \nu_i - \left( \frac{\kappa}{N}\sum_{k=1}^N I(\theta_k) \right) \sin \theta_i,\quad \forall~i\in[N].
\end{align}
Note that system \eqref{B-5} corresponds to the generalized Adler equation with the state-dependent coupling strength:
\[ \tilde\kappa(\Theta) := \frac{\kappa}{N}\sum_{k=1}^N I(\theta_k). \]
In this manner, the Winfree model can cast as a generalized Adler's equation.

\subsection{Previous results} \label{sec:2.2}
In what follows, we briefly discuss previous results in \cite{KHY} on the asymptotic patterns of the Winfree model \eqref{A-5} with deterministic high-order couplings. First, we introduce a rotation number $\rho_i$ of the $i$-th oscillator:
\begin{align} \label{B-5-1}
\rho_i := \lim_{t\to\infty} \frac{\theta_i(t)}{t},
\end{align}
if the right-hand side exists. Before we present the previous results, we first define several asymptotic patterns of phase ensemble in terms of  rotation numbers.
\begin{definition} \label{D2.1}
\emph{\cite{AS, KHY}} Let $\Theta = \Theta(t)$ be a phase configuration whose dynamics is governed by  \eqref{A-5}.  Then, asymptotic patterns can be defined as follows. \newline
\begin{enumerate}
		\item The phase ensemble $\Theta(t)$ exhibits (oscillator) death if and only if all the rotation numbers are zero:
		\[
			\rho_i \equiv 0, \quad \forall~i\in[N].
		\]
		\item The phase ensemble $\Theta(t)$ exhibits (phase) locking if and only if all the rotation numbers are equal to a nonzero constant:
		\[
			\rho_i \equiv \rho \ne 0, \quad \forall~i\in[N].
		\]
		\item The phase ensemble $\Theta(t)$ exhibits (complete) incoherence if and only if all the rotation numbers are different:
		\[
			\rho_i \ne \rho_j, \quad \forall~i \neq j\in[N].
		\]
	\end{enumerate}
\end{definition}
\vspace{0.5cm}
In the next theorem, we recall sufficient conditions for the emergence of incoherence, death, and locking for the Winfree model \eqref{A-5}.  For $n \geq 1$, let $\beta_n\in(0, \pi)$ be a minimum point of the coupling function $SI_n$ in $(0, \pi)$, and let $a_n$ be  positive constants defined by the following relation
\[ \beta_n := \cos^{-1}\bigg( \frac{n}{n+1} \bigg) \quad \mbox{and} \quad  a_n = \frac{(2n)!!}{2^n(2n-1)!!}. \]
\begin{theorem} \label{T2.1}
	\cite{KHY}
	Let $\Theta = \Theta(t)$ be a global solution of \eqref{A-5} with initial data $\Theta^{in}$. Then, the following assertions on asymptotic patterns hold.
	\begin{enumerate}
		\item If system parameters satisfy
		\begin{align*}
		\min_{i\ne j} |\nu_i-\nu_j|  > 0, \quad 0 \leq \kappa < \frac{\min_{i\ne j} |\nu_i-\nu_j|}{2^{n+1}a_n} \sim \mathcal O \bigg( \frac{1}{\sqrt{n}} \bigg),
		\end{align*}
		then $\Theta$ exhibits incoherence. \vspace{.2cm}
		\item
		If system parameters and initial data satisfy
		\[ \hspace{.4cm} \beta_n < \alpha_n < \pi, \quad  \theta_i^{in} \in (-\alpha_n, \alpha_n) \quad \mbox{for all}~i\in[N], \quad  \kappa > \frac{\|\mathcal V\|_\infty}{a_n\sin\alpha_n(1+\cos\alpha_n)^n} \sim \mathcal O(1),
		\]
		then $\Theta$ exhibits death. \vspace{.2cm}
		\item If system parameters and initial data satisfy
		\begin{align*}
		\begin{aligned}
			& \nu_i = \nu \quad \mbox{for all}~i\in[N],  \quad 0 < \alpha_n < \bigg( \frac{\pi}{2^{n+1}a_n} \frac{n}{n+1} -\frac{1}{2^n} \bigg) \frac{1}{\sqrt{2n-1}} \bigg( \frac{2n}{2n-1} \bigg)^{n-1} \sim \mathcal O \bigg( \frac1n \bigg), \\
			& 0 < \kappa < \frac{\nu}{2^{n+1}a_n}, \quad  {\mathcal D}(\Theta^{in}) <  \alpha_n \exp \bigg[ - \frac{2^n a_n \kappa}{\nu - 2^n a_n \kappa} \left( \alpha_n \sqrt{2n-1} \bigg( \frac{2n-1}{2n} \bigg)^{n-1} + \frac{1}{2^{n-1}} \right) \bigg],
		\end{aligned}
		\end{align*}
		then $\Theta$ exhibits locking.
	\end{enumerate}
\end{theorem}

\subsection{The Winfree model with uncertain couplings} \label{sec:2.3}
In this subsection, we study the existence of a positively invariant subset to the Cauchy problem \eqref{A-6}. Note that once we find a bounded positive invariant subset, then all rotation numbers will be zero. Hence, the existence of a bounded positively invariant subset implies the emergence of oscillator death. Since the system \eqref{A-6} is random, we reformulate the previous results with respect to $[1, \infty)$-valued random variable $z$ using the jargon of probability theory. First, we define a square box ${\mathcal B}_c(0)$ in $\mathbb{R}^N$ around the origin:
\begin{align*}
{\mathcal B}_c(0):=\bigg\{\Theta=(\theta_1,\cdots,\theta_N)\in\mathbb{R}^N:~|\theta_i|<c,\quad\forall~i\in[N]\bigg\}.
\end{align*}
Let $z: \Omega \to [1, \infty)$ be a random variable on the sample space $\Omega$, and define the \textit{maximum point function} $\beta:[1,\infty)\to(0,\pi/2)$ by
\[\beta(z) := \underset{0<\theta<\pi/2}{\arg\max} ~ \Big\{  \sin\theta(1+\cos\theta)^z \Big \} =\arccos\left(\frac{z}{z+1}\right).\]
Let $c \in(0,\pi)$ be a positive constant which is called \textit{initial bound} of configuration, that is,
\[{\color{black}\Theta^{\mathrm{in}}(z) \in {\mathcal B}_c(0),\quad\mbox{almost surely}.} \]
Now, we define the \textit{adjoint bound function} $c^*(z)$ such that
\begin{align} \label{B-6-1}
0<c^*(z)\le\beta(z)<\frac{\pi}{2} \quad \mbox{and} \quad \sin c^*(z)(1+\cos c^*(z))^z=\sin c(1+\cos c)^z.
\end{align}
See Figure \ref{adjoint} for a graphical description of $c^*(z)$.
\begin{figure}[h]
\centering
\includegraphics[width=0.7\textwidth]{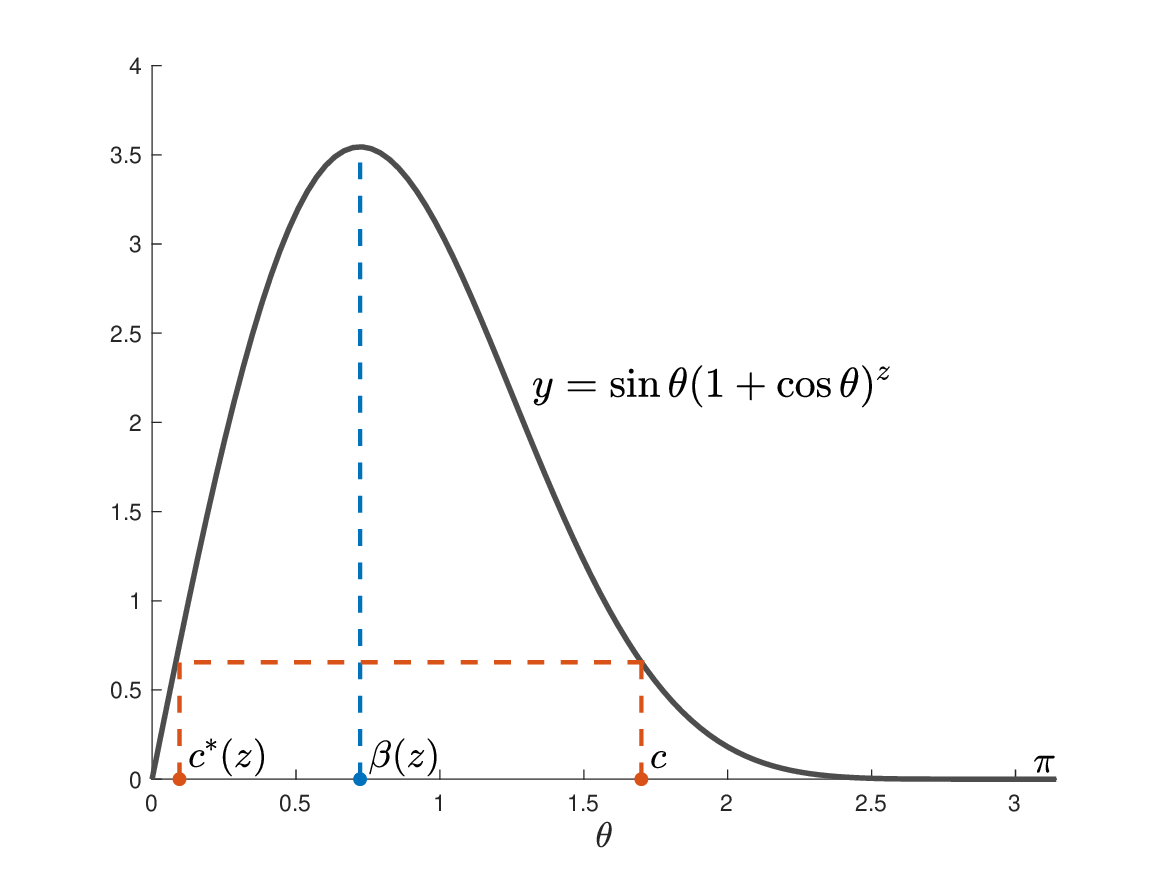}
\caption{Schematic diagram of $\beta(z)$ and $c^*(z)$ with $z=3$.}
\label{adjoint}
\end{figure}

Below, we state an assumption to be crucially used throughout this paper.
\begin{tcolorbox}
\begin{definition} \label{D2.2}
\emph{($\kappa$-assumption)}
	Let $c\in(0,\pi), \kappa$ and $z$ be the deterministic initial bound, coupling strength and $[1, \infty)$-valued random variable, respectively. Then we call the following relation as ``$\kappa$-{\it assumption}":
	\begin{align}\label{kA_ineq}
		\kappa>\frac{\|\mathcal{V}\|_\infty}{a(z)\sin c(1+\cos c)^z}, \qquad \mbox{almost surely}{\color{black},}
	\end{align}
	{\color{black}where $a(z)$ is defined in \eqref{def_az} and $\|\mathcal V\|_\infty:=\max_{i\in[N]}|\nu_i|$.}
\end{definition}
\end{tcolorbox}
\begin{remark}
	Since the denominator in the lower bound of $\kappa$ in $\kappa$-assumption goes to zero as $z$ becomes larger, i.e.,
	\[\lim_{z\to\infty}a(z)(1+\cos c)^z=0,\quad\forall~c\in(0,\pi),\]
	the $\kappa$-assumption implies $z\in L^\infty(\Omega)$.
\end{remark}
Next, we consider the uniform boundedness of $\Theta$ which is the random counterpart of the result in \cite{KHY}.
\begin{proposition}\label{P2.1}
\emph{(Existence of a trapping set)}
	Suppose that $\kappa$-assumption in Definition \ref{D2.2} holds and let $\Theta(t,z)$ be a global solution process to \eqref{A-6} {\color{black}with initial data $\Theta^{in}(z)\in\mathcal B_c(0)$, almost surely}.  Then, the set ${\mathcal B}_c(0)$ is positively invariant:
	\begin{align} \label{B-7}
	 \Theta(t,z)\in  {\mathcal B}_c(0),\quad\forall~t > 0,\quad\mbox{almost surely.}
	\end{align}
\end{proposition}
\begin{proof} First, we claim that
\begin{align} \label{B-9}
|\theta_i(t,z(\omega))| < c, \quad \forall~ t > 0,~~ i \in [N],
\end{align}
for a sample $\omega \in \Omega$ satisfying \eqref{kA_ineq}. \newline

\noindent {\it Proof of \eqref{B-9}}: ~We use the continuous induction argument. For a fixed $\omega$ satisfying \eqref{kA_ineq}, we set
	\[
	\mathcal{T}(z(\omega)):=\{t>0:\Theta(s, z(\omega))\in {\mathcal B}_c(0) \mbox{ for all }s\in[0,t)\}, \quad \tau^{\infty}(z(\omega)) :=\sup\mathcal{T}(z(\omega)).
	\]
Note that ${\mathcal T}$ and $\tau^{\infty}$ are random set and random time, and as long as it is clear in context, we suppress $\omega$ dependence in ${\mathcal T}$ and $\tau^{\infty}$ in what follows:
\[
\mathcal{T} := \mathcal{T}(z(\omega)) \quad \mbox{and} \quad  \tau^{\infty} := \tau^{\infty}(z(\omega)).
\]
Since $\Theta(\cdot, z)$ is continuous in time and $\Theta^{in} \in {\mathcal B}_c(0)$, we have
\[ \mathcal{T} \neq \emptyset  \quad \mbox{and} \quad \tau^{\infty} > 0. \]
  We claim that
	\begin{align} \label{B-10}
	 \tau^{\infty} = \infty.
	 \end{align}
	Suppose that the contrary holds, i.e.,
	\begin{align*}
	 \tau^{\infty} <\infty.
	  \end{align*}
	 Then, there exists a random index $i^* \in[N]$ such that
	 \begin{align*}
	  |\theta_{i^*}(\tau^{\infty}, z)|=c,
	   \end{align*}
	  and
	  \begin{align} \label{B-13}
	   \partial_t \theta_{i^*}(t,z)=\nu_{i^*}-\frac{a(z)\kappa}{N}\sin\theta_{i^*}(t,z)\sum_{j=1}^N\big(1+\cos\theta_j(t,z)\big)^z.
	  \end{align}
	   We multiply \eqref{B-13} by $\sgn(\theta_{i^*}(t,z))$ and use \eqref{kA_ineq}  to find
	\begin{align*}
		\partial_t|\theta_{i^*}(t, z)|\Big|_{t = \tau^{\infty}} &=~\nu_{i^*}\sgn (\theta_{i^*}(\tau^{\infty}, z)) -\frac{a(z)\kappa}{N}\sin c \sum_{j=1}^N\Big(1+\cos\theta_j(\tau^{\infty}, z)\Big)^{z}\\
		&\le \|\mathcal V\|_\infty-a(z)\kappa\sin c(1+\cos c)^{z}<0.
	\end{align*}
	This yields a contradiction to the definition of $\tau^{\infty}$. Therefore, we have \eqref{B-10} which implies \eqref{B-9}. Hence, we have the desired assertion \eqref{B-7}.
\end{proof}
As a direct corollary of Proposition \ref{P2.1}, we can show that a global solution will be confined in the smaller region ${\mathcal B}_{c^*}(0)$ with $c^*(z) \leq c$ after a certain (uniform) time, almost surely.
\begin{corollary}\label{C2.1}
	Suppose that $\kappa$-assumption holds and let $\Theta(t,z)$ be a global solution process to \eqref{A-6} {\color{black}with initial data $\Theta^{in}(z)\in\mathcal B_c(0)$, almost surely}. Then, there exists a constant $\tau_e  <\infty$, called an {\it entrance time}, such that
	\begin{align} \label{B-14}
	\Theta(t,z)\in {\mathcal B}_{c^*(z)}(0), \quad \forall~t \ge \tau_e, \quad\mbox{almost surely}.
	\end{align}
\end{corollary}
\begin{proof}  We choose a sample point $\omega \in \Omega$ satisfying the condition \eqref{kA_ineq}.  By Proposition \ref{P2.1}, one has
\[  \Theta(t,z)\in {\mathcal B}_c(0) \quad \mbox{for all $t\ge0$}. \]
For given $c$ and the adjoint bound function $c^*(z)$ satisfying \eqref{B-6-1} (see Figure \ref{adjoint}), we consider two cases. \newline

\noindent $\bullet$~Case A:~ Suppose that
\[ c^*(z)=c. \]
Then, we can choose the entrance time $t_e(z) = 0$ to find the desired estimate \eqref{B-14}.  \newline

\noindent $\bullet$~Case B:~Suppose that
\[  c^*(z)<c. \]
Then, we can define the time-dependent index $M = M(t, z)\in[N]$ such that
	\begin{align*}
		|\theta_M(t, z)| = \max_{i\in[N]} |\theta_i(t, z)|.
	\end{align*}
For an instant $t$ such that
\begin{align*}
 |\theta_M(t,z)|\in(c^*(z), c),
 \end{align*}
we have
	\begin{align}\label{L3.2P1}
		\begin{aligned}
			\partial_t|\theta_M(t,z)| &= \nu_i~\sgn(\theta_M(t,z))-\frac{a(z)\kappa}{N}\sin|\theta_M(t,z)|\sum_{j=1}^N(1+\cos\theta_j(t,z))^{z}\\
			&\le \|\mathcal V\|_\infty-a(z)\kappa \sin|\theta_M(t, z)| (1+\cos\theta_M(t, z))^{z} \\
			&\le \|\mathcal V\|_\infty-a(z)\kappa \sin c (1+\cos c)^{z} < 0,
		\end{aligned}
	\end{align}
	where the definition of $c^*(z)$ is used in the last second inequality (see Figure \ref{adjoint}). Since the last nonzero term is constant in time $t$ for a given $z$, this shows the existence of $t_e(z)<\infty$. \newline

\noindent Now, it suffices to show that there exists a finite uniform bound $\tau_e$ for $t_e(z)$:
	\[  t_e(z)\le \tau_e, \quad \mbox{almost surely}. \]
For $t_e(z)<\infty$, it is obvious that
	\begin{align*}
		|\theta_{M(t_e(z), z)}(t_e(z), z)|=c^*(z)\quad\mbox{and}\quad|\theta_{M(0,z)}(0,z)|<c.
	\end{align*}
	This implies
	\begin{align*}
		c-c^*(z)&>|\theta_{M(0,z)}(0,z)|-|\theta_{M(t_e(z),z)}(t_e(z),z)|\\
		&=-\int_0^{t_e(z)}\partial_t|\theta_{M(t,z)}(t,z)|dt \\
		&\ge t_e(z)\big(-\|\mathcal{V}\|_\infty+a(z)\kappa\sin c(1+\cos c)^{z}\big),
	\end{align*}
	where the last inequality comes from the estimate \eqref{L3.2P1}. Note that the last term in this relation is positive, hence, we have
	\[t_e(z)\le\frac{c-c^*(z)}{a(z)\kappa\sin c(1+\cos c)^{z}-\|\mathcal{V}\|_\infty}.\]
Since $z\in L^\infty(\Omega)$, the right-hand side is bounded almost surely. Therefore, the existence of a constant $\tau_e$ is verified.
\end{proof}
In the following section, we discuss the extension of the second assertion in Theorem \ref{T2.1} to the random Winfree model with high-order couplings \eqref{A-6}.

\section{The Winfree model with uncertain high-order couplings} \label{sec:3}
\setcounter{equation}{0}
In this section, we provide fast relaxation toward death and local sensitivity analysis on the propagation of regularity in $z$-variable.

\subsection{Emergence of oscillator death} \label{sec:3.1}
In this subsection, we deal with oscillator death under the $\kappa$-assumption \eqref{kA_ineq}. If there exists an equilibrium in a trapped region, we can consider the convergence toward an equilibrium.
First, we consider a condition which guarantees the existence of an equilibrium.
\subsubsection{Existence of an equilibrium} \label{sec:3.1.1}
Suppose that $\Phi(z)=(\phi_1(z),\cdots,\phi_N(z)) \in (-\pi, \pi)^N$ is a random equilibrium of \eqref{A-6}:
\begin{align}\label{equil}
	0=\nu_i+\frac{\kappa}{N}S(\phi_i(z))\sum_{j=1}^NI_z(\phi_j(z)),
\end{align}
where
\begin{align} \label{C-1}
S(\theta)=-\sin\theta\quad\mbox{and}\quad I_z(\theta)= a(z) (1+\cos\theta)^z.
\end{align}
Now, we consider two types of ensemble:
\begin{align*}
\begin{aligned}
& \nu_i  \equiv 0,~\forall~i \in [N]:~\mbox{Homogeneous ensemble}, \\
& \exists~i \neq j \in [N]~\mbox{such that}~\nu_i \neq \nu_j:~\mbox{Heterogeneous ensemble}.
\end{aligned}
\end{align*}

\vspace{0.5cm}

\noindent $\bullet$ Case A (Homogeneous ensemble):~If all natural frequencies are zero:
\[ \nu_i=0, \quad  \forall~i \in [N], \]
then system \eqref{equil} - \eqref{C-1} becomes
\[ 0 = \frac{a(z)\kappa}{N} \sin(\phi_i(z))\sum_{j=1}^N (1 + \cos \phi_j(z))^z, \quad \forall~i \in [N]. \]
Since $\phi_i(z) \in (-\pi, \pi)$ for all $i \in [N]$,  we have
\[\Phi(z)=\left(0,\cdots,0\right).\]

\vspace{0.2cm}

\noindent $\bullet$ Case B (Heterogeneous ensemble):~If there exists a nonzero natural frequency $\nu_k\ne 0$, we divide \eqref{equil} by $\nu_k$  to find
\begin{align} \label{C-2}
		1+\frac{\kappa}{N}\frac{S(\phi_k(z))}{\nu_k}\sum_{j=1}^NI_z(\phi_j(z))=0.
\end{align}
Motivated by this relation, we define a function $F_z:=\left[-\frac{1}{\|\mathcal V\|_{\infty}},\frac{1}{\|\mathcal V\|_{\infty}}\right]\to\mathbb{R}$ by
\begin{align} \label{C-2-0}
F_z(s):=1+\frac{\kappa}{N}s\sum_{j=1}^NI_z(S^{-1}(\nu_js)).
\end{align}
Once we have a solution $s(z)$ to $F_z= 0$, i.e.,
\begin{align*}
F_z(s(z)) = 0,
\end{align*}
then, the state $\Phi(z)=\left(\phi_1(z),\cdots,\phi_N(z)\right)$ defined by
\begin{align*}
	\phi_j(z):= S^{-1}(\nu_js(z)),\quad\forall~j \in [N],
\end{align*}
satisfies \eqref{C-2}, i.e., it is an equilibrium. In addition, when $\nu_i=0$ for all $i\in[N]$, above definition covers Case A.
\begin{proposition}\label{P3.1}
\emph{(Existence of an equilibrium)}
Suppose that the $\kappa$-assumption holds. Then, there exists an equilibrium $\Phi(z) = (\phi_1(z), \ldots, \phi_N(z))$ of \eqref{A-6} in ${\mathcal B}_{c^*}(0)$, almost surely.
\end{proposition}
\begin{proof}
First, if $\|\mathcal V\|_\infty=0$, it is obvious with $\Phi(z)=(0,\cdots,0)$, hence, we assume $\|\mathcal V\|_{\infty}\ne0$. We choose $\omega \in \Omega$ satisfying the condition \eqref{kA_ineq}, i.e.,
\begin{align} \label{C-2-2}
\kappa> -\frac{\|\mathcal{V}\|_\infty}{S(c) I_z(c)} =   -\frac{\|\mathcal{V}\|_\infty}{S( c^*(z(\omega))) I_z(c^*(z(\omega))},
 \end{align}
where we use \eqref{B-6-1}. We split the proof into two steps. \newline

\noindent $\bullet$~Step {\color{black}I}: First, we claim that for each $\omega \in \Omega$ satisfying \eqref{C-2-2}, we have
	\begin{align} \label{C-3}
	\exists~s(z(\omega))\in\left[\frac{S(c^*(z(\omega)))}{\|\mathcal V\|_\infty},0\right]\mbox{ such that }F_{z(\omega)}\big(s(z(\omega))\big)=0.
	\end{align}
Again, we suppress $\omega$-dependence in $z$, i.e., $z = z(\omega)$. \newline

\noindent {\it Proof of \eqref{C-3}}: Since $F_{z}$ is continuous on $\left[\frac{S(c^*(z))}{\|\mathcal V\|_\infty},~0\right]$ and $F_{z}(0)=1>0$, it suffices to show that
\begin{align} \label{C-3-1}
F_{z}\left(\frac{S(c^*(z))}{\|\mathcal V\|_\infty}\right)<0.
 \end{align}
 by the intermediate value theorem.

\vspace{0.2cm}

\noindent	$\diamond$ Case A: Suppose that there exists $i \in [N]$ such that $\nu_i\ge0$. Since
	\[-1<S(c^*(z))\le\frac{\nu_i}{\|\mathcal V\|_\infty}S(c^*(z))\le0,\]
	we have
	\[\frac{\pi}{2}>c^*(z)\ge S^{-1}\left(\frac{\nu_i}{\|\mathcal V\|_\infty}S(c^*(z))\right)\ge0,\]
	and successively,
	\begin{align} \label{C-3-2}
	I_{z}(c^*(z))\le I_{z}\left(S^{-1}\left(\frac{\nu_i}{\|\mathcal V\|_\infty}S(c^*(z))\right)\right).
	\end{align}
Then, we use \eqref{C-2-0},  \eqref{C-2-2}, \eqref{C-3-2} and
\[ \frac{S(c^*(z))}{\|\mathcal V\|_\infty} < 0 \]
to see
\begin{align*}
		F_{z}\left(\frac{S(c^*(z))}{\|\mathcal V\|_\infty}\right)&=1+\frac{\kappa}{N}\frac{S(c^*(z))}{\|\mathcal V\|_\infty}\sum_{j=1}^NI_{z}\left(S^{-1}\left(\nu_j\frac{S(c^*(z))}{\|\mathcal V\|_\infty}\right)\right)\\
		&\le1+ \kappa\frac{S(c^*(z))}{\|\mathcal V\|_\infty}I_{z}(c^*(z))<0,
	\end{align*}
which verifies \eqref{C-3-1}.

\vspace{0.2cm}

\noindent	$\diamond$ Case B:~Suppose that there exists $i \in [N]$ such that $\nu_i<0$. By the same argument as in Case A, we obtain
	\[I_z(-c^*(z))=I_{z}(c^*(z))\le I_{z}\left(S^{-1}\left(\frac{\nu_i}{\|\mathcal V\|_\infty}S(c^*(z))\right)\right).\]
This yields the desired estimate \eqref{C-3-1}.

	\vspace{0.2cm}

\noindent $\bullet$~Step {\color{black}II}:~By \eqref{C-3}, there exists a root $s(z)$ of $F_{z} = 0$ on $\left[\frac{S(c^*(z))}{\|\mathcal V\|_\infty},0\right]$ for a sample point $\omega$ satisfying \eqref{kA_ineq}. Hence, we can construct an equilibrium point of \eqref{A-6}.

	\vspace{.2cm}

\noindent $\bullet$~Step {\color{black}III}:~For a root $s(z)$ of $F_z$ obtained in Step {\color{black}II}, we need to show that
\[\Phi(z)=\Big(S^{-1}(\nu_1s(z)),\cdots,S^{-1}(\nu_Ns(z))\Big)\in\mathcal B_{c^*(z)}(0).\]
For $\nu_i\ge0$, we have
\[S(c^*(z))\le\frac{S(c^*(z))}{\|\mathcal V\|_\infty}\nu_i\le s(z)\nu_i\le0.\]
This leads to
\[c^*(z)\ge S^{-1}(\nu_is(z))\ge0.\]
Similarly, we can show for the case $\nu_i<0$.
\end{proof}

\subsubsection{Relaxation dynamics} \label{sec:3.1.2}
Let $\Theta$ and $\Phi(z)=(\phi_1(z),\cdots,\phi_N(z)) \in {\mathcal B}_{c^*}(0)$ be a solution processes  to \eqref{A-6} and an equilibrium whose existence is guaranteed by Proposition \ref{P3.1}, i.e., it satisfies
\begin{align}\label{C-3-3}
0=\nu_i-\frac{a(z)\kappa}{N}\sin\phi_i(z)\sum_{j=1}^N\big(1+\cos\phi_j(z)\big)^z.
\end{align}
It follows from Corollary \ref{C2.1} that
\[
\Theta(t,z)\in {\mathcal B}_{c^*(z)}(0), \quad \forall~ t > \tau_e,
\]
where $\tau_e$ has the same meaning in Corollary \ref{C2.1}. Now, we introduce a deviation from the equilibrium $\Phi(z)$:
\[ \hat{\theta}_i(t, z):=\theta_i(t, z)-\phi_i(z), \quad  \forall~i \in [N]. \]
By \eqref{C-3-3}, the deviation $\hat{\theta}_i$ satisfies
	\begin{align}
	\begin{aligned} \label{C-3-4}
		\partial_t\hat{\theta}_i &= \partial_t  (\theta_i - \phi_i)  \\
		&=~\nu_i-\frac{a(z)\kappa}{N}\sin\theta_i\sum_{j=1}^N(1+\cos\theta_j)^{z}   \\
		&=~\frac{a(z)\kappa}{N} \left[  \sin\phi_i \sum_{j=1}^N\big(1+\cos\phi_j \big)^z - \sin\theta_i\sum_{j=1}^N(1+\cos\theta_j)^{z} \right ] \\
		& =~-\frac{a(z)\kappa}{N} \left[  ( \sin \theta_i - \sin\phi_i ) \sum_{j=1}^N\big(1+\cos\phi_j \big)^z \right. \\
		& \hspace{2.5cm} + \left. \sin \theta_i \sum_{j=1}^N \Big (  (1+\cos\theta_j)^z  - \big(1+\cos\phi_j \big)^z  \Big) \right ] \\
		&=-\frac{a(z)\kappa}{N}\cos {\tilde \theta}_{i,1}\cdot\hat{\theta}_i\sum_{j=1}^N\left(1+\cos \phi_j\right)^{z} \\
		&\hspace{.4cm}+\frac{za(z)\kappa}{N}\sin\theta_i \sum_{j=1}^N \left(1+\cos{\tilde \theta}_{j,2}\right)^{{z}-1}\sin {\tilde \theta}_{j,2}\cdot \hat{\theta}_j,
	\end{aligned}
	\end{align}
	where $\tilde{\theta}_{i,1}(z,t)$ and $\tilde{\theta}_{j,2}(z,t)$ are intermediate values determined by the mean value theorem. Using this relation, we will show the convergence toward an equilibrium in the following theorem.
\begin{theorem} \label{T3.1}
	Suppose the $\kappa$-assumption holds, and let $\Theta(t,z)$ be a global solution process to \eqref{A-6} {\color{black}with initial data $\Theta^{in}(z)\in\mathcal B_c(0)$, almost surely}. Then, $\Theta(t,z)$ converges to an equilibrium $\Phi(z) \in {\mathcal B}_{c^*(z)}(0)$ exponentially fast, almost surely.
\end{theorem}
\begin{proof}
We split its proof into two steps. \newline

\noindent $\bullet$~Step {\color{black}I} (Derivation of differential inequality):~By Corollary \ref{C2.1} and Proposition \ref{P3.1}, there exists an equilibrium $\Phi(z)$ in the squared box ${\mathcal B}_{c^*(z)}(0)$ and the solution process $\Theta$ is captured in the same box ${\mathcal B}_{c^*(z)}(0)$ after a certain time $\tau_e>0$ almost surely:
	\begin{align} \label{C-3-4-1}
	\Phi(z),~~\Theta(t, z)\in {\mathcal B}_{c^*(z)}(0) \subset {\mathcal B}_{\frac{\pi}{2}}(0),\quad\mbox{for all}~~t\ge \tau_e,\quad\mbox{almost surely.}
	\end{align}
	Now, we consider $t\ge\tau_e$. It follows from \eqref{C-3-4} that
\begin{align} \label{C-3-5}
\partial_t\hat{\theta}_i = -\frac{a(z)\kappa}{N}\cos {\tilde \theta}_{i,1}\cdot\hat{\theta}_i\sum_{j=1}^N\left(1+\cos \phi_j\right)^{z} +\frac{za(z)\kappa}{N}\sin\theta_i \sum_{j=1}^N \left(1+\cos{\tilde \theta}_{j,2}\right)^{{z}-1}\sin {\tilde \theta}_{j,2}\cdot \hat{\theta}_j.
\end{align}
We multiply $\sgn(\hat{\theta}_i)$ to \eqref{C-3-5} and find
	\begin{align}
	\begin{aligned} \label{C-3-6}
		\partial_t|\hat{\theta}_i| &=- \Big( \frac{a(z)\kappa}{N}\cos\tilde{\theta}_{i,1}\cdot  \sum_{j=1}^N \left(1+\cos \phi_j\right)^{z} \Big) |\hat{\theta}_i| \\
		&\hspace{.4cm}+\frac{z a(z)\kappa}{N}\sin\theta_i \cdot \sgn(\hat{\theta}_i)\sum_{j=1}^N\left(1+\cos\tilde{\theta}_{j,2}\right)^{z-1}\sin\tilde{\theta}_{j,2}\cdot\hat{\theta}_j.
	\end{aligned}
	\end{align}
Now, we sum up \eqref{C-3-6} over all $i \in [N]$ to get
	\begin{align}
	\begin{aligned}  \label{C-3-7}
		\partial_t \left(\sum_{i=1}^N|\hat{\theta}_i|\right) &=-\frac{a(z)\kappa}{N}\sum_{i,j=1}^N\cos\tilde{\theta}_{i,1}\left(1+\cos\phi_j\right)^{z}  |\hat{\theta}_i|  \\
		&\hspace{.4cm}+\frac{za(z)\kappa}{N}\sum_{i,j=1}^N \sin\theta_i \cdot \sgn(\hat{\theta}_i) \left(1+\cos {\tilde \theta}_{j,2} \right)^{z-1}\sin\tilde{\theta}_{j,2}\cdot\hat{\theta}_j \\
		&=: {\mathcal I}_{11} + {\mathcal I}_{12}.
	\end{aligned}
	\end{align}
Below, we estimate the term ${\mathcal I}_{1i}$ one by one. \newline

\noindent $\diamond$~Case A (Estimate of ${\mathcal I}_{11}$): We use \eqref{C-3-4-1} to see that
\[
\cos\tilde{\theta}_{i,1} \geq  \cos c^*(z)\quad\mbox{and} \quad  \left(1+\cos \phi_j\right)^{z} \geq (1 +  \cos c^*(z))^z.
\]
These imply
\begin{align} \label{C-3-8}
	\begin{aligned}
		{\mathcal I}_{11} &= -\frac{a(z)\kappa}{N}\sum_{i,j=1}^N\cos\tilde{\theta}_{i,1}\left(1+\cos\phi_j\right)^{z}  |\hat{\theta}_i|  \\
		&\leq - \kappa a(z) \cos c^*(z) (1 + \cos c^*(z))^z  \sum_{i=1}^{N}  |\hat{\theta}_i|.
	\end{aligned}
\end{align}

\vspace{0.2cm}

\noindent $\diamond$~Case B (Estimate of ${\mathcal I}_{12}$):~First, we use \eqref{C-3-4-1} to see
\begin{align} \label{C-3-9}
 | \sin\theta_i\cdot \sgn(\hat{\theta}_i)| \leq  \sin c^*(z).
 \end{align}
 On the other hand, we use the fact that
	\[f(x)=\sin x\left(1+\cos x\right)^{z-1}~~\mbox{is monotonically increasing on}~~\left[0,\arccos\left(\frac{z-1}{z}\right)\right], \]
\begin{align*}
0<c^*(z)\le\beta(z)<\frac{\pi}{2}\quad\mbox{and}\quad \beta(z)=\arccos\left(\frac{z}{z+1}\right)<\arccos\left(\frac{z-1}{z}\right)
 \end{align*}
to see that
\begin{align} \label{C-3-10}
\Big| \left(1+\cos\tilde{\theta}_{j,2}\right)^{z-1}\sin\tilde{\theta}_{j,2}  \Big| \leq (1 + \cos c^*(z))^{z-1} \sin c^*(z).
\end{align}
Now, we use \eqref{C-3-9} and \eqref{C-3-10} to see
\begin{align}
\begin{aligned} \label{C-3-11}
|{\mathcal I}_{12}| &\leq \frac{\kappa za(z)}{N}\sum_{i,j=1}^N  \Big| \sin\theta_i \cdot \sgn(\hat{\theta}_i) \left(1+\cos {\tilde \theta}_{j,2} \right)^{z-1}\sin\tilde{\theta}_{j,2}\cdot \hat{\theta}_j \Big|  \\
&\leq \frac{\kappa za(z)}{N}\sum_{i,j=1}^N   (1 + \cos c^*(z))^{z-1} \sin^2 c^*(z) |\hat{\theta}_j|  \\
&= \kappa za(z)  (1 + \cos c^*(z))^{z-1} \sin^2 c^*(z) \sum_{j=1}^{N} |\hat{\theta}_j|.
\end{aligned}
\end{align}
In \eqref{C-3-7}, we combine \eqref{C-3-8} and \eqref{C-3-11} to find
	\begin{align*}
	\begin{aligned} \label{C-3-12}
		\partial_t \sum_{i=1}^N|\hat{\theta}_i |  &\le- \kappa a(z) \left(\cos c^*(z)\left(1+\cos c^*(z)\right)^{z}\right)\sum_{i = 1}^N|\hat{\theta}_i| \\
		&\hspace{.4cm}+ \kappa z a(z) \left(1+\cos c^*(z)\right)^{z-1}\sin^2 c^*(z)\sum_{j=1}^N|\hat{\theta}_j|\\
		&=\kappa a(z)(1+\cos c^*(z))^{z}\Big(-\cos c^*(z)+ z(1-\cos c^*(z))\Big)\sum_{i=1}^N|\hat{\theta}_i| \\
		&=\kappa z a(z)(1+\cos c^*(z))^{z}\Big(1 - \frac{z+1}{z} \cos c^*(z) \Big)\sum_{i=1}^N|\hat{\theta}_i| \\
		&=:C(z)\sum_{i=1}^N|\hat{\theta}_i|.
	\end{aligned}
	\end{align*}

\vspace{0.5cm}

\noindent $\bullet$ Step {\color{black}II}:~From
\[c^*(z) < \beta(z) = \cos^{-1}\left(\frac{z}{z+1}\right),\]
we get
\[ \cos c^*(z) > \cos \beta(z) =   \cos \Big( \cos^{-1}\left(\frac{z}{z+1}\right) \Big) = \frac{z}{z+1}  \]
i.e.,
\begin{align} \label{C-3-12-1}
 \frac{z+1}{z} \cos c^*(z) > 1,
\end{align}
which makes $C(z)$ be negative and we obtain desired exponential decay of $\sum_{i=1}^N|\hat{\theta}_i |$.
\end{proof}

\subsection{Local sensitivity analysis} \label{sec:3.2}
In this subsection, we provide a local sensitivity analysis for the random Winfree model \eqref{A-6} with high-order couplings:
\begin{align*}
\partial_t \theta_i(t,z)=\nu_i-\frac{\kappa}{N} a(z) \sin\theta_i(t,z)\sum_{j=1}^N\big(1+\cos\theta_j(t,z)\big)^z.
 \end{align*}
To see the random effect, we expand the phase process $\theta_{i}(t,z + dz)$  via Taylor's expansion in $z$-variable:
\begin{align} \label{C-3-14}
\theta_i(t, z + dz) = \theta_i(t,z) +  \partial_z \theta_i(t,z)  dz + \frac{1}{2} \partial_z^2 \theta_i(t,z) (dz)^2 + \cdots.
\end{align}
Thus, the local sensitivity estimates deal with the dynamic behaviors of the sensitivity vectors $ \partial^k_z \Theta$ consisting of coefficients in the R.H.S. of \eqref{C-3-14}.  For this, we consider $L^\infty$-random variable $z:\Omega\to[1,\infty)$, which satisfies
\[z_\sharp\mathbb{P}\ll\mu,\]
where $\mathbb{P}$ is the probability measure on $(\Omega,\mathcal F,\mathbb{P})$ and $\mu$ is the Borel measure on $\mathbb{R}$. In other words, we do not take into account a random variable which is not appropriate in this context, for example, constant, Dirac-delta and so on. Hence, we can represent the push-forward measure $z_\sharp\mathbb{P}$ in terms of the Borel measure $\mu$ using the probability density function $\rho(s)$ as follows:
\begin{align*}
	\int_{E}g(s)dz_\sharp\mathbb{P}(s)=\int_{E}g(s)\rho(s)d\mu(s)\quad\mbox{for all}~~ E\in\mathcal{B}\big([1,\infty)\big).
\end{align*}
For notation simplicity, we will abbreviate $d\mu(s)$ as $ds$. Before we move on further, we suppress $t$-dependence in $\Theta(t,z)$:
\[ \Theta(z) \equiv \Theta(t,z), \]
and  set several norms. For $\Theta(z) = (\theta_1(z), \cdots, \theta_N(z))$, we denote the Euclidean norm $\ell^2$ of $\Theta$ by $|\cdot|$. Similarly, we use $\| \cdot \|$ to denote the weighted $L^2(\rho dz)$-norm in $[1, \infty)$:
\begin{align*}
\begin{aligned}
& |\partial^k_z\Theta(z)| :=\left(\sum_{i=1}^N|\partial^k_z\theta_i(z)|^2\right)^{1/2}, \quad k= 0,1,  \qquad \|\Theta \| :=\left(\int_{[1,\infty)} |\Theta(z)|^2 \rho(z)dz\right)^{1/2}, \\
& \|\partial_z\Theta\| :=\left(\int_{[1,\infty)} |\partial_z\Theta(z)|^2 \rho(z)dz\right)^{1/2}, \quad \|\Theta\|_{H_z^1} :=\Big(\|\Theta\|^2+\|\partial_z\Theta\|^2\Big)^{1/2}.
\end{aligned}
\end{align*}
Unlike the Cucker-Smale and Kuramoto models, which are governed by the distance between particles, the dynamics of Winfree model is influenced by the locations of each particle. Hence, it is hard to consider the local sensitivity in terms of flocking. Even if we analyze the relative distance $\theta_i-\theta_j$, since we have to exploit the logarithm to get the derivative of $(1+\cos\theta_j(t,z))^z$, $\Theta(t,z)$ should be in some bounded region contained in ${\mathcal B}_{\frac{\pi}{2}}(0)$. \newline

For a later use, we set
	\begin{align*}
		C_{1,1,1}(z) :=
		\begin{dcases}
			2^z\cos c,\quad\mbox{if}~~\cos c<0,\\
			\cos c(1+\cos c)^z,\quad\mbox{if}~~\cos c\ge0,
		\end{dcases}
	\end{align*}
and define random coefficients $C_i(t,z),~i=1,2$  as follows:
	\begin{align}
	\begin{aligned} \label{C-3-14-1}
	&	C_1(t,z):=
		\begin{dcases}
			\kappa a(z)\left(-C_{1,1,1}+\frac{z}{\sqrt{2z-1}}\left(\frac{2z-1}{z}\right)^z\right)=:C_{1,1}(z), \quad & t\in(0, \tau_e),\\
			\kappa a(z)z\left(1-\frac{z+1}{z}\cos c^*(z)\right)\big(1+\cos c^*(z)\big)^z=:C_{1,2}(z),\quad & t\in[\tau_e,\infty),
		\end{dcases} \\
	&	C_2(t,z):=
		\begin{dcases}
			\kappa|a'(z)|2^z\sqrt{N}+\kappa a(z)2^z\sqrt{N}\ln 2=:C_{2,1}(z), \qquad  & t \in(0,\tau_e),\\
			\kappa|a'(z)|2^z\sqrt{N}\sin c^*(z)+\kappa a(z)2^z\sqrt{N}\ln 2\cdot\sin c^*(z)=:C_{2,2}(z),\quad & t\in[\tau_e,\infty).
		\end{dcases}
	\end{aligned}
	\end{align}
In what follows, for notation simplicity, we suppress $z$-dependence in $\theta_i$ and $C_i$ as long as the context is clear:
\[ \theta_i(t) \equiv \theta_i(t,z). \]

\begin{lemma}\label{L3.1}
\emph{(Propagation of $\partial_z\Theta$)}
	Suppose that the $\kappa$-assumption holds, and let $\Theta = \Theta(t,z)$ be a global solution process to \eqref{A-6}. Then, $\|\partial_z\Theta(t) \|$ satisfies a differential inequality:
	\[
			\partial_t\|\partial_z\Theta(t) \| \le C_1(t{\color{black},z}) \|\partial_z\Theta(t) \| +C_2(t{\color{black},z}), \quad\forall~t > 0.
	\]
\end{lemma}
\begin{proof}
We differentiate \eqref{A-6} with respect to $z$ to get
	\begin{align}
	\begin{aligned} \label{C-3-15}
		\partial_t\partial_z\theta_i &=-\frac{\kappa}{N}\sum_{j=1}^N \Bigg(a'(z)\sin\theta_i+a(z)\cos\theta_i\cdot\partial_z\theta_i +a(z)\sin\theta_i\ln(1+\cos\theta_j) \\
		& \hspace{2cm}  -\frac{a(z)z\sin\theta_i\sin\theta_j\cdot\partial_z\theta_j}{1+\cos\theta_j}\Bigg) (1+\cos\theta_j)^z.
	\end{aligned}
	\end{align}
	We multiply \eqref{C-3-15} by $\partial_z\theta_i$ and sum up the resulting relation with respect to all $i \in [N]$ to find
	\begin{align}\label{C-3-16}
		\begin{aligned}
			& \partial_t\left(\frac{1}{2}\sum_{i=1}^N|\partial_z\theta_i|^2\right) \\
			& \hspace{1cm} =-\frac{\kappa}{N}\sum_{i,j = 1}^N a'(z)\partial_z\theta_i\cdot\sin\theta_i(1+\cos\theta_j)^z-\frac{\kappa}{N}\sum_{i,j=1}^{N} a(z)\cos\theta_i\cdot|\partial_z\theta_i|^2(1+\cos\theta_j)^z\\
			& \hspace{1.4cm}-\frac{\kappa}{N}\sum_{i,j =1}^Na(z)\sin\theta_i\cdot\partial_z\theta_i(1+\cos\theta_j)^z\ln(1+\cos\theta_j)\\
			& \hspace{1.4cm} +\frac{a(z)z\kappa}{N}\sum_{i,j=1}^{N} \sin\theta_i\sin\theta_j\cdot\partial_z\theta_i\cdot\partial_z\theta_j\cdot(1+\cos\theta_j)^{z-1} \\
			& \hspace{1cm} =:\mathcal{I}_{21} +\mathcal{I}_{22} +\mathcal{I}_{23} +\mathcal{I}_{24}.
		\end{aligned}
	\end{align}
In the sequel, we estimate the term ${\mathcal I}_{2i}$ one by one. \newline

 Recall that our goal is to derive an upper bound by the term containing $\|\partial_z\Theta\|$. Here, we will use some basic relations in trigonometric functions:  For $0<c<\pi$ with
	\[ 0<c^*(z)\le\beta(z)\le c<\pi, \]
	we have
	\begin{align*}
			\cos c<\cos x\quad\mbox{for}~~x\in(-c,c) \quad\mbox{and} \quad  0<\cos c^*(z)<\cos x\quad\mbox{for}~~x\in(-c^*(z),c^*(z)).
	\end{align*}

	\vspace{0.2cm}

\noindent $\bullet$ Case A (Estimate of $\mathcal{I}_{21}$): We use the Cauchy-Schwarz inequality and $0 \leq (1+\cos\theta_j)^z \leq 2^z$  to obtain
	\begin{align}
	\begin{aligned} \label{C-3-17}
		|\mathcal{I}_{21}| &= \Big| -\frac{\kappa}{N}\sum_{i,j =1}^{N} a'(z)\partial_z\theta_i\cdot\sin\theta_i(1+\cos\theta_j)^z \Big| \\
		&\le\frac{\kappa|a'(z)|}{N}\sum_{i =1}^{N} |\partial_z\theta_i\cdot\sin\theta_i|\sum_{j=1}^{N} (1+\cos\theta_j)^z\\
		&\le\kappa|a'(z)|2^z\left(\sum_{i =1}^{N} |\partial_z\theta_i|^2\right)^{1/2}\left(\sum_{i=1}^{N} \sin^2\theta_i\right)^{1/2}.
	\end{aligned}
	\end{align}
	Now we use \eqref{C-3-17} and Corollary \ref{C2.1} to get
	\begin{align}
	\begin{aligned} \label{C-3-18}
	 |\mathcal{I}_{21}| \le
		\begin{dcases}
			\kappa|a'(z)|2^z\sqrt{N}\left(\sum_{i=1}^{N} |\partial_z\theta_i|^2\right)^{1/2},\hspace{1.8cm}  & t\in(0,\tau_e),\\
			\kappa|a'(z)|2^z\sqrt{N}\sin c^*(z)\left(\sum_{i=1}^{N} |\partial_z\theta_i|^2\right)^{1/2},\quad ~& t\in[\tau_e,\infty).\\
		\end{dcases}
	\end{aligned}
	\end{align}

	\vspace{0.2cm}

\noindent	$\bullet$ Case B  (Estimate of $\mathcal{I}_{22}$): By direct calculation, we have
\begin{align}
\begin{aligned} \label{C-3-19}
{\mathcal I}_{22} &= -\frac{\kappa}{N}\sum_{i,j =1}^{N} a(z)\cos\theta_i\cdot|\partial_z\theta_i|^2(1+\cos\theta_j)^z \\
&<
		\begin{dcases}
			-\kappa a(z)2^z\cos c\sum_{i=1}^{N} |\partial_z\theta_i|^2,\hspace{3.2cm} ~t\in(0, \tau_e),\quad\mbox{if}~~\cos c<0,\\
			-\kappa a(z)\cos c(1+\cos c)^z\sum_{i=1}^{N} |\partial_z\theta_i|^2,\hspace{1.7cm} ~t\in(0, \tau_e),\quad\mbox{if}~~\cos c\ge0,\\
			-\kappa a(z)\cos c^*(z)(1+\cos c^*(z))^z\sum_{i=1}^{N} |\partial_z\theta_i|^2,\quad t\in[\tau_e,\infty).
		\end{dcases}
\end{aligned}
\end{align}

	\vspace{0.2cm}

\noindent	$\bullet$ Case C (Estimate of $\mathcal{I}_{23}$): ~For $t\in(0, \tau_e)$, one has
	\begin{align*}
	\begin{aligned}
		|\mathcal{I}_{23}| &= \Big| \frac{\kappa}{N}\sum_{i,j =1}^{N} a(z)\sin\theta_i\cdot\partial_z\theta_i\cdot(1+\cos\theta_j)^z\ln(1+\cos\theta_j) \Big| \\
		&\leq \frac{\kappa a(z)}{N}\sum_{i=1}^{N} |\sin\theta_i\cdot\partial_z\theta_i |\cdot \sum_{j=1}^{N} |(1+\cos\theta_j)^z\ln(1+\cos\theta_j) | \\
		&\le\kappa a(z)2^z \ln 2 \sum_{i=1}^{N} |\sin\theta_i\cdot\partial_z\theta_i|  \\
		&\le\kappa a(z)2^z \ln 2 \left(\sum_{i=1}^{N} |\partial_z\theta_i|^2\right)^{1/2}\left(\sum_{i=1}^{N} \sin^2\theta_i\right)^{1/2} \\
		&\le\kappa a(z)2^z\sqrt{N}\ln2\cdot\left(\sum_{i=1}^{N} |\partial_z\theta_i|^2\right)^{1/2},
	\end{aligned}
	\end{align*}
where we used the following relation:
\[  \sum_{j=1}^{N} |(1+\cos\theta_j)^z\ln(1+\cos\theta_j) | \leq N 2^z \ln 2. \]
	Similarly, we obtain that for $t\in[\tau_e,\infty)$,
	\begin{align} \label{C-3-21}
		|\mathcal{I}_{23}| \le\kappa a(z) 2^z\sqrt{N}\ln 2\cdot\sin c^*(z)\left(\sum_{i=1}^{N} |\partial_z\theta_i|^2\right)^{1/2}.
	\end{align}

	\vspace{0.2cm}

\noindent	$\bullet$  Case D (Estimate of $\mathcal{I}_{24}$):  We use the Cauchy-Schwarz inequality to find
	\begin{align}
	\begin{aligned} \label{C-3-22}
		|\mathcal{I}_{24}| &\leq  \frac{a(z)z\kappa}{N}\sum_{i,j =1}^{N} \Big| \sin\theta_i\sin\theta_j\cdot\partial_z\theta_i\cdot\partial_z\theta_j\cdot(1+\cos\theta_j)^{z-1} \Big| \\
		&\le\frac{a(z)z\kappa}{N}\left(\sum_{i=1}^{N} |\partial_z\theta_i|^2\right)^{1/2}\left(\sum_{i=1}^{N} \sin^2\theta_i\right)^{1/2}\left(\sum_{j=1}^{N}|\partial_z\theta_j|^2\right)^{1/2} \\
		&\hspace{.4cm}\times \left(\sum_{j=1}^{N} \sin^2\theta_j(1+\cos\theta_j)^{2(z-1)}\right)^{1/2}.
	\end{aligned}
	\end{align}

	Note that for an even function $f(x)=\sin^2x(1+\cos x)^{2(z-1)}$, it has a local maximum at $x=\arccos\left(\frac{z-1}{z}\right)$:
	\begin{align}
	\begin{aligned} \label{C-3-23}
	&	f'(x)
		\begin{dcases}
			>0\quad\mbox{for}~~ x\in\left(0,\arccos\left(\frac{z-1}{z}\right)\right),\\
			<0\quad\mbox{for}~~ x\in\left(\arccos\left(\frac{z-1}{z}\right),\pi\right),
		\end{dcases}  \quad \mbox{and} \\
       & f\left(\arccos\left(\frac{z-1}{z}\right)\right)=\frac{1}{2z-1}\left(\frac{2z-1}{z}\right)^{2z}.
      \end{aligned}
      \end{align}
	Hence, it follows from \eqref{C-3-22}, \eqref{C-3-23} and $\beta(z)<\arccos\left(\frac{z-1}{z}\right)$ that
	\begin{align} \label{C-3-24}
		|\mathcal{I}_{24}| \le
		\begin{dcases}
			\frac{a(z)z\kappa}{\sqrt{2z-1}}\left(\frac{2x-1}{z}\right)^z\sum_{i=1}^{N} |\partial_z\theta_i|^2,\hspace{1.8cm}\mbox{for}~~t\in(0,\tau_e),\\
			a(z)z\kappa(1-\cos c^*(z))(1+\cos c^*(z))^z\sum_{i=1}^{N} |\partial_z\theta_i|^2,\quad\mbox{for}~~t\in[\tau_e,\infty).\\
		\end{dcases}
	\end{align}
	In \eqref{C-3-16}, we finally combine all the estimates \eqref{C-3-18}, \eqref{C-3-19}, \eqref{C-3-21} and \eqref{C-3-24} to derive the desired estimate.
	\end{proof}
\begin{remark}
	Since $z$ is a $L^\infty$-random variable, random coefficients $C_1$ and $C_2$ are integrable with respect to $\omega$, for each $0\le t<\infty$, i.e.,
	\[\int_\Omega C_1(t,z(\omega))d\omega<\infty \quad \mbox{and} \quad \int_\Omega C_2(t,z(\omega))d\omega<\infty,\quad\forall~t\ge0.\]
\end{remark}

One can easily figure out that the convergence of asymptotic value of $\|\partial_z\Theta(t)\|$ totally depends on the sign of $C_{1,2}(z)$. If the coefficient $C_{1,2}(z)$ is negative almost surely, i.e., $c^*(z)<\beta(z)$, then $|\partial_z\Theta(z,t)|$ goes to zero exponentially fast in time. Hence, this is crucial in uniform-in-time $H_z^1$-regularity.
\begin{proposition}\label{P3.2}
Suppose that the $\kappa$-assumption holds and $c^*(z)<\beta(z)$ almost surely, and let $\Theta(z,t)$ be a global solution process to \eqref{A-6}. Then, the following estimates hold. \newline
\begin{enumerate}
\item
If $C_{1,1}(z)\ne0$,
	\begin{align*}
		|\partial_z\Theta(t,z) | \le
		\begin{dcases}
			\frac{C_{2,1}(z)}{C_{1,1}(z)}\left(e^{tC_{1,1}(z)}-1\right),\quad\hspace{5cm}  &t\in(0,\tau_e),\\
			\left[ \frac{C_{2,1}(z)}{C_{1,1}(z)}\left(e^{\tau_eC_{1,1}(z)}-1\right)+\frac{C_{2,2}(z)}{C_{1,2}(z)}\right ] e^{C_{1,2}(z)(t-\tau_e)}-\frac{C_{2,2}(z)}{C_{1,2}(z)},~~& t\in[\tau_e,\infty).
		\end{dcases}
	\end{align*}
\item
If $C_{1,1}(z)=0$,
	\begin{align*}
		|\partial_z\Theta(t,z)|\le
		\begin{dcases}
			~tC_{2,1}(z),\hspace{6.9cm} &t\in(0,\tau_e),\\
			\left(\tau_eC_{2,1}(z)+\frac{C_{2,2}(z)}{C_{1,2}(z)}\right)e^{C_{1,2}(z)(t-\tau_e)}-\frac{C_{2,2}(z)}{C_{1,2}(z)},\quad &t\in[\tau_e,\infty).
		\end{dcases}
	\end{align*}
\end{enumerate}
\end{proposition}
\begin{proof} We use the result in Lemma \ref{L3.1} and Gr\"onwall's Lemma.
\end{proof}
\begin{remark}
Note that Proposition \ref{P3.2} implies that $|\partial_z\Theta(t,z) |$ is finite for any finite time $t$.
\end{remark}
Now, we are ready to provide our first result on the propagation of regularity in random space.
\begin{theorem}
\emph{(Local sensitivity analysis)}  \label{T3.2}
Suppose that  the $\kappa$-assumption holds, and $c^*(z)<\beta(z)$ almost surely, and let $\Theta(z,t)$ be a global solution process to \eqref{A-6} {\color{black}with initial data $\Theta^{in}(z)\in\mathcal B_c(0)$, almost surely}. Then, there exists a constant $C>0$ such that
	\[\|\Theta(t)\|_{H^1_z}<C,\quad\forall~t\ge0.\]
\end{theorem}
\begin{proof} By definition of $H_z^1$-norm, we get
	\begin{align}\label{T3.3P1}
		\begin{aligned}
			\|\Theta(t)\|_{H_z^1}^2&=\int_1^\infty |\Theta(t,z)|^2 \rho(z)dz+\int_1^\infty |\partial_z\Theta(t,z)|^2 \rho(z)dz\\
			&=\int_1^{\|z\|_{L^\infty(\Omega)}} |\Theta(t,z)|^2 \rho(z)dz+\int_1^{\|z\|_{L^\infty(\Omega)}} |\partial_z\Theta(t,z)|^2 \rho(z)dz \\
			&=: {\mathcal I}_{31} + {\mathcal I}_{32}.
		\end{aligned}
	\end{align}

\noindent $\bullet$~(Estimate of ${\mathcal I}_{31}$):~By Proposition \ref{P2.1}, one has
	\begin{align} \label{C-4}
		{\mathcal I}_{31} = \int_1^{\|z\|_{L^\infty(\Omega)}} |\Theta(z,t)|^2 \rho(z)dz\le\int_1^{\|z\|_{L^\infty(\Omega)}}Nc^2 \rho(z) dz=(\|z\|_{L^\infty(\Omega)} -1)Nc^2.
	\end{align}
	Note that the upper bound does not depend on time $t$. \newline

\noindent $\bullet$~(Estimate of ${\mathcal I}_{32}$):~We use Proposition \ref{P3.2} and the almost-sure negativity  of $C_{1,2}(z)$ (see \eqref{C-3-12-1} and \eqref{C-3-14-1}) to find
	\begin{align} \label{C-5}
	\lim_{t\to\infty} {\mathcal I}_{32} = \lim_{t\to\infty}\int_1^{ \|z\|_{L^\infty(\Omega)} } |\partial_z\Theta(t,z)|^2 \rho(z)dz=-\int_1^{\|z\|_{L^\infty(\Omega)}}\frac{C_{2,2}(z)}{C_{1,2}(z)}\rho(z)dz<\infty.
	\end{align}
Finally, we combine \eqref{C-4} and \eqref{C-5} to get the uniform boundedness of $\|\Theta(t)\|_{H^1_z}$ in time.
\end{proof}
\section{Death state for the random kinetic Winfree model} \label{sec:4}
\setcounter{equation}{0}
In this section, we study the existence of a measure-valued stationary death state for the random kinetic Winfree model when the natural frequency follows a uniform distribution or a Dirac distribution. Recall the Cauchy problem to the random kinetic Winfree model:
\begin{align}\label{E-0-0}
	\begin{cases}
		\displaystyle \partial_t f +\partial_\theta(f L[f]) = 0, \quad (t, \theta, \nu, z) \in  \bbr_+ \times \bbt \times \bbr \times [1,\infty),  \vspace{0.1cm} \\
		\displaystyle f \Big|_{t  = 0+} = f^{\mathrm{\mathrm{in}}} \geq 0,~~\quad \int_{\bbt\times\mathbb R} f^{\mathrm{\mathrm{in}}} (\theta, \nu{\color{black},z})g(\nu)d\nu d\theta = 1,{\color{black}\quad\mbox{almost surely,}}
	\end{cases}
\end{align}
where the alignment functional $L[f]$ is given as follows:
\begin{align}\label{E-0-1}
	\begin{cases}
		\displaystyle  L[f](t, \theta, \nu, z) = \nu-\sigma(t,z) \sin\theta, \quad  I(\theta, z) = a(z) (1+\cos\theta)^z, \vspace{0.1cm} \\
		\displaystyle  \sigma(t,z) = \kappa \int_{\bbt\times\bbr} I(\theta_*, z) f(t, \theta_*, \nu_*, z) g(\nu_*) d\theta_* d\nu_*.
	\end{cases}
\end{align}
\subsection{A measure-valued stationary death state} \label{sec:4.1}
In this subsection, we consider the existence of a measure-valued stationary death state of \eqref{E-0-0}, which is a stationary solution to \eqref{E-0-0} in the form of
\begin{align}\label{E-0-3}
	f(t,\theta,\nu,z)=\delta(\theta-\theta^*(\nu,z))\quad\mbox{for some}\quad\theta^*(\nu,z)\in\left[-\frac{\pi}{2},\frac{\pi}{2}\right].
\end{align}
Since the stationary solution does not depend on time $t$, we simply denote it as $f^\infty(\theta,\nu,z)$ in this subsection. First, we recall a measure-theoretic formulation for stationary solution $f^{\infty}$ to \eqref{E-0-0}:
\begin{align}\label{E-1}
\begin{dcases}
\int_{\bbt \times \bbr \times [1, \infty)}\Big(  (f^{\infty}L[f^{\infty}]) \partial_\theta \varphi \Big)g(\nu)\rho(z)d \theta d\nu dz = 0, \\
L[f^{\infty}](\theta, \nu, z) = \nu- {\tilde \sigma}(z) \sin\theta, \\
{\tilde \sigma}(z) = \kappa \int_{\bbt\times\bbr} I(\theta_*, z) f^{\infty}(\theta_*, \nu_*, z) g(\nu_*) d\theta_* d\nu_*,
\end{dcases}
\end{align}
where $\varphi = \varphi(\theta, \nu, z)$ be a $C^1_0(\bbt \times \bbr \times [1, \infty))$ test function. In $\eqref{E-1}_3$, we replace $f^\infty$ by the form in \eqref{E-0-3} and to find
\begin{align}\label{E-2}
	\tilde\sigma(z)=\kappa a(z)\int_{\mathbb R}\big(1+\cos\theta^*(\nu_*,z)\big)^zg(\nu_*)d\theta_*d\nu_*.
\end{align}
Since $\theta^*(\nu,z)\in[-\pi/2,\pi/2]$, we have $\tilde\sigma(z)>0$ for any $z$. Now, let $L[f^\infty](\theta^*(\nu,z),\nu,z)=0$, i.e.,
\begin{align*}
	\nu=\tilde\sigma(z)\sin\theta^*(\nu,z)\quad\iff\quad\sin\theta^*(\nu,z)=\frac{\nu}{\tilde\sigma(z)},
\end{align*}
which leads to
\[\cos\theta^*(\nu,z)=\sqrt{1-\frac{\nu^2}{\tilde\sigma(z)^2}}.\]
Then, we can rewrite \eqref{E-2} as
\[\tilde\sigma(z)=\kappa a(z)\int_{\mathbb R}\left( 1+\sqrt{1-\nu_*^2/{\tilde \sigma}(z)^2} \right)^z g(\nu_*) d\nu_*.\]
In the following proposition, we study the meaning of this equation.
\begin{proposition} \label{P4.1}
Let ${\tilde \sigma}: [1, \infty) \to \bbr_+$ be a solution to the following integral equation:
\begin{align} \label{E-3}
{\tilde \sigma}(z) = \kappa a(z) \int_{\mathbb R} \left( 1+\sqrt{1-\nu_*^2/{\tilde \sigma}(z)^2} \right)^z g(\nu_*) d\nu_*.
\end{align}
Then, the state $f^{\infty}(\theta, \nu, z) = \delta(\theta-\theta^*(\nu,z))$ with $\sin\theta^*(\nu,z) = \frac{\nu}{{\tilde \sigma}(z)}$ is a measure-valued stationary death state of \eqref{E-0-0}.
\end{proposition}
\begin{proof}  Let $\tilde{\sigma}$ be a solution satisfying the integral equation \eqref{E-3}, and we define $\theta^*(\nu, z)$ to satisfy the following relation:
\[ \sin\theta^*(\nu,z) = \frac{\nu}{{\tilde \sigma}(z)}. \]
We claim that
\begin{align} \label{E-3-1}
f^{\infty} = \delta(\theta-\theta^*(\nu,z))~\mbox{is a measure-valued state of \eqref{E-0-0}},\quad\mbox{i.e.,}
 \end{align}
 \begin{align} \label{E-4}
  \int_{\bbt \times \bbr \times [1, \infty)} \Big(  (f^{\infty}L[f^{\infty}]) \partial_\theta \varphi \Big)g(\nu)\rho(z) d \theta d\nu dz = 0.
  \end{align}
 {\it Proof of \eqref{E-4}}: By direct calculation, it is easy to see that
\begin{align}
\begin{aligned} \label{E-5}
& \kappa \int_{\bbt\times\bbr} I(\theta_*, z) f^{\infty}(\theta_*, \nu_*, z) g(\nu_*) d\theta_* d\nu_* \\
&  \hspace{0.5cm} = \kappa \int_{\bbt\times\bbr} I(\theta_*, z)   \delta(\theta_*-\theta^*(\nu_*,z)) g(\nu_*) d\theta_*  d\nu_*  =  \kappa \int_{\bbr} I(\theta^*(\nu_*, z), z)  g(\nu_*) d\nu_*  \\
&  \hspace{0.5cm} =  \kappa a(z) \int_{\bbr} (1 + \cos \theta^*(\nu_*, z))^z g(\nu_*) d\nu_* =  \kappa a(z) \int_{\bbr}  \Big(1 + \sqrt{1- \sin^2(\theta^*(\nu_*, z))} \Big)^z g(\nu_*) d\nu_* \\
 & \hspace{0.5cm} = \kappa a(z) \int_{\mathbb R} \left( 1+\sqrt{1-\nu_*^2/{\tilde \sigma}(z)^2} \right)^z g(\nu_*) d\nu_* = {\tilde \sigma}(z).
 \end{aligned}
\end{align}
By $\eqref{E-1}_2$ and \eqref{E-5}, we have
\begin{align} \label{E-6}
 L[f^{\infty}](\theta, \nu, z) = \nu- \tilde{\sigma}(z) \sin\theta = \tilde{\sigma}(z) (\sin \theta^*(\nu, z) - \sin \theta).
\end{align}
Now, we use \eqref{E-3-1} and \eqref{E-6} to see
\begin{align*}
\begin{aligned}
& \int_{\bbt \times \bbr \times [1, \infty)} \Big(  (f^{\infty}L[f^{\infty}]) \partial_\theta \varphi \Big) d \theta d\nu dz  \\
&\hspace{0.5cm} =  \int_{\bbt \times \bbr \times [1, \infty)}  \tilde{\sigma}(z) (\sin \theta^*(\nu, z) - \sin \theta) \delta(\theta-\theta^*(\nu,z)) \partial_\theta \varphi(\theta, \nu, z) d \theta d\nu dz = 0.
\end{aligned}
\end{align*}
\end{proof}
\begin{remark} It follows from Proposition \ref{P4.1} that the existence of measure-valued stationary death state depends on the solvability of \eqref{E-3}. In the next two subsections, we show that the equation \eqref{E-3} with respect to $\tilde\sigma(z)$ has one or two solutions, when the distribution for natural frequency is given as a uniform or Dirac distribution. Any choice between them exhibits a measure-valued stationary death state in the form presented in Proposition \ref{P4.1}.
\end{remark}

In what follows, we investigate the solvability of \eqref{E-3} for two explicit distribution functions $g(\nu)$ for natural frequency.
\subsection{Uniform distribution}  \label{sec:4.2}
Consider a uniform distribution $g$ defined by
\begin{align}\label{E-7}
g(\nu) = \begin{dcases}
\frac{1}{2\gamma}, & \nu\in[1-\gamma, 1+\gamma], \\
0, & \mbox{otherwise},
\end{dcases}
\end{align}
where $\gamma \in (0,1)$.  Then, the relation \eqref{E-3} can be rewritten as
\begin{align}\label{E-8}
\frac{1}{\kappa} = \frac{1}{x} \frac{a(z)}{2\gamma} \int_{1-\gamma}^{1+\gamma} \left( 1+\sqrt{1-\nu_*^2/x^2 } \right)^z d\nu_* =: {\mathcal F}(x, \gamma, z),
\end{align}
for $ (x, \gamma, z) \in [1 + \gamma, \infty) \times (0, 1) \times [1, \infty).$

\subsubsection{Preparatory lemmas} \label{sec:5.2.1} In this part, we provide several lemmas to be used in the solvability of \eqref{E-8}.  Note that the function ${\mathcal F}$ is well-defined for $x \geq 1 + \gamma$. This is due to the following observation:
\[ 1- \frac{\nu_*^2}{x^2} \geq 0, \quad \forall~\nu_* \in [1-\gamma, 1 + \gamma] \quad \mbox{and} \quad  x > 0 \quad \iff \quad  x \geq 1 + \gamma.  \]
\begin{lemma} \label{L4.1}
The equation \eqref{E-8} is solvable if and only if
\begin{align} \label{E-9}
 \max_{x\geq1+\gamma} {\mathcal F}(x, \gamma, z) \geq \frac{1}{\kappa},  \quad \mbox{for}~(\gamma, z) \in (0,1) \times [1, \infty).
 \end{align}
\end{lemma}
\begin{proof}
Let $(\gamma, z) \in (0,1) \times [1, \infty)$ be fixed. Then, the map $x~\mapsto~{\mathcal F}(x, \gamma, z)$ is continuous, and it satisfies
\[ \lim_{x \to \infty} {\mathcal F}(x, \gamma, z) = 0. \]

\noindent (i)~($\Longrightarrow$ direction):~Suppose that the equation \eqref{E-8} is solvable, i.e., there exists ${\hat x} \geq 1 + \gamma$ such that
\[ \frac{1}{\kappa}  = {\mathcal F}({\hat x}, \gamma, z). \]
Hence
\[    \max_{x\geq1+\gamma}\mathcal F(x, \gamma, z)  \geq  {\mathcal F}(\hat{x}, \gamma, z) = \frac{1}{\kappa}.   \]

\vspace{0.2cm}

\noindent (ii)~($\Longleftarrow$ direction):~Suppose that the following relation holds:
\[  \max_{x\geq1+\gamma}\mathcal F(x, \gamma, z) \geq \frac{1}{\kappa}. \]
Then by the intermediate value theorem , there exists a ${\tilde \sigma}(\gamma, z) \geq 1 + \gamma$ such that
\[
\frac{1}{\kappa} = {\mathcal F}({\tilde \sigma}(\gamma, z), \gamma, z),
\]
i.e., the equation \eqref{E-8} is solvable.
\end{proof}
\begin{remark}
By combining the result of Proposition \ref{P4.1} and Lemma \ref{L4.1}, the existence of measure-valued stationary death state depends on the fact that whether ${\mathcal F}$ satisfies the relation \eqref{E-9} or not.
\end{remark}

\subsubsection{Piecewise monotonicity of ${\mathcal F}$} \label{sec:4.2.2}
In this part, we study the piecewise monotonicity of the differentiable function ${\mathcal F}(\cdot, \gamma, z)$. Recall that our aim is to show to study maximum of ${\mathcal F}(\cdot, \gamma, z)$ for fixed $(\gamma, z) \in (0,1) \times [1, \infty)$. \newline

For this, we study the piecewise monotonic behavior of ${\mathcal F}$ by taking a partial derivative of $F$ with respect to $x$:
\begin{align}
\begin{aligned} \label{E-10}
\partial_x {\mathcal F}(x, \gamma, z) &= \frac{a(z)}{2\gamma} \int_{1-\gamma}^{1+\gamma} \frac{\partial}{\partial x} \left[ \frac{1}{x} \left( 1+\sqrt{1-\nu^2/x^2 } \right)^{z} \right] d\nu \\
&= -\frac{a(z)}{2\gamma} \frac{1}{x^2} \int_{1-\gamma}^{1+\gamma} \left( 1+\sqrt{1-\nu^2/x^2 } \right)^{z-1} \left[ 1+\sqrt{1-\nu^2/x^2 } -\frac{z \nu^2/x^2}{\sqrt{1-\nu^2/x^2}} \right] d\nu \\
& = -\frac{a(z)}{2\gamma} \frac{1}{x} \int_{\frac{1-\gamma}{x}}^{\frac{1+\gamma}{x}} \left( 1+\sqrt{1-y^2 } \right)^{z-1} \left[ 1+\sqrt{1-y^2} -\frac{z y^2}{\sqrt{1-y^2}} \right] dy \\
& = -\frac{a(z)}{2\gamma} \frac{1}{x} \int_{\frac{1-\gamma}{x}}^{\frac{1+\gamma}{x}} \frac{\partial}{\partial y} \left[ y \left( 1+\sqrt{1-y^2} \right)^{z} \right] dy \\
& = -\frac{a(z)}{2\gamma} \frac{1}{x} \left[ H\left(\frac{1+\gamma}{x}, z \right) -H\left(\frac{1-\gamma}{x}, z \right) \right],
\end{aligned}
\end{align}
where the function $H$ is define by
\begin{align} \label{E-11}
H(y, z) := y \left( 1+\sqrt{1-y^2} \right)^{z}, \quad \forall~(y, z) \in [0,1] \times [1, \infty).
\end{align}
For the graphical representation of $H$, we refer to Figure \ref{Fig_H} in which for each $z \in [1, 5]$, the maximum value $\displaystyle  \max_{0 \leq y \leq 1} H(y,z)$  is depicted in yellow curve.
\begin{figure}
	\centering
	\mbox{\includegraphics[width=0.5\textwidth]{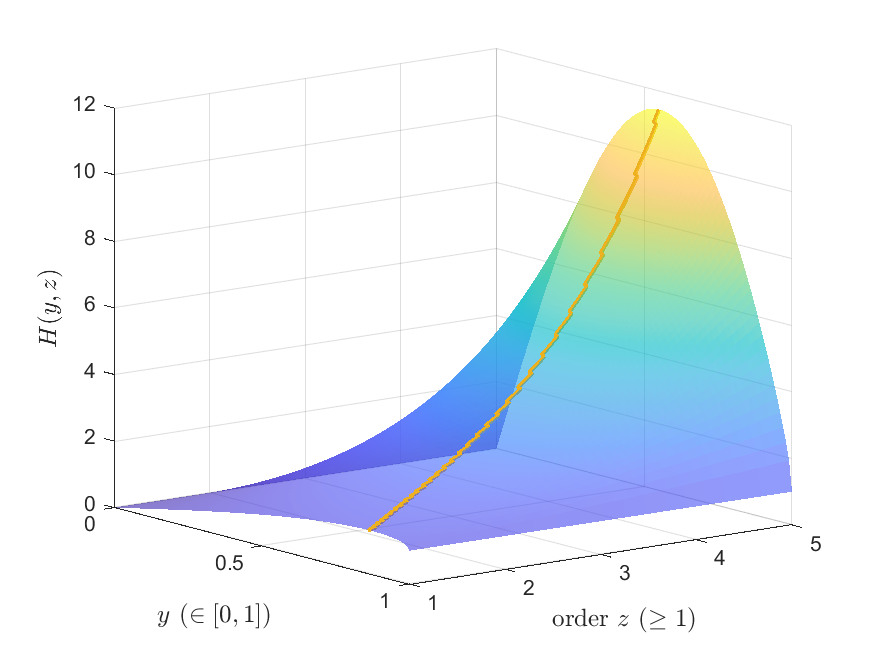}
		\hspace{-.0cm}
		\includegraphics[width=0.5\textwidth]{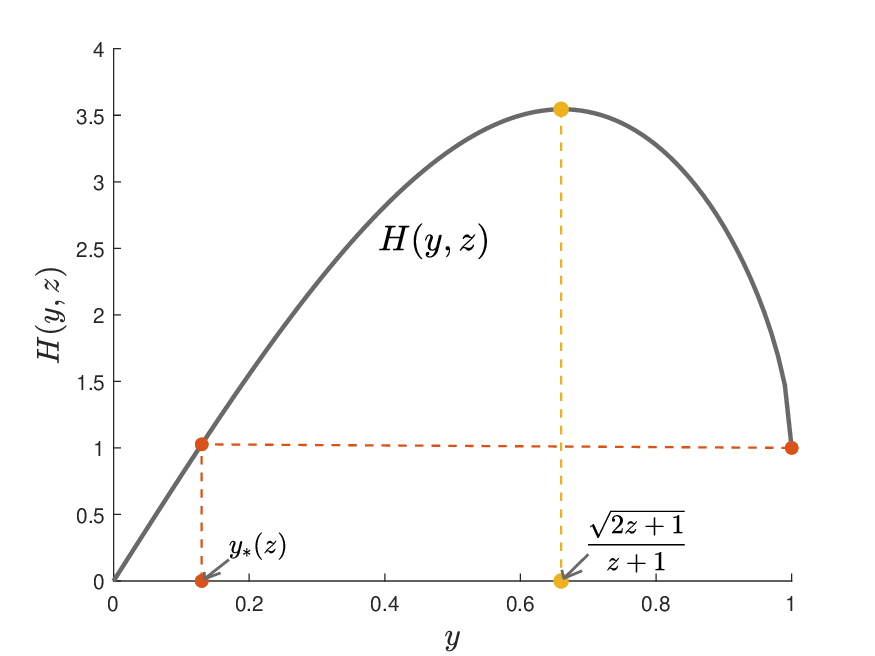}
	}
	\caption{(Left) Surface of $H(u,z)$ and the maximum value for each $z$ in the yellow line. (Right) Graph of $H(u,z)$ when $z=3$.}
	\label{Fig_H}
\end{figure}
In the following lemma, we study the piecewise monotonicity of $H(\cdot, z)$.
\begin{figure}
	\centering
	\mbox{\includegraphics[width=0.5\textwidth]{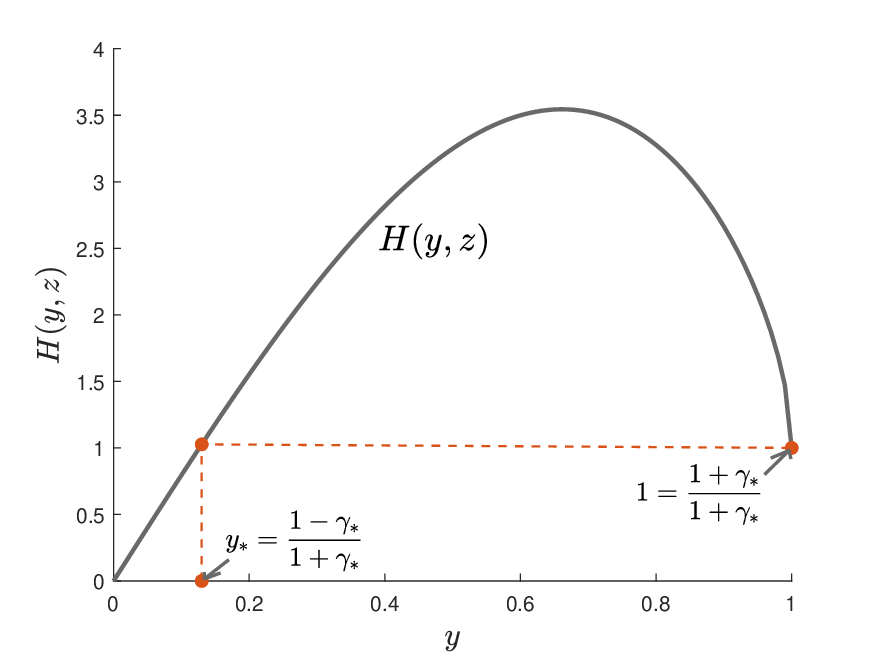}
		\hspace{-.0cm}
		\includegraphics[width=0.5\textwidth]{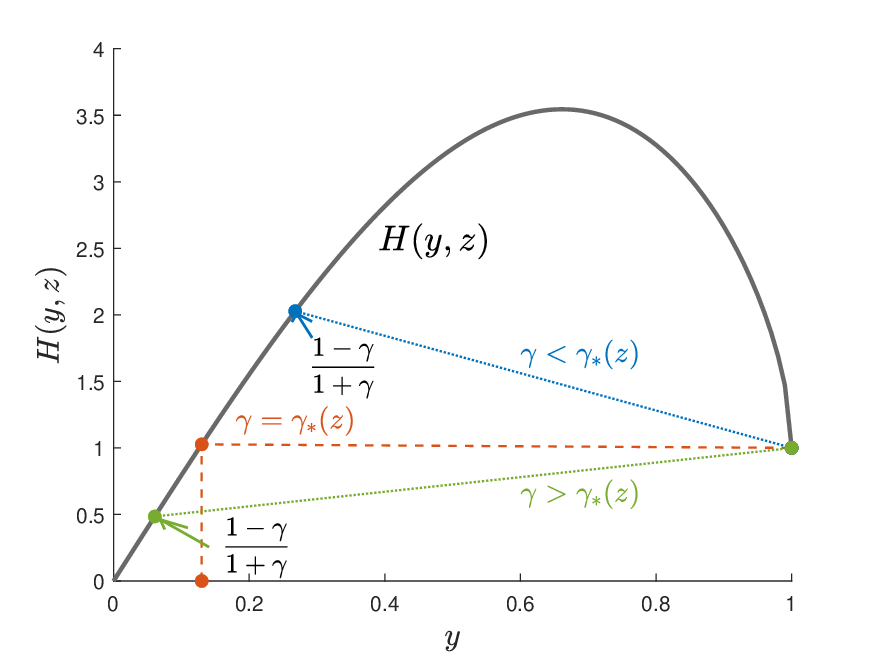}
	}
	\caption{(Left) Schematic diagram of $u_*$ and $\gamma_*$. (Right) Relation between $H(1)$ and $H\left(\frac{1-\gamma}{1+\gamma}\right)$ according to $\gamma$.}
	\label{Fig_H_line}
\end{figure}
\begin{lemma} \label{L4.2}
\begin{enumerate}
\item
The function $H$ is nonnegative and
\[
H(0,z) = 0, \quad H(1,z) = 1, \quad \forall~z \in [1, \infty).
\]
\item
For $z \in [1, \infty)$, the map $y~\mapsto~H(y, z)$ is strictly increasing and strictly decreasing  on the intervals $\displaystyle \left[ 0, \frac{\sqrt{2z+1}}{z+1} \right)$ and $\displaystyle \left(\frac{\sqrt{2z+1}}{z+1}, 1 \right)$, respectively,
\end{enumerate}
\end{lemma}
\begin{proof}
(i)~The first assertion is obvious. \newline

\noindent (ii)~By direct calculation, it is easy to see that for $y\in(0,1)$,
\begin{align*}
	\partial_yH(y,z)&=y\left(1+\sqrt{1-y^2}\right)^z\left(1-\frac{zy^2}{\sqrt{1-y^2}+1-y^2}\right)\frac{1}{y}\\
	&=H(y,z)\frac{\sqrt{1-y^2}+1-y^2-zy^2}{y\left(\sqrt{1-y^2}+1-y^2\right)},
\end{align*}
and
\[  \frac{\sqrt{1-y^2} \left(1 + \sqrt{1-y^2}\right) - y^2z }{y \sqrt{1-y^2} \left(1 + \sqrt{1-y^2}\right)}  = 0 \quad \mbox{and} \quad y \in (0, 1) \quad \Longleftrightarrow \quad y = \frac{\sqrt{2z+1}}{z+1}. \]
\end{proof}
Before we study the maximum point of $\mathcal F(x,\gamma,z)$ for each $(\gamma,z)$ on $[1+\gamma,\infty)$ in detail, note that the sign of $(\partial_x\mathcal F)(\cdot,\gamma,z)$ only depends on
\[H\left(\frac{1+\gamma}{x},z\right)-H\left(\frac{1-\gamma}{x},z\right).\]
Note that this is positive for sufficiently large $x\gg1+\gamma$ (see the right one in Figure \ref{Fig_H}). In particular, at $x=1+\gamma$, we have
\begin{align}\label{E-11-1}
	(\partial_x\mathcal F)(1+\gamma,\gamma,z)&=-\frac{a(z)}{2\gamma}\frac{1}{1+\gamma}\left[H(1,z)-H\left(\frac{1-\gamma}{1+\gamma},z\right)\right].
\end{align}

\vspace{.2cm}

\indent On the other hand, for each $z$, we can find $y_*(z)\in(0,1)$ such that (see the left one in Figure \ref{Fig_H_line})
\begin{align}\label{y_star}
	H(1,z)=H\left(y_*(z),z\right),
\end{align}
and define $\gamma_*(z)\in(0,1)$ such that
\[y_*(z)=\frac{1-\gamma_*(z)}{1+\gamma_*(z)},\quad\mbox{i.e.,}\quad\gamma_*(z):=\frac{1-y_*(z)}{1+y_*(z)}.\]
Since $\gamma\mapsto\frac{1-\gamma}{1+\gamma}$ is monotonically decreasing, there are three cases for \eqref{E-11-1} as follows (see the right one in Figure \ref{Fig_H_line}):
\[(i)~\gamma<\gamma_*(z),\quad(ii)~\gamma=\gamma_*(z),\quad(iii)~\gamma>\gamma_*(z).\]
In other words, this is equivalent to
\[(i)~(\partial_x\mathcal F)(1+\gamma,\gamma,z)>0,\quad(ii)~(\partial_x\mathcal F)(1+\gamma,\gamma,z)=0,\quad(iii)~(\partial_x\mathcal F)(1+\gamma,\gamma,z)<0.\]
In the following lemma, we will show that
\begin{align*}
	(i)~&:~\mathcal F \mbox{ increases until }\tilde x~~\mbox{and then decreases, for some }\tilde x>1+\gamma,\\
	(ii),(iii)~&:~\mathcal F\mbox{ decreases.}
\end{align*}
\begin{lemma} \label{L4.3}
	For $z\in[1,\infty)$, there exists $\gamma_* = \gamma_*(z) > 0$ such that the following two assertions hold:
	\begin{enumerate}
		\item  If $0<\gamma < \gamma_*(z)$, there exists $x_*(\gamma,z) \in (1+\gamma, \infty)$ such that $x~\mapsto~{\mathcal F}(x, \gamma, z)$ is increasing on $[1+\gamma, x_*]$ and decreasing on $[x_*, \infty).$
		\vspace{0.2cm}

		\item If $\gamma \geq \gamma_*(z)$, the map $x~\mapsto~{\mathcal F}(x, \gamma, z)$ is monotonically decreasing.
	\end{enumerate}
\end{lemma}
\begin{proof}
	As it is mentioned above, note that $(\partial_x\mathcal F)(x,\gamma,z)$ is negative for sufficiently large $x\gg1+\gamma$, i.e.,
	\begin{align}\label{E-12}
		(\partial_x\mathcal F)(x,\gamma,z)<0,\quad\mbox{for}~~x\gg1+\gamma,
	\end{align}
	hence, we only consider the zeros of $\partial_x\mathcal F$.

	\vspace{.2cm}

	\noindent(1) By the definition of $\gamma_*(z)$ and the strict decreasing property of the map $\gamma\mapsto\frac{1-\gamma}{1+\gamma}$, we have
	\[(\partial_x\mathcal F)(1+\gamma,\gamma,z)>0,\quad\mbox{for}~~\gamma\in\big(0,\gamma_*(z)\big).\]
	Since we have \eqref{E-12}, it is enough to show that
	\[h(x):=H\left(\frac{1+\gamma}{x},z\right)-H\left(\frac{1-\gamma}{x},z\right)\]
	touches zero only once in $(1+\gamma,\infty)$. Let $x_1,x_2\ge1+\gamma$ be zeros of $h(x)=0$. Without loss of generality, assume that $x_1\ge x_2$. Then, Lemma \ref{L4.2} implies (see Figure \ref{Fig_H_line})
	\begin{align*}
		\frac{1-\gamma}{x_1} \leq \frac{1-\gamma}{x_2} \quad \mbox{and} \quad  \frac{1+\gamma}{x_2} \leq \frac{1+\gamma}{x_1},
	\end{align*}
	i.e.,
	\[x_2\le x_1\quad\mbox{and}\quad x_1\le x_2,\]
	which concludes $x_1 = x_2$ and we complete the proof of our first assertion.

	\vspace{.4cm}

	\noindent(2) Suppose that $\gamma\ge\gamma_*(z)$. By the similar argument above, we have
	\[(\partial_x\mathcal F)(1+\gamma,\gamma,z)\le0,\]
	and it is enough to show that $h$ does not have zeros in $(1+\gamma,\infty)$. Suppose that there exists a zero $x>1+\gamma$ of $h(x)=0$. Once again, by Lemma \ref{L4.2}, we obtain
	\begin{align}\label{E-12-1}
		y_*(z)<\frac{1-\gamma}{x}\quad\mbox{and}\quad\frac{1+\gamma}{x}<1.
	\end{align}
	On the other hand, $\gamma\ge\gamma_*(z)$ yields
	\begin{align}\label{E-12-2}
		\frac{1-\gamma}{1+\gamma}\le\frac{1-\gamma_*(z)}{1+\gamma_*(z)}=y_*(z).
	\end{align}
	The above two relations \eqref{E-12-1} and \eqref{E-12-2} give $x=1+\gamma$, which is a contradiction. Hence, we complete the proof of our second assertion.
\end{proof}

\subsubsection{Solvability of \eqref{E-8}} \label{sec:4.2.3}
In this part, we study the solvability of \eqref{E-8} which will be enough to show the existence of measure-valued stationary death state (see Proposition \ref{P4.1}).  First, we begin with the characterization of $x_*(\gamma,z)$  at which the monotonicity of ${\mathcal F}(\cdot,\gamma,z)$ is changing from the increasing mode to the decreasing mode (see the first assertion in Lemma \ref{L4.3}). \newline

For $z \in [1, \infty)$, we define a set ${\mathcal S}_z$, maps ${\mathcal H}_z:~{\mathcal S}_z~\mapsto~\bbr$ and ${\mathcal A}:~\bbr^2 \to \bbr$ as follows.
\begin{gather}\label{SHA}
\begin{aligned}
& \mathcal S_{z} := \Big \{(x_1, x_2) \in (0,1)^2:~ y_*(z) \leq x_1 < x_2 \leq 1, ~H(x_1, z) = H(x_2,z) \Big \}, \\
& \mathcal H_{z}(x_1, x_2) := \frac{x_2-x_1}{x_2+x_1}, \quad \forall~(x_1, x_2) \in \mathcal S_{z}, \quad {\mathcal A}(x, y) := x + y, \quad \forall~ x, y \in \bbr.
\end{aligned}
\end{gather}
\begin{lemma} \label{L4.4}
The following assertions hold.
\begin{enumerate}
\item
The map $\mathcal H_z$ is one-to-one.
\vspace{0.2cm}
\item
The point $x_* \in (1 + \gamma, \infty)$ appearing in Lemma \ref{L4.3} (i) is explicitly given as a function of $\gamma$:
\[ x_*(\gamma,z) = \frac{2}{{\mathcal A}(\mathcal H_z^{-1} (\gamma))}. \]
\end{enumerate}
\end{lemma}
\begin{proof}
(i)~We choose
\[ (x_1, x_2) \in {\mathcal S}_z \quad \mbox{and} \quad  ({\tilde x}_1,{\tilde x}_2) \in {\mathcal S}_z. \]
Then, we have
\[  y_*(z) \leq  x_1 < x_2 \leq 1 \quad \mbox{and} \quad y_*(z) \leq {\tilde x}_1 < {\tilde x}_2 \leq 1. \]
Without loss of generality, we may assume
\[ x_1 \leq {\tilde x}_1. \]
Then, we use Lemma \ref{L4.1} to see that (see Figure \ref{Fig_H_line})
\begin{align} \label{E-17}
x_1 \leq {\tilde x}_1 < {\tilde x}_2 \leq x_2.
\end{align}
Suppose that
\[
 \mathcal H_z(x_1, x_2) = \mathcal H_z ({\tilde x}_1, {\tilde x}_2).
\]
 This yields
 \begin{align}\label{E-18}
 {\tilde x}_1  x_2  = x_1 {\tilde x}_2.
 \end{align}
 Now, we use \eqref{E-17} and \eqref{E-18} to find
\begin{align*}
1\leq \frac{{\tilde x}_1}{x_1} = \frac{{\tilde x}_2}{x_2} \leq 1,
\end{align*}
which yields
\[ x_1 = {\tilde x}_1 \quad \mbox{and} \quad  x_2 = {\tilde x}_2. \]
 Hence, $\mathcal H_z$ is one-to-one. \newline

\noindent (ii)~From Lemma \ref{L4.3} with \eqref{E-10}, we have
\[
H\left(\frac{1-\gamma}{x_*}, z \right) = H\left(\frac{1+\gamma}{x_*}, z \right) \quad \mbox{and} \quad \Big(\frac{1-\gamma}{x_*},  \frac{1+\gamma}{x_*} \Big) \in {\mathcal S}_z.
\]
This implies
\[ \mathcal H_z \left( \frac{1-\gamma}{x_*}, \frac{1+\gamma}{x_*} \right) =  \frac{ \frac{1+ \gamma}{x_*} - \frac{1-\gamma}{x_*}} { \frac{1+ \gamma}{x_*} + \frac{1-\gamma}{x_*}} = \gamma
\quad \mbox{or equivalently} \quad  x_*(\gamma,z) = \frac{2}{{\mathcal A}(\mathcal H_z^{-1} (\gamma))}.
 \]
\end{proof}
Now we are ready to provide an existence of measure-valued stationary death state for \eqref{E-0-0}.

\begin{theorem} \label{T4.1}
Let $g = g(\nu)$, ${\mathcal F} = {\mathcal F}(x, \gamma, z)$, and $y_* = y_*(z)$ be given as \eqref{E-7}, \eqref{E-8}, and \eqref{y_star}, respectively. Then, the equation \eqref{E-0-0} admits the measure-valued stationary death state if and only if
\begin{align*}
\kappa \geq
\begin{dcases}
\frac{1}{ {\mathcal F} \left(\frac{2}{{\mathcal A}\left(\mathcal H_{z}^{-1}(\gamma)\right)}, \gamma, z \right)}, & 0<\gamma <\gamma_*(z), \\
\frac{1}{ {\mathcal F}(1+\gamma, \gamma, z)}, & \gamma_*(z) \leq \gamma \leq 1, \\
\end{dcases}
\qquad \gamma_*(z) = \frac{1-y_*(z)}{1+y_*(z)}, \quad \forall~z\in[1,\infty),
\end{align*}
where maps $\mathcal H_z$ and ${\mathcal A}$ are defined in \eqref{SHA}.
\end{theorem}
\begin{proof} By Lemma \ref{L4.1}, the solvability of \eqref{E-8} is equivalent to the following relation:
\begin{align} \label{E-19}
\displaystyle  \kappa \geq \frac{1}{ \max_{x\geq1+\gamma} {\mathcal F}(x, \gamma, z)}, \quad \mbox{for}~(\gamma, z) \in (0,1) \times [1, \infty).
 \end{align}
\noindent $\bullet$~Case A: Suppose that
\[  0 < \gamma < \gamma_*(z). \]
Then, it follows from Lemma \ref{L4.3} and Lemma \ref{L4.4} (ii) that
\begin{align} \label{E-20}
\max_{x\geq1+\gamma} {\mathcal F}(x, \gamma, z) = {\mathcal F}(x_*(\gamma,z), \gamma, z) =  {\mathcal F} \left( \frac{2}{{\mathcal A}(\mathcal H_z^{-1} (\gamma))}, \gamma, z \right).
\end{align}

\vspace{0.2cm}

\noindent $\bullet$~Case B: Suppose that
\[   \gamma \geq \gamma_*(z). \]
In this case, by Lemma \ref{L4.3} (ii), we have
\begin{align} \label{E-21}
 \max_{x\geq1+\gamma} {\mathcal F}(x, \gamma, z) = {\mathcal F}(1 + \gamma, \gamma, z).
\end{align}
Finally, we combine \eqref{E-19}, \eqref{E-20} and \eqref{E-21} to find the desired estimate.
\end{proof}

\subsection{Dirac distribution} \label{sec:4.3}
Suppose that natural frequency follows a Dirac distribution concentrated at $\nu_0\in\mathbb R$. In this case, the distribution function $g$ is given by
\[ g(\nu) = \delta(\nu-\nu_0). \]
\indent In the next theorem, we study a sufficient and necessary condition for a measure-valued stationary death state when natural frequency follows a Dirac distribution concentrated at $\nu_0\in\mathbb R$.
\begin{theorem} \label{T4.2}
Suppose that the distribution $g$ of the natural frequency $\nu$ takes the form of
\[  g(\nu) = \delta(\nu-\nu_0) \quad \mbox{for some}\quad\nu\in\mathbb R. \]
Then, the equation \eqref{E-0-0} admits a measure-valued stationary death state if and only if the following relations hold:
\begin{align}\label{E-23}
	\begin{aligned}
		&(i)~\mbox{If $\nu_0=0$,}\qquad\kappa>0.\\
		&(ii)~\mbox{If $\nu_0\ne0$,}\qquad\kappa \geq \frac{|\nu_0|}{a(z)} \frac{z+1}{\sqrt{2z+1}} \left( \frac{z+1}{2z+1} \right)^{z}, \quad \forall~z\geq1.
	\end{aligned}
\end{align}
\end{theorem}
\begin{proof}
(i)~For $\nu_0=0$, the relation \eqref{E-3} becomes
\[\tilde\sigma(z)=\kappa a(z)2^z.\]
Hence, $\tilde\sigma(z)$ exists as a positive real number as long as $\kappa>0$.

\vspace{.2cm}

\noindent(ii)~Assume that $|\nu_0|\ne0$. Then, the relation \eqref{E-3} becomes, with $x=\tilde\sigma(z)$,
\begin{align} \label{E-22}
	\frac{1}{\kappa} = \frac{a(z)}{x} \left( 1+\sqrt{1-\frac{\nu_0^2}{x^2}} \right)^{z} =  \frac{a(z)}{|\nu_0|} \frac{|\nu_0|}{x} \left( 1+\sqrt{1-\frac{|\nu_0|}{x^2}} \right)^{z}      = \frac{a(z)}{|\nu_0|} H\left(\frac{|\nu_0|}{x}, z \right),
\end{align}
where $H$ is defined in \eqref{E-11}. By Lemma \ref{L4.2}, the equation \eqref{E-22} has a solution if and only if $(\kappa, \nu_0, z)$ satisfies
\begin{align*}
\kappa \geq \frac{|\nu_0|}{a(z)} \left( \max_{x\geq|\nu_0|} H\left(\frac{|\nu_0|}{x}, z \right) \right)^{-1} = \frac{|\nu_0|}{a(z)} H\left(\frac{\sqrt{2z+1}}{z+1}, z \right)^{-1} = \frac{|\nu_0|}{a(z)} \frac{z+1}{\sqrt{2z+1}} \left( \frac{z+1}{2z+1} \right)^{z},
\end{align*}
for all $z \geq 1$.
\end{proof}
\begin{remark} \label{R4.4}
It is worthwhile to note that the right-hand side of \eqref{E-23} converges, as $z \to \infty$. Recall Stirling's formula
\begin{align*}
	\Gamma(x+1) \sim \sqrt{2\pi x} \bigg( \frac{x}{e} \bigg)^x \quad \mbox{as} \quad x~\to~\infty
\end{align*}
to see
\begin{align} \label{E-24}
a(z) = \frac{2^z\Gamma(z+1)^2}{\Gamma(2z+1)} \sim \frac{2^z \cdot 2\pi z (z/e)^{2z}}{\sqrt{4\pi z}(2z/e)^{2z}} \quad \mbox{as} \quad z~\to~\infty.
\end{align}
Now, we use the asymptotic estimate for $a(z)$ in \eqref{E-24} to see that
\begin{align*}
\begin{aligned}
& \lim_{z\to\infty} \frac{1}{a(z)} \frac{z+1}{\sqrt{2z+1}} \left( \frac{z+1}{2z+1} \right)^{z} = \lim_{z\to\infty} \frac{\sqrt{4\pi z}(2 z/e)^{2z}}{2^z \cdot 2\pi z (z/e)^{2z}} \cdot \frac{z+1}{\sqrt{2z+1}} \bigg( \frac{2z+1}{z+1} \bigg)^{-z} \\
&\hspace{1cm} = \lim_{z\to\infty} \frac{(2z/e)^{2z}}{(2z/e)^{2z}} \cdot \frac{\sqrt{4\pi z}}{\sqrt{2z+1}} \cdot \frac{z+1}{2\pi z} \bigg( \frac{2z+1}{2z+2} \bigg)^{-z} = \sqrt{\frac{e}{2\pi}}.
\end{aligned}
\end{align*}
\end{remark}
\section{Local sensitivity analysis kinetic Winfree model} \label{sec:5}
\setcounter{equation}{0}
In this section, we study the local sensitivity analysis to the random kinetic Winfree model. First, we recall the Cauchy problem to the random Winfree model with high-order coupling:
\begin{align}
\begin{cases} \label{D-0}
		\displaystyle \partial_t f +\partial_\theta(f L[f]) = 0, \quad (t, \theta, \nu, z) \in  \bbr_+ \times \bbt \times \bbr \times [1,\infty),  \vspace{0.1cm} \\
		\displaystyle f \Big|_{t  = 0+} = f^{\mathrm{\mathrm{in}}} \geq 0,~~\quad \int_{\bbt\times\mathbb R} f^{\mathrm{\mathrm{in}}} (\theta, \nu{\color{black},z})g(\nu)d\nu d\theta = 1,{\color{black}\quad\mbox{almost surely,}}
\end{cases}
\end{align}
where the alignment functional $L[f]$ is given as follows:
\begin{align}
\begin{cases}  \label{D-1}
 \displaystyle  L[f](t, \theta, \nu, z) = \nu-\sigma(t,z) \sin\theta, \quad  I(\theta, z) = a(z) (1+\cos\theta)^z, \vspace{0.1cm} \\
 \displaystyle  \sigma(t,z) = \kappa \int_{\bbt\times\bbr} I(\theta_*, z) f(t, \theta_*, \nu_*, z) g(\nu_*) d\theta_* d\nu_*.
\end{cases}
\end{align}
For the local sensitivity analysis of \eqref{D-0}, we employ the probability density functions $\rho(z)$ and $g(\nu)$ for a random variable $z$ and natural frequency $\nu$, respectively, and define the weighted norm in $L_{\theta, z}^p (\bbt\times [1, \infty))$, $L_{\theta, \nu, z}^p (\bbt \times \bbr \times [1, \infty))$, and $W^{k, p}_{\theta, \nu, z} (\bbt \times \mathbb R \times [1, \infty))$ for $p\in[1,\infty]$ and $k\in\mathbb{N}$, as follows:
\begin{align*}
\begin{aligned}
	& \|h\|_{L^p_{\theta,z}} := \left( \int_{\bbt\times[1,\infty)} |h(\theta, z)|^p \rho(z) d\theta dz  \right)^{1/p}, \\
	& \|h\|_{L^p_{\theta, \nu, z}} := \left( \int_{\bbt\times\bbr\times[1,\infty)} |h(\theta, \nu, z)|^p \rho(z)g(\nu) d\theta d\nu dz \right)^{1/p}, \quad  \|h\|_{W^{k,p}_{\theta, \nu, z}} := \left(\sum_{l=0}^k \|\partial_z^l h\|_{L^p_{\theta, \nu, z}}^p\right)^{1/p}.
\end{aligned}
\end{align*}
Next, we estimate the local sensitivity analysis in the following subsection.
\subsection{Propagation of $\| \partial_\theta^k f(t)\|_{L^p_{\theta, \nu, z}}$} \label{sec:5.1}
In this subsection, we discuss a sufficient framework for the propagation of regularity of a solution to \eqref{D-0}-\eqref{D-1}:
\[  \| \partial_\theta^k f(t)\|_{L^p_{\theta, \nu, z}}  < \infty, \quad \forall~t > 0. \]
One can refer to \cite{DPZ,MTZ,ZJ} on the propagation of random regularity to Vlasov-Fokker-Planck-type equations. First, we define a sequence of constants $\Lambda_{0,k}$ as
	\begin{align*}
		\Lambda_{0,0} = 1, \quad \Lambda_{0,k} = 2^k \sum_{l=0}^{k-1} \Lambda_{0,l}, \quad \forall~k\geq1.
	\end{align*}
We begin with the propagation of $\theta$-regularity.
\begin{proposition} \label{P5.1}
\emph{(Propagation of $\theta$-regularity)}
Suppose that system parameters and initial datum satisfy
\begin{gather*}
	0\le k, \quad 1<p\le\infty,\quad C := \|  I \|_{L^\infty_{\theta,z}} <+\infty,\\
	\partial_\theta^l f^{\mathrm{\mathrm{in}}} \in L^p_{\theta, \nu, z}(\mathbb T\times\mathbb R\times[1,\infty)), \quad \forall~l \in \{ 0 \} \cup [k],
\end{gather*}
and let $f = f(t, \theta, \nu, z)$ be a global solution process of \eqref{D-0} - \eqref{D-1}. Then, we have
\[ \partial_\theta^k f(t) \in L^p_{\theta, \nu, z}(\mathbb T\times\mathbb R\times[1,\infty)), \quad \mbox{$\forall~t>0$}. \]
In particular, we have
	\begin{align}\label{D-2}
		\| \partial_\theta^k f(t)\|_{L^p_{\theta, \nu, z}} \leq \Lambda_{0,k}e^{(k+1)C\kappa t} \sum_{l=0}^k \| \partial_\theta^l f^{\mathrm{\mathrm{in}}}\|_{L^p_{\theta, \nu, z}},\quad \forall~t\ge0.
	\end{align}
\end{proposition}
\begin{proof}
The proof is quite lengthy and tedious, hence, we leave it in Appendix \ref{Proof_P5.1}.
\end{proof}

\subsection{Propagation of $\|f(t)\|_{W^{k,p}_{\theta, \nu, z}}$} \label{sec:5.2}
In this subsection, we study the propagation of $\|f(t)\|_{W^{k,p}_{\theta, \nu, z}}$. For this, we need two a priori lemmas.
\begin{lemma} \label{L5.1}
Suppose that the following conditions hold:
\[  1\le k,\quad0\le l,\quad1<p\le\infty, \quad C:= \max_{0\leq i\leq k} \|\partial^i_zI\|_{L^\infty_{\theta,z}} <\infty, \]
and let $f = f(t, \theta, \nu, z)$ be a global solution of \eqref{D-0}-\eqref{D-1}.  Then, we have
	\[\|\partial_z^k\partial_\theta^l L[f](t)\|_{L^p_{\theta, \nu, z}} \le 2^k \pi C\kappa\|f(t)\|_{W^{k,p}_{\theta, \nu,z}},\quad \forall~ t\ge0.\]
\end{lemma}
\begin{proof}
	Note that
	\begin{align}
	\begin{aligned} \label{D-6-0}
		\partial_z^k \partial_\theta^l L[f](t, \theta, z)&=-\kappa \big( \partial_\theta^l \sin\theta \big) \int_{\mathbb{T}\times\mathbb{R}}\partial_z^k\Big(I(\theta_*,z)f(t, \theta_*, \nu_*, z)\Big) g(\nu_*)d\nu_*d\theta_* \\
		&=-\kappa \big( \partial_\theta^l \sin\theta \big) \int_{\mathbb{T}\times\mathbb{R}} \bigg[ \sum_{i=0}^k\binom{k}{i}\partial_z^iI(\theta_*, z)\cdot\partial_z^{k-i}f(t, \theta_*, \nu_*, z) \bigg] g(\nu_*)d\nu_*d\theta_*,
	\end{aligned}
	\end{align}
	Then, we use \eqref{D-6-0} and Jensen's inequality to find
	\begin{align}
	\begin{aligned} \label{D-6-1}
		&|\partial_z^k \partial_\theta^l L[f](t, \theta, z)|^p \\
		& \hspace{1cm} \le \kappa^p\left(\sum_{j=0}^k\binom{k}{j}\right)^{p-1}\cdot \sum_{i=0}^k\binom{k}{i}\int_{\mathbb{T}\times\mathbb{R}}\big|\partial_z^iI(\theta_*, z)\cdot\partial_z^{k-i}f(t, \theta_*, \nu_*, z)\big|^pg(\nu_*)d\nu_*d\theta_* \\
		&\hspace{1cm} = 2^{k(p-1)} \kappa^p \sum_{i=0}^k\binom{k}{i}\int_{\mathbb{T}\times\mathbb{R}}\big|\partial_z^iI(\theta_*,z)\cdot\partial_z^{k-i}f(t, \theta_*, \nu_*, z)\big|^pg(\nu_*)d\nu_*d\theta_*.
	\end{aligned}
	\end{align}
	Now, we use \eqref{D-6-1} to obtain
	\begin{align*}
		& \|\partial_z^k \partial_\theta^l L[f](t)\|_{L^p_{\theta, \nu, z}}^p \\
		& \hspace{1cm} =\int_{\mathbb{T}\times\mathbb{R}\times[1,\infty)}|\partial_z^k \partial_\theta^l L[f](t, \theta, z)|^p\rho(z)g(\nu)dzd\nu d\theta\\
		& \hspace{1cm} \le 2^{k(p-1)} \kappa^p \sum_{i=0}^k\binom{k}{i} \\
		& \hspace{1cm} \times \int_{\mathbb{T}\times\mathbb{R}\times[1,\infty)}\int_{\mathbb{T}\times\mathbb{R}}\big|\partial_z^iI(\theta_*,z)\cdot\partial_z^{k-i}f(t, \theta_*, \nu_*,z)\big|^pg(\nu_*)\rho(z)g(\nu)d\nu_*d\theta_*dzd\nu d\theta\\
		& \hspace{1cm} \le2\pi \cdot 2^{k(p-1)} \kappa^p\sum_{i=0}^k\binom{k}{i}\|\partial_z^iI\|^p_{L^\infty_{\theta,z}}\int_{[1,\infty)}\int_{\mathbb{T}\times\mathbb{R}}|\partial_z^{k-i}f(t, \theta_*, \nu_*, z)|^pg(\nu_*)\rho(z)d\nu_*d\theta_*dz\\
		& \hspace{1cm} \le \pi \big( 2^k C\kappa \big)^p \sum_{i=0}^k\int_{\mathbb{T}\times\mathbb{R}\times[1,\infty)}|\partial_z^{k-i}f(t, \theta_*, \nu_*,z)|^p\rho(z)g(\nu_*)dzd\nu_*d\theta_*\\
		& \hspace{1cm} = \pi \big( 2^k C\kappa \big)^p \|f\|^p_{W^{k,p}_{\theta, \nu, z}}.
	\end{align*}
	This implies the desired result for finite $p$. Likewise the previous proposition, we can obtain the case of $p=\infty$.
\end{proof}
Now, we investigate the propagation of $\big\|\partial_z^k\partial_\theta^lf(t)\big\|_{L^p_{\theta,\nu,z}}$in time. For this, we introduce a notation to be used in a priori condition. For parameters
\[1\le k,\quad0\le l,\quad1<p\le\infty,\quad0<T<\infty,\]
and a global solution $f=f(t,\theta,\nu,z)$ to \eqref{D-0}-\eqref{D-1}, define a constant $M_{k,l,p}$ as follows: if $l=0$,
\begin{align}\label{def_M_0}
	M_{k,0,p}:=\max_{\substack{0\leq i\leq k-1 \\ 0\leq j\leq 1}} \sup_{0\leq t\leq T}\big\| \partial_z^i \partial_\theta^j f(t) \big\|_{L^\infty_{\theta, \nu,z}},
\end{align}
and if $l\ge1$,
\begin{align}\label{def_M}
	M_{k,l,p}:=	\max\left\{\max_{\substack{0\le i\le k-1\\ 0\le j\le l+1}}\sup_{0\le t\le T}\big\|\partial_z^i\partial_\theta^jf(t)\big\|_{L^\infty_{\theta,\nu,z}},~~\max_{0\le j\le l-1}\sup_{0\le t\le T}\big\|\partial_z^k\partial_\theta^jf(t)\big\|_{L^p_{\theta,\nu,z}}\right\}.
\end{align}
Although there is no $p$-dependency in $M_{k,0,p}$, we keep this notation for consistency with other $M_{k,l,p}$.
\begin{lemma} \label{L5.2}
Suppose that parameters satisfy
\[1\le k,\quad0\le l,\quad 1<p\le\infty,\quad0<T<\infty,\]
and let $f = f(t, \theta, \nu, z)$ be a global solution of \eqref{D-0}-\eqref{D-1} satisfying a priori conditions:
	\begin{gather*}
	 C := \max_{0\leq i\leq k} \|\partial^i_zI\|_{L^\infty_{\theta,z}} <\infty\quad\mbox{and}\quad M_{k,l,p}<\infty,
	\end{gather*}
	where $M_{k,l,p}$ is defined in \eqref{def_M_0}-\eqref{def_M}. Then, there exist two nonnegative constants $A_{k, l,p}$ and $B_{k, l, p}$ depending on $\kappa$, $C$ and $M_{k,l, p}$ such that
	\begin{align*}
		\partial_t \|\partial_z^k \partial_\theta^l f(t)\|_{L^p_{\theta,\nu,z}} \leq  A_{k, l,p} \|\partial_z^k \partial_\theta^l f(t)\|_{L^p_{\theta,\nu,z}} +B_{k,l,p}, \quad \forall~t\in[0, T].
	\end{align*}
\end{lemma}
\begin{proof}
Since the proof is lengthy, we leave it in Appendix \ref{Proof_L5.2}.
\end{proof}
\begin{remark}\label{R5.1}
	If one wishes to verify the finiteness of $\big\|\partial_z^k\partial_\theta^lf(t)\big\|_{L^p_{\theta,\nu,z}}$ in the bounded time interval $[0, T]$, Lemma \ref{L5.2} can be summarized as follows: if $M_{k,l,p}$ is finite, then by Gr\"onwall's Lemma,
	\[\sup_{0\le t\le T}\big\|\partial_z^k\partial_\theta^lf(t)\big\|_{L^p_{\theta,\nu,z}}<\infty,\quad\mbox{for any}\quad 0<T<\infty.\]
\end{remark}
Here, we give the last preparatory lemma.
\begin{lemma}\label{L5.3}
	Suppose that parameters and initial datum satisfy
	\begin{gather*}
		1\le k,\quad0<T<\infty,\quad C:=\max_{0\le l\le k}\big\|\partial_z^lI\big\|_{L^\infty_{\theta,z}}<\infty,\\
		\partial_\theta^jf^{\mathrm{in}}\in L^\infty_{\theta,\nu,z}(\mathbb T\times\mathbb R\times[1,\infty)),\quad\forall j~\in\left\{0\right\}\cup[k],
	\end{gather*}
	and let  $f=f(t,\theta,\nu,z)$ be a global solution to \eqref{D-0}-\eqref{D-1}. Then, one has
	\[M_{i,0,\infty}<\infty,\quad\forall~i=1,\cdots,k-1.\]
\end{lemma}
\begin{proof}
	First, we consider the condition for $M_{i,0,\infty}<\infty$ with $i=2,\cdots,k$. From the definition of $M_{i,0,\infty}$ in \eqref{def_M_0} and Remark \ref{R5.1}, the finiteness of $M_{i,0,\infty}$ is guaranteed if
	\begin{align}\label{step1}
		\max_{\substack{0\le \tilde i\le i-1\\0\le j\le 1}}M_{\tilde i,j,\infty}=M_{i-1,1,\infty}<\infty.
	\end{align}
	On the other hand, the condition for $M_{i,l,\infty}<\infty$ with $l\ge1$ can be obtained, again by Remark \ref{R5.1} and the definition of constants $M_{i,l,\infty}$, as follows:
	\[\max_{\substack{0\le \tilde i\le i-1\\0\le j\le l+1}}M_{\tilde i,j,\infty}<\infty\quad\mbox{and}\quad\max_{0\le j\le l-1}M_{i,j,\infty}<\infty,\]
	that is,
	\begin{align}\label{step2}
		M_{i-1,l+1,\infty}<\infty\quad\mbox{and}\quad M_{i,l-1,\infty}<\infty.
	\end{align}
	\begin{figure}
		\mbox{
			\begin{tikzpicture}[scale=0.55] 
				\footnotesize

				\def\blockwidth{2.4}
				\def\blockheight{1.8}

				\draw[draw=none] (3*\blockwidth, 3*\blockheight) rectangle ++(\blockwidth, \blockheight);
				\draw[draw=none] (3*\blockwidth, 4*\blockheight) rectangle ++(\blockwidth, \blockheight);
				\draw[draw=none] (4*\blockwidth, 4*\blockheight) rectangle ++(\blockwidth, \blockheight);

				\foreach \i in {1,...,2} {
					\draw[fill={rgb,255:red,255; green,250; blue,168}] (\blockwidth, \i*\blockheight) rectangle ++(\blockwidth, \blockheight);
					\draw (2*\blockwidth+\i*\blockwidth, 2*\blockheight+\i*\blockheight) -- (3*\blockwidth+\i*\blockwidth, 2*\blockheight+\i*\blockheight);
					\draw (3*\blockwidth+\i*\blockwidth, 2*\blockheight+\i*\blockheight) -- (3*\blockwidth+\i*\blockwidth, 3*\blockheight+\i*\blockheight);
				}

				\draw(2*\blockwidth, 5*\blockheight) rectangle ++(\blockwidth, \blockheight);
				\draw (3*\blockwidth, 5*\blockheight) rectangle ++(2*\blockwidth, \blockheight);
				\draw(5*\blockwidth, 5*\blockheight) rectangle ++(\blockwidth, \blockheight);


				\draw[fill={rgb,255:red,255; green,250; blue,168}] (\blockwidth, 5*\blockheight) rectangle ++(\blockwidth, \blockheight);
				\draw[fill={rgb,255:red,255; green,250; blue,168}] (1*\blockwidth, 3*\blockheight) rectangle ++(\blockwidth, 2*\blockheight);
				\draw(2*\blockwidth, 3*\blockheight) rectangle ++(\blockwidth, 2*\blockheight);
				\draw(2*\blockwidth, 2*\blockheight) rectangle ++(\blockwidth, \blockheight);
				\node at (1*\blockwidth+0.5*\blockwidth, 1*\blockheight+0.5*\blockheight) {$\partial_z^k f$};
				\node at (1*\blockwidth+0.5*\blockwidth, 2*\blockheight+0.5*\blockheight) {$\partial_z^{k-1} f$};
				\node at (2*\blockwidth+0.5*\blockwidth, 2*\blockheight+0.5*\blockheight) {$\partial_z^{k-1} \partial_\theta f$};
				\node at (1*\blockwidth+0.5*\blockwidth, 5*\blockheight+0.5*\blockheight) {$\partial_z f$};
				\node at (2*\blockwidth+0.5*\blockwidth, 5*\blockheight+0.5*\blockheight) {$\partial_z\partial_\theta f$};
				\node at (5*\blockwidth+0.5*\blockwidth, 5*\blockheight+0.5*\blockheight) {$\partial_z\partial_\theta^{k-1} f$};
				\node at (1*\blockwidth+0.5*\blockwidth, 4*\blockheight) {$\vdots$};
				\node at (2*\blockwidth+0.5*\blockwidth, 4*\blockheight) {$\vdots$};
				\node at (4*\blockwidth, 5*\blockheight+0.5*\blockheight) {$\cdots$};
				\node at (3*\blockwidth+0.5*\blockwidth, 3*\blockheight+0.5*\blockheight) {$\cdots$};
				\node at (4*\blockwidth+0.5*\blockwidth, 4*\blockheight+0.5*\blockheight) {$\cdots$};
			\end{tikzpicture}\hspace{.5cm}
			\begin{tikzpicture}[scale=0.55] 
				\footnotesize

				\def\blockwidth{2.4}
				\def\blockheight{1.8}

				\draw[fill={rgb,255:red,251; green,235; blue,232}, draw=none] (3*\blockwidth, 3*\blockheight) rectangle ++(\blockwidth, \blockheight);
				\draw[fill={rgb,255:red,251; green,235; blue,232}, draw=none] (3*\blockwidth, 4*\blockheight) rectangle ++(\blockwidth, \blockheight);
				\draw[fill={rgb,255:red,251; green,235; blue,232}, draw=none] (4*\blockwidth, 4*\blockheight) rectangle ++(\blockwidth, \blockheight);

				\foreach \i in {1,...,2} {
					\draw[fill={rgb,255:red,251; green,235; blue,232}] (\blockwidth, \i*\blockheight) rectangle ++(\blockwidth, \blockheight);
					\draw (2*\blockwidth+\i*\blockwidth, 2*\blockheight+\i*\blockheight) -- (3*\blockwidth+\i*\blockwidth, 2*\blockheight+\i*\blockheight);
					\draw (3*\blockwidth+\i*\blockwidth, 2*\blockheight+\i*\blockheight) -- (3*\blockwidth+\i*\blockwidth, 3*\blockheight+\i*\blockheight);
				}

				\draw[fill={rgb,255:red,249; green,195; blue,210}] (2*\blockwidth, 5*\blockheight) rectangle ++(\blockwidth, \blockheight);
				\draw[fill={rgb,255:red,249; green,195; blue,210}] (3*\blockwidth, 5*\blockheight) rectangle ++(2*\blockwidth, \blockheight);
				\draw[fill={rgb,255:red,249; green,195; blue,210}] (5*\blockwidth, 5*\blockheight) rectangle ++(\blockwidth, \blockheight);


				\draw[fill={rgb,255:red,249; green,195; blue,210}] (\blockwidth, 5*\blockheight) rectangle ++(\blockwidth, \blockheight);
				\draw[fill={rgb,255:red,251; green,235; blue,232}] (1*\blockwidth, 3*\blockheight) rectangle ++(\blockwidth, 2*\blockheight);
				\draw[fill={rgb,255:red,251; green,235; blue,232}] (2*\blockwidth, 3*\blockheight) rectangle ++(\blockwidth, 2*\blockheight);
				\draw[fill={rgb,255:red,251; green,235; blue,232}] (2*\blockwidth, 2*\blockheight) rectangle ++(\blockwidth, \blockheight);
				\node at (1*\blockwidth+0.5*\blockwidth, 1*\blockheight+0.5*\blockheight) {$\partial_z^k f$};
				\node at (1*\blockwidth+0.5*\blockwidth, 2*\blockheight+0.5*\blockheight) {$\partial_z^{k-1} f$};
				\node at (2*\blockwidth+0.5*\blockwidth, 2*\blockheight+0.5*\blockheight) {$\partial_z^{k-1} \partial_\theta f$};
				\node at (1*\blockwidth+0.5*\blockwidth, 5*\blockheight+0.5*\blockheight) {$\partial_z f$};
				\node at (2*\blockwidth+0.5*\blockwidth, 5*\blockheight+0.5*\blockheight) {$\partial_z\partial_\theta f$};
				\node at (5*\blockwidth+0.5*\blockwidth, 5*\blockheight+0.5*\blockheight) {$\partial_z\partial_\theta^{k-1} f$};
				\node at (1*\blockwidth+0.5*\blockwidth, 4*\blockheight) {$\vdots$};
				\node at (2*\blockwidth+0.5*\blockwidth, 4*\blockheight) {$\vdots$};
				\node at (4*\blockwidth, 5*\blockheight+0.5*\blockheight) {$\cdots$};
				\node at (3*\blockwidth+0.5*\blockwidth, 3*\blockheight+0.5*\blockheight) {$\cdots$};
				\node at (4*\blockwidth+0.5*\blockwidth, 4*\blockheight+0.5*\blockheight) {$\cdots$};
			\end{tikzpicture}
}
	\caption{Structure of estimates in Lemma \ref{L5.3}: In the left diagram, the aim of Lemma \ref{L5.3} is colored in yellow. The right diagram describes the step-shape strategy for estimates. By \eqref{step1} and \eqref{step2}, it suffices to show the dark pink ones for whole pink area.}
	\label{Fig_step}
	\end{figure}
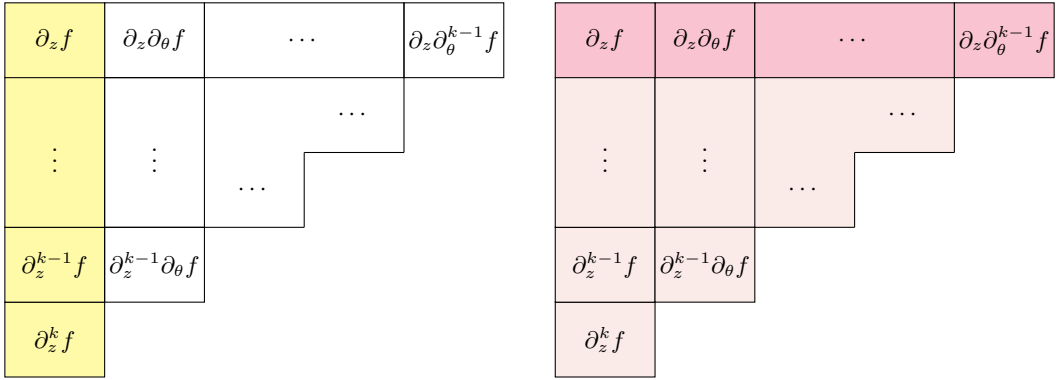
	From \eqref{step1} and \eqref{step2}, one can notice that to show
	\[M_{i,0,\infty}<\infty,\quad\mbox{for all}\quad i=2,\cdots,k,\]
	the following is needed: (see Figure \ref{Fig_step})
	\begin{align*}
		\max_{1\le i\le k}\max_{0\le j\le k-i}M_{i,j,\infty}<\infty,
	\end{align*}
	which is equivalent to
	\[M_{1,j,\infty}<\infty,\quad\forall~j=0,\cdots,k-1.\]
	By Proposition \ref{P5.1}, one has
	\[\sup_{0\le t\le T}\big\|\partial_\theta^jf(t)\big\|_{L^\infty_{\theta,\nu,z}}<\infty,\quad\forall~0\le j\le k-1,\]
	hence,
	\begin{align}\label{M_0_finite}
		M_{1,0,\infty}<\infty.
	\end{align}
	As mentioned in Remark \ref{R5.1}, we have
	\[\sup_{0\le t\le T}\big\|\partial_zf(t)\big\|_{L^\infty_{\theta,\nu,z}}<\infty\]
	by Lemma \ref{L5.2} and Gr\"onwall's Lemma. Here, we give a detail: By Lemma \ref{L5.2} with \eqref{M_0_finite}, we have two nonnegative constants
	\[A_{1,0,\infty}=A_{1,0,\infty}(T,C,\kappa,f^{\mathrm{in}})\quad\mbox{and}\quad B_{1,0,\infty}=B_{1,0,\infty}(T,C,\kappa,f^{\mathrm{in}})\]
	such that
	\[\partial_t\big\|\partial_zf(t)\big\|_{L^\infty_{\theta,\nu,z}}\le A_{1,0,\infty}\big\|\partial_zf(t)\big\|_{L^\infty_{\theta,\nu,z}}+B_{1,0,\infty},\quad\forall~t\in[0,T].\]
	Then, Gr\"onwall's Lemma yields
	\[\big\|\partial_zf(t)\big\|_{L^\infty_{\theta,\nu,z}}\le \frac{B_{1,0,\infty}}{A_{1,0,\infty}}\left(e^{A_{1,0,\infty}t}-1\right),\quad\forall~t\in[0,T].\]
	This leads to
	\[\sup_{0\le t\le T}\big\|\partial_zf(t)\big\|_{L^\infty_{\theta,\nu,z}}<\infty.\]
	From the definition of $M_{k,l,\infty}$ in \eqref{def_M} for $l\ge1$, one has
	\begin{align*}
		M_{1,l,\infty}&:=\max\left\{\max_{0\le j\le l+1}\sup_{0\le t\le T}\big\|\partial_\theta^jf(t)\big\|_{L^\infty_{\theta,\nu,z}},~~\max_{0\le j\le l-1}\sup_{0\le t\le T}\big\|\partial_z\partial_\theta^jf(t)\big\|_{L^\infty_{\theta,\nu,z}}\right\},
	\end{align*}
	which implies by Remark \ref{R5.1} that $M_{1,l,\infty}$ is finite if
	\begin{align}\label{cond_M1}
		\max_{0\le j\le l+1}\sup_{0\le t\le T}\big\|\partial_\theta^jf(t)\big\|_{L^\infty_{\theta,\nu,z}}<\infty\quad\mbox{and}\quad M_{1,l-1,\infty}<\infty.
	\end{align}
	Note that the first condition in \eqref{cond_M1} holds for all $l\ge0$ by Proposition \ref{P5.1}. Since we have $M_{1,0,\infty}<\infty$ in \eqref{M_0_finite}, we can obtain $M_{1,l,\infty}<\infty$ for all $1\le l\le k-1$, inductively. This completes the desired estimate.
\end{proof}
Now, we are ready to provide our second main result on the propagation of regularity of the norm $\|f\|_{W^{k,p}_{\theta, \nu, z}}$.
\begin{theorem} \label{T5.1}
Suppose that parameters and initial datum satisfy
\begin{gather*}
	1\le k,\quad1<p\le\infty, \quad0<T<\infty, \\
	C:= \max_{0\leq l\leq k} \|\partial^l_zI\|_{L^\infty_{\theta,z}} <\infty, \quad  \partial_\theta^j f^{\mathrm{\mathrm{in}}}\in L^\infty_{\theta, \nu, z}, \quad \forall~~j\in\left\{0\right\}\cup[k],
\end{gather*}
and let $f = f(t, \theta, \nu, z)$ be a global solution to \eqref{D-0}-\eqref{D-1}. Then, there exist constants $\Lambda_{1,k} = \Lambda_{1,k}(T, C, \kappa, f^{\mathrm{\mathrm{in}}})$ and $\Lambda_{2,k} = \Lambda_{2,k}(T, C, \kappa, f^{\mathrm{\mathrm{in}}})$ such that
\begin{align*} 
		\|f\|_{W^{k,p}_{\theta, \nu, z}} \leq \|f^{\mathrm{\mathrm{in}}}\|_{L^p_{\theta, \nu, z}} e^{\Lambda_{1,k}t} +\frac{\Lambda_{2,k}}{\Lambda_{1,k}} (e^{\Lambda_{1,k}t} -1), \quad \forall~t\in[0, T].
\end{align*}
\end{theorem}
\begin{proof}
By Lemma \ref{L5.3}, we have
\[M_{i,0,\infty}<\infty,\quad\forall~i=1,\cdots,k.\]
Then, it follows Lemma \ref{L5.2} and the fact of
\begin{align}\label{p_indep}
	M_{i,0,p}=M_{i,0,\infty},\quad\forall~i\ge1,\quad1<p<\infty,
\end{align}
that there exist nonnegative constants $A_{i,0,p}$ and $B_{i,0,p}$ for all $i=1,\cdots,k$ and $1<p<\infty$ such that
\begin{align}\label{T5.1P1}
	\partial_t\big\|\partial_z^if(t)\big\|_{L^p_{\theta,\nu,z}}\le A_{i,0,p}\big\|\partial_z^if(t)\big\|_{L^p_{\theta,\nu,z}}+B_{i,0,p},\quad\forall~t\in[0,T].
\end{align}
On the other hand, recall that we derived
\begin{align}\label{T5.1P2}
	\partial_t\big\|f(t)\big\|_{L^p_{\theta,\nu,z}}\le C\kappa\big\|f(t)\big\|_{L^p_{\theta,\nu,z}},\quad\forall~t\ge0,\quad1<p\le\infty
\end{align}
in the proof of Proposition \ref{P5.1}. With \eqref{T5.1P1}, \eqref{T5.1P2} and H\"older's inequality, we conclude that
\begin{align*}
		\partial_t\|f\|_{W^{k,p}_{\theta, \nu, z}}^p &\leq pC\kappa \| f\|_{L^p_{\theta, \nu, z}}^p +p\sum_{i=1}^k A_{i,0,p} \|\partial_z^i f\|_{L^p_{\theta, \nu, z}}^p +p\sum_{i=1}^k B_{i,0,p} \|\partial_z^i f\|_{L^p_{\theta, \nu, z}}^{p-1} \\
		&\leq p\max\Big\{ C\kappa, \max_{1\leq i\leq k} A_{i, 0,p} \Big\} \|f\|_{W^{k,p}_{\theta, \nu, z}}^p +p\|f\|_{W^{k,p}_{\theta, \nu, z}}^{p-1} \bigg( \sum_{i=1}^k B_{i, 0,p}^p \bigg)^{1/p} \\
		&\leq p\max\Big\{ C\kappa, \max_{1\leq i\leq k} A_{i, 0,p} \Big\} \|f\|_{W^{k,p}_{\theta, \nu, z}}^p +p\|f\|_{W^{k,p}_{\theta, \nu, z}}^{p-1} \bigg( \sum_{i=1}^k B_{i, 0,p} \bigg).
\end{align*}
We set
\[\Lambda_{1,k}:=\max\Big\{ C\kappa, \max_{1\leq i\leq k} A_{i, 0,p} \Big\}\quad\mbox{and}\quad\Lambda_{2,k}:=\sum_{i=1}^k B_{i, 0,p} ,\]
and note that $\Lambda_{1,k}$ and $\Lambda_{2,k}$ do not depend on $p$ by \eqref{p_indep}. By Gr\"onwall's Lemma, we can get the desired estimate for $p<\infty$. For the case of $p=\infty$, we take the limit of $p\to\infty$ to obtain the desired estimate.
\end{proof}
\section{Numerical results} \label{sec:6}
\setcounter{equation}{0}
In this section, we present several numerical tests on the uncertainty propagation for the Winfree model \eqref{A-6}.  For all numerical tests, we use the recently developed class of Monte Carlo stochastic Galerkin (sG) particle methods for kinetic equations. For interested readers,  we refer to \cite{CCP,CZ,PZ} for a brief introduction.

\subsection{Test 1. Consistency of the mean-field limit in the uncertain case}
A sG reformulation of the kinetic Winfree model is based on the generalized polynomial chaos expansion of the solution of the kinetic Winfree model $f(t, \theta, z)$ of the form:
\begin{align}\label{eq:kinetic_SG}
f(t, \theta, \nu, z) \approx f^M(t, \theta, \nu, z) = \sum_{k=0}^M \hat{f}_k(t, \theta, \nu) \Phi_k(z),
\end{align}
where $\{\Phi_k\}_{k=0}^M$ is a set of orthonormal polynomials such that
\[
\mathbb E[\Phi_h(z)\Phi_k(z)] = \int_{[1,+\infty)}\Phi_h(z)\Phi_k(z)\rho(z)dz = \delta_{hk}.
\]
Here $\delta_{hk}$ is the Kronecker delta function. In \eqref{eq:kinetic_SG}, we indicated with $\hat{f}_k(\theta,t)$ the projection
\begin{align}
\label{eq:fk}
\hat{f}_k(t, \theta, \nu) = \int_{[1,+\infty)}f(t, \theta, \nu, z)\Phi_k(z)\rho(z)dz.
\end{align}
We substitute the approximation $f^M$ into the kinetic Winfree model with uncertainties, and we obtain
\[
\partial_t f^M  + \partial_\theta \Big(f^M\mathcal L[f^M] \Big) = 0,
\]
where $\mathcal L[\cdot]$ has been defined in \eqref{A-7}. We represent $f^M$ by means of the above expansion and multiply the equation by a polynomial $\Phi_h(z)$ of the orthonormal basis, and we yield the following coupled system of $M+1$ deterministic partial differential equations determining the evolution of the projections $\{\hat f_h\} $ of $f$ onto the polynomial space:
\begin{align}
\label{eq:evo_fh}
\partial_t \hat{f}_h(t,\theta) +\partial_\theta \sum_{k=0}^M L_{hk}[f^M](t,\theta) \hat{f}_k(t,\theta), \qquad h = 0,\dots,M,
\end{align}
where
\[
L_{hk}[f^M](t,\theta) :=  \int_{[1,+\infty)} \Big ( \nu - \sigma(t,z)\sin\theta \Big)\Phi_h(z)\Phi_k(z)\rho(z)dz,
\]
with $\sigma(t,z)$ as in \eqref{A-7}. The approximation of the statistical quantities of interest are defined in terms of the introduced projections. In particular, it follows from \eqref{eq:fk}  that we have
\[
\mathbb E[f(t, \theta, \nu, z)] \approx \hat{f}_0(t, \theta, \nu),
\]
and its evolution is approximated by \eqref{eq:evo_fh}. Thanks to the orthonormality of the polynomial basis,  we also have
\[
\textrm{Var}_z[f(t, \theta, \nu, z)] \approx \mathbb E\left[ \left( \sum_{h=0}^M \hat{f}_h(t, \theta, \nu)\Phi_h(z) - \hat{f}_0(t, \theta, \nu)\right)^2 \right] = \sum_{h=0}^ M \hat f_h^2 - \hat f_0^2(t, \theta, \nu).
\]
Amongst the most importance advantages of the gPC-sG methods, is their exponential convergence with respect to the uncertain variable, see \cite{X}. The consistency of averages quantities can be also compared with the reconstruction of the particle density in the presence of uncertainties. \textcolor{black}{The computational cost of the introduced particle sG approach is $\mathcal O(M^2 N^2)$ and can be mitigated through a random batch approach \cite{AP,Jin} as detailed in \cite{CPZ,CZ}. }\newline

\textcolor{black}{From \eqref{A-6} we may observe how the uncertainties impact on the dynamics of each particle which become dependent by all the possible realizations of the considered random variable. For this reason, we}
 introduce the empirical measure associate to the random Winfree model
\[
f^{(N)}(t, \theta, \nu, z) = \left[ \dfrac{1}{N} \sum_{i=1}^N \delta(\theta-\theta_i(t,z){\color{black})}\right]g(\nu).
\]

\begin{figure}
\centering
\hspace{-.4cm}
\includegraphics[width=0.36\textwidth]{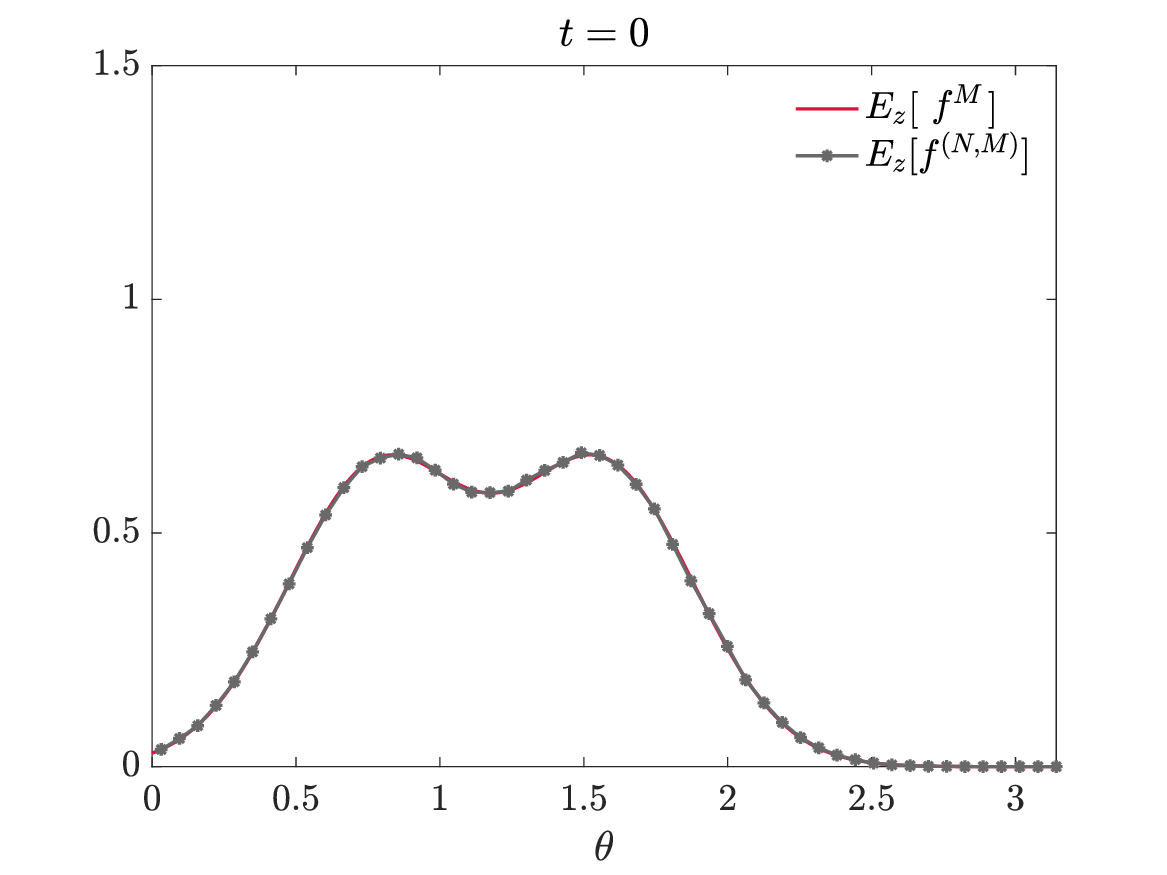}
\hspace{-.75cm}
\includegraphics[width=0.36\textwidth]{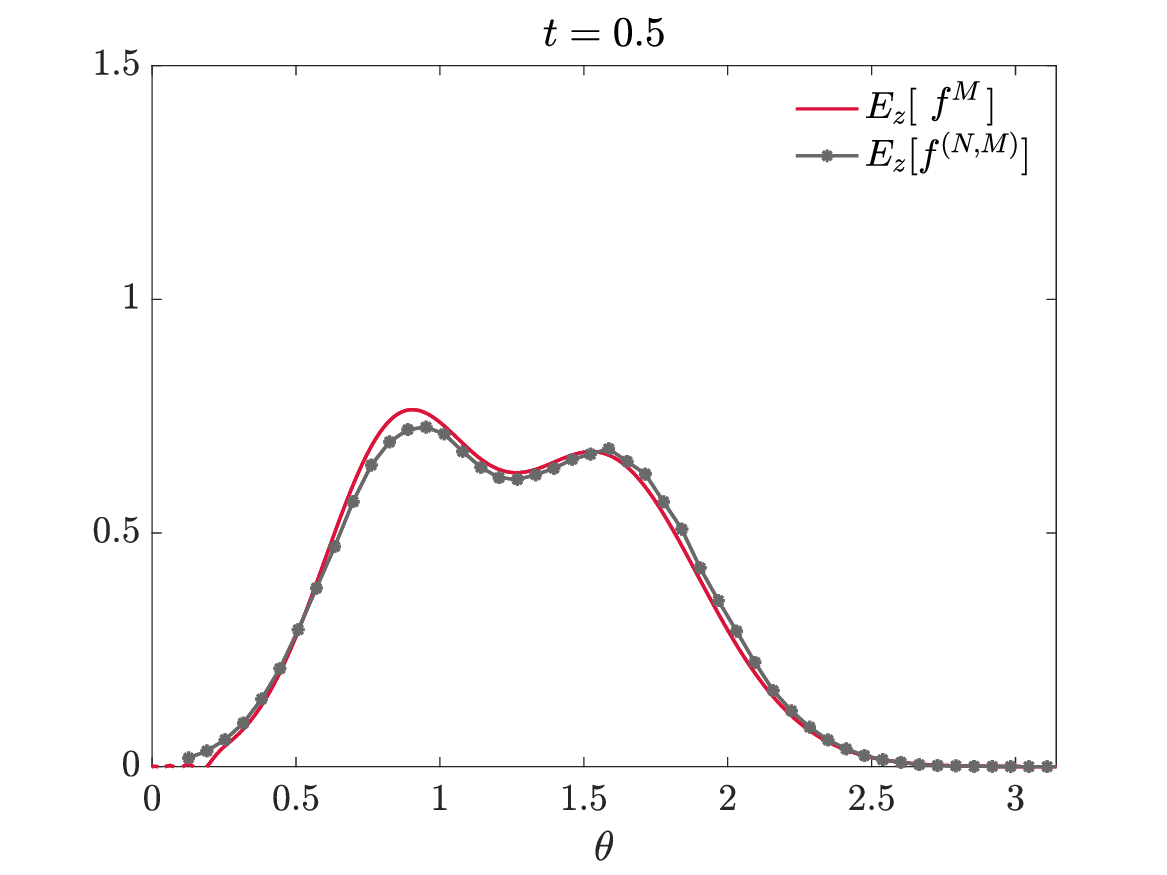}
\hspace{-.75cm}
\includegraphics[width=0.36\textwidth]{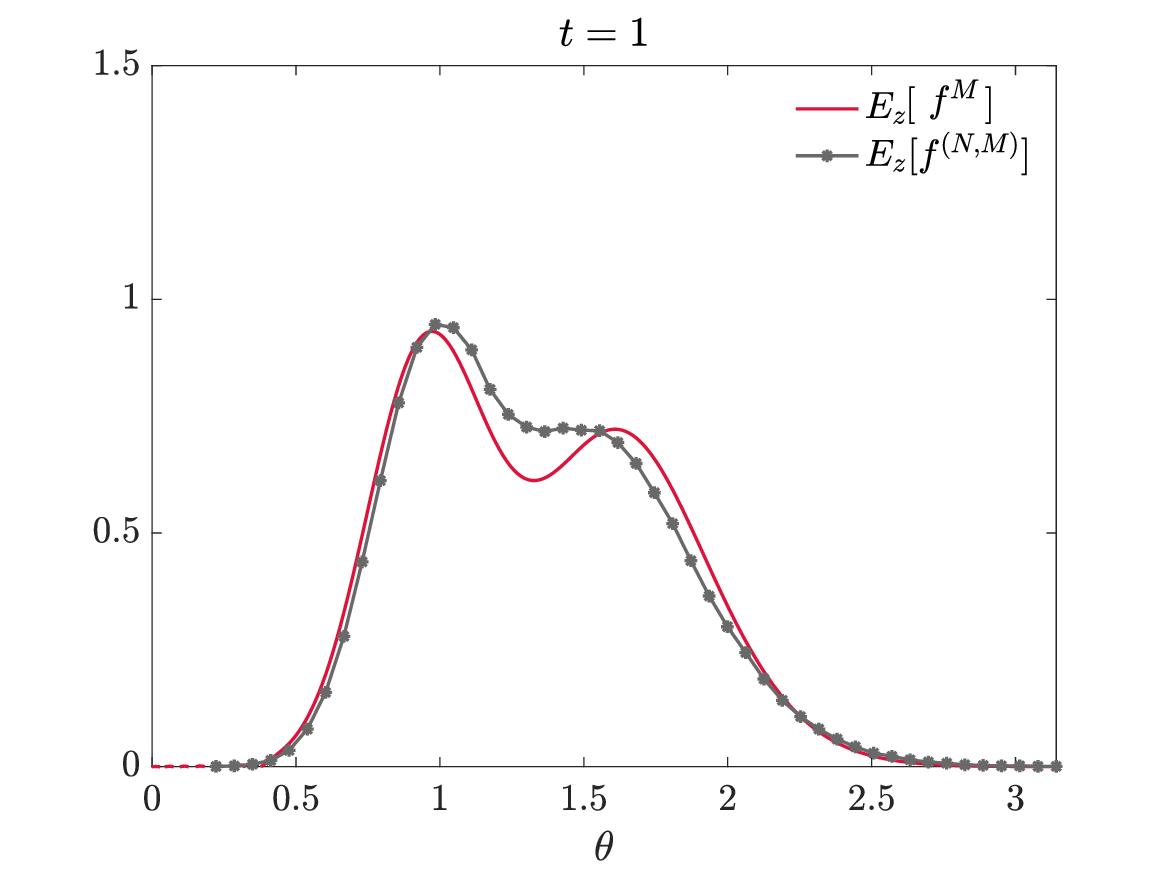}\\
\hspace{-.4cm}
\includegraphics[width=0.36\textwidth]{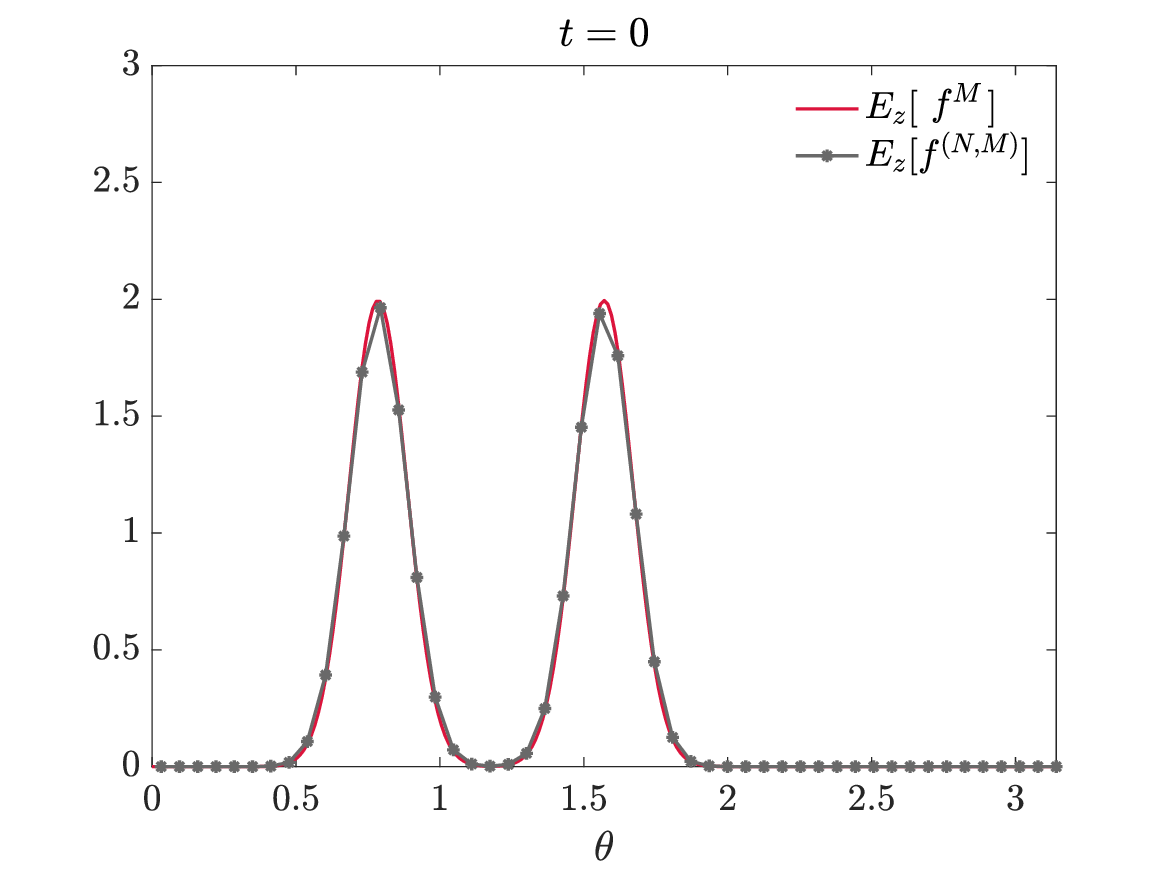}
\hspace{-.75cm}
\includegraphics[width=0.36\textwidth]{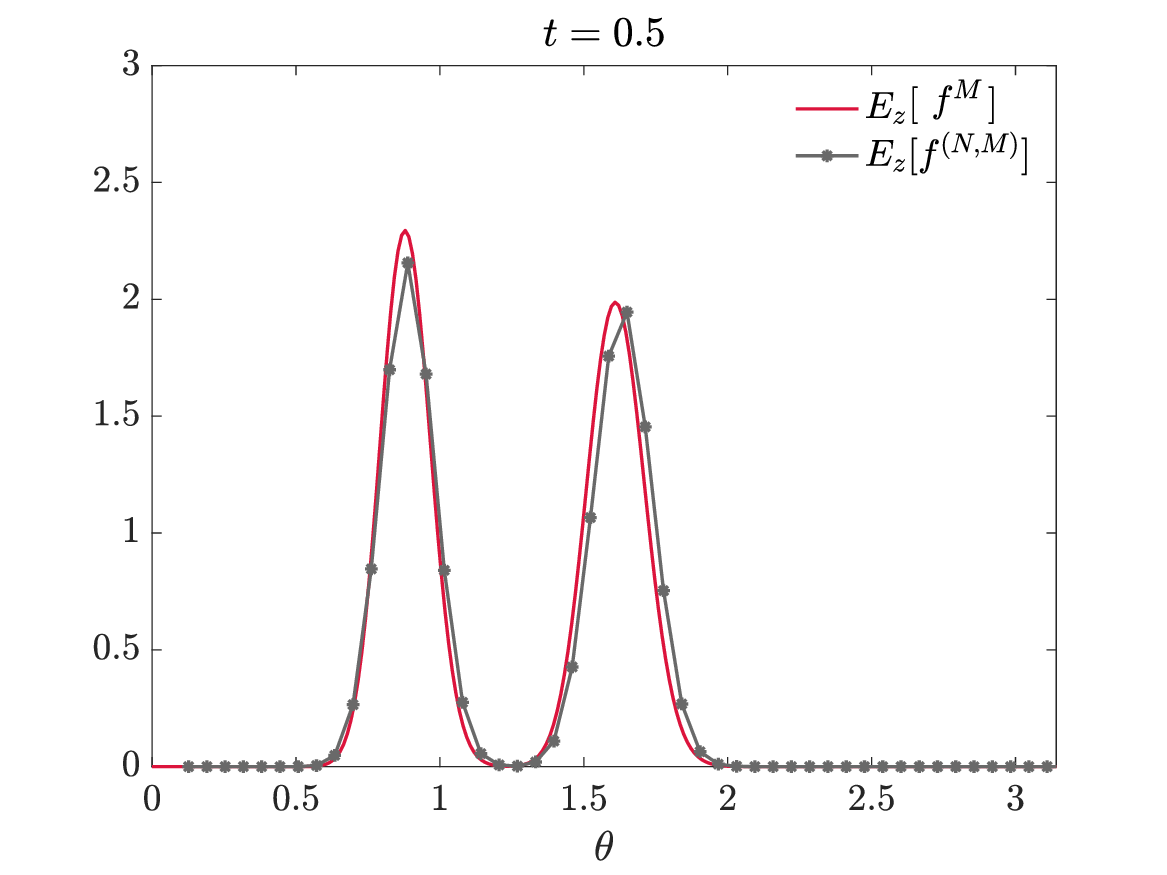}
\hspace{-.75cm}
\includegraphics[width=0.36\textwidth]{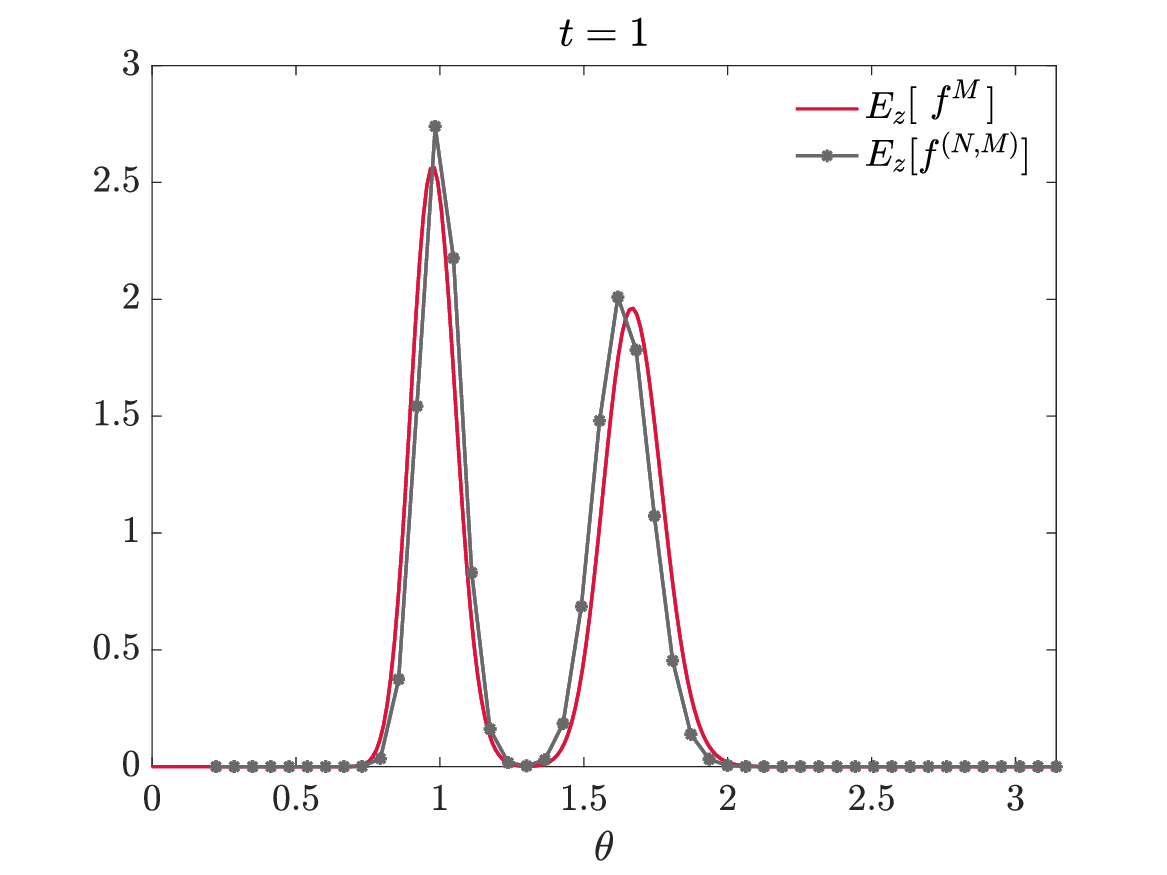}\\
\caption{We compare the evolution of the expected solution of the random kinetic Winfree model $\mathbb E_z[f]$ with its particle approximation $\mathbb E_z[f^{(N,M)}]$. In both cases we considered $z \sim \mathcal U([1,3])$ and we approximate the dynamics over the time interval $[0,1]$ with $\Delta t = \Delta \theta^2$ by introducing a space discretization of the interval $[0,2\pi]$ obtained with $N_\theta = 101$ grid points. We solved \eqref{eq:evo_fh} and the system of sG projections  \eqref{eq:thetaM}  $N = 10^4$. In both cases we fixed $M=2$ and the initial distribution \eqref{eq:f0} with $\bar \theta_1 = \pi/4$, $\bar \theta_2 = \pi/2$ and $\sigma_0^2 = \frac{1}{10}$ (top row) and $\sigma_0^2 = \frac{1}{100}$ (bottom row). }
	\label{fig:MF}
\end{figure}

Following the approach in \cite{CPZ,CZ,MTZ} we may study the convergence of the empirical measure
\[
f^{(N,M)}(t, \theta, \nu, z) = \left[ \dfrac{1}{N} \sum_{i=1}^N \delta(\theta-\theta_i^M(t,z){\color{black})}\right] g(\nu),
\]
to $f^{(N)}(t, \theta, \nu)$ as $M\to +\infty$ for all $t\ge0$, and $\theta_i^M(t,z)$ is a solution to
\begin{align}
\label{eq:thetsG}
\dfrac{d}{dt} \theta_i^M(t,z) = \nu_i + \dfrac{\kappa}{N} \sum_{j=1}^N a(z)(1+\cos\theta_j^M(t,z))^z \sin\theta_i^M,\quad\forall~ i \in [N],
\end{align}
where
\[ \theta_i^M = \sum_{k=0}^M \hat{\theta}_{i,k}\Phi_k(z). \]
The sG approximation of \eqref{eq:thetsG} is therefore given for all $h=0,\dots,M$ by
\begin{align}
\label{eq:thetaM}
\dfrac{d}{dt} \hat{\theta}_{i,h}(t) = \nu_i \delta_{0h} + \dfrac{\kappa}{N} \sum_{i=1}^N \sum_{k=0}^M E_{hk}(\theta_j^M, \theta_i^M),
\end{align}
where we have
\[
E_{hk}(\theta_j^M,\theta_i^M) =\int_{[1,+\infty)}a(z)(1+\cos\theta_j^M(t,z))^z \sin\left( \sum_{k=0}^M\hat{\theta}_{i,k}\Phi_k(z)\right) \Phi_h(z)\rho(z)dz.
\]
In the following, we consider the case $\nu \equiv 0$ and the initial distribution
\begin{align}
\label{eq:f0}
f_0(\theta) = \dfrac{1}{2\sqrt{2\pi \sigma_0^2}}\exp\left\{-\dfrac{|\theta-\bar\theta_1|^2}{2\sigma_0^2}\right\} +
\dfrac{1}{2\sqrt{2\pi \sigma_0^2}}\exp\left\{-\dfrac{|\theta-\bar\theta_2|^2}{2\sigma_0^2}\right\},
\end{align}
with
\[ \bar \theta_1 = \frac{\pi}{4}, \quad \bar \theta_2 = \frac{\pi}{2}. \]
In Figure \ref{fig:MF}, we represent the evolution of the expected value $\mathbb E_z[f^M(t, \theta, z)]$ where the initial distribution is chosen according to \eqref{eq:f0} and $\sigma_0^2 = \frac{1}{10}$ (top row), or $\sigma_0^2 = \frac{1}{100}$ (bottom row). We considered an uncertain exponent $z = h+2$, with $h \sim \mathcal U([-1,1])$ and the polynomial basis is therefore considered as Legendre polynomials, see \cite{X}.  We solve \eqref{eq:evo_fh} over a computational grid obtained from a homogeneous discretization of $[0,2\pi]$ with $N_\theta = 101$ grid points. The time discretization of the interval $[0,T]$ has been fixed as $\Delta t = \Delta \theta^2$, being $\Delta \theta$ the mesh size. Over the same interval we reconstruct the expected density of the interacting system $\mathbb E_z[f^{(N,M)}(t, \theta, z)]$ with $N  = 10^4$ \textcolor{black}{by means of a reconstruction of all the realizations in terms of the uncertainty variable $\{\theta_i^M(t,z)\}_{i=1}^N$ through a Gauss-Legendre scheme. In particular, the expectation is approximated as
\[
\mathbb E[f](t,\theta,\nu) \approx \dfrac{1}{N} \sum_{i=1}^N \int_{[0,\infty)}S_{\Delta \theta}(\theta-\theta_i^M(t,z))\rho(z)dz,
\]
being $S_{\Delta \theta}(\cdot)\ge0$ a smoothing function such that
\[
\Delta \theta \int_{0}^{2\pi} S_{\Delta \theta}(\theta)d\theta = 1.
\]
In the presented tests, we considered $S_{\Delta \theta}(\theta) = \chi(|\theta|\le \Delta \theta/2)/\Delta \theta$, where $\chi(\cdot)$ is the indicator function, which corresponds to the standard histogram reconstruction, see \cite{CPZ,PTZ}.} We may observe that the kinetic density correctly approximates the particle density in both the considered regimes.
\textcolor{black}{ In Figure \ref{fig:bands}, we present the obtained expected profiles together with the confidence bands obtained by the estimated variance.
\begin{figure}
\includegraphics[scale = 0.35]{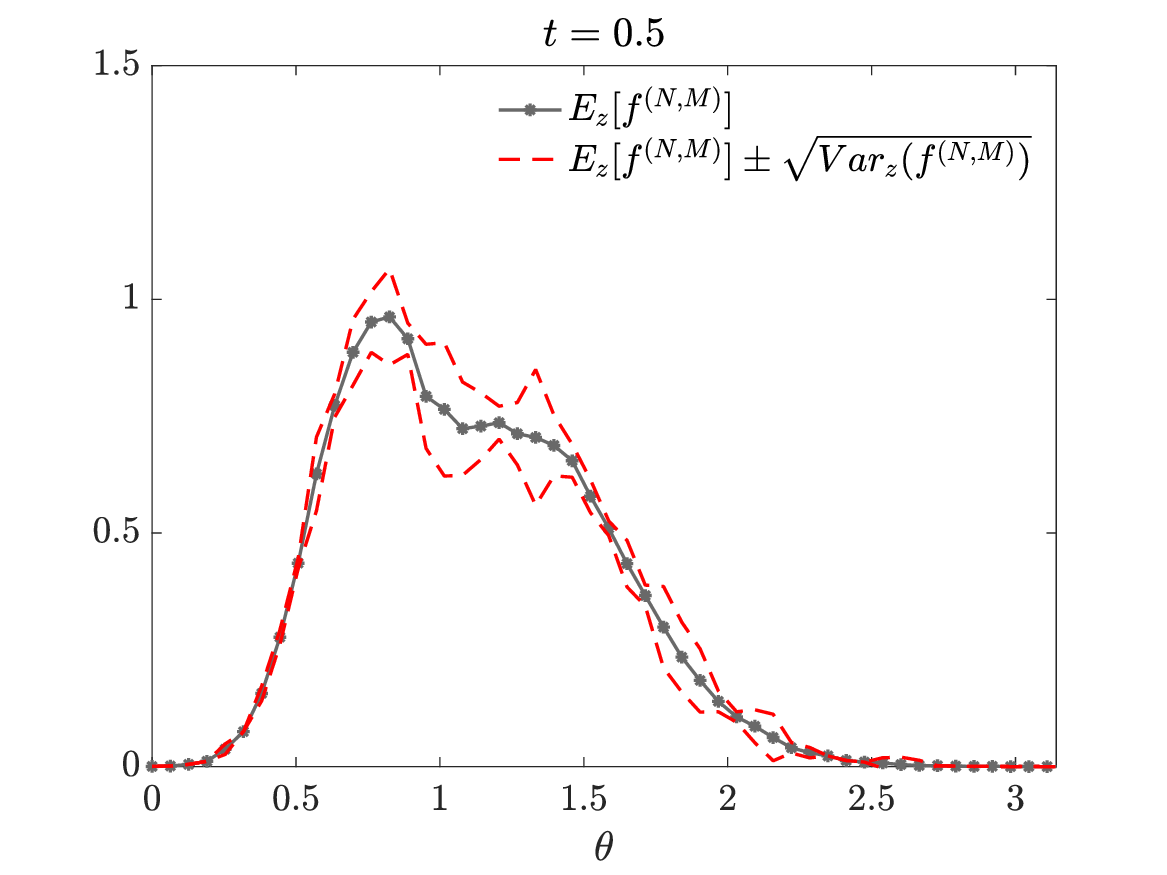}
\includegraphics[scale = 0.35]{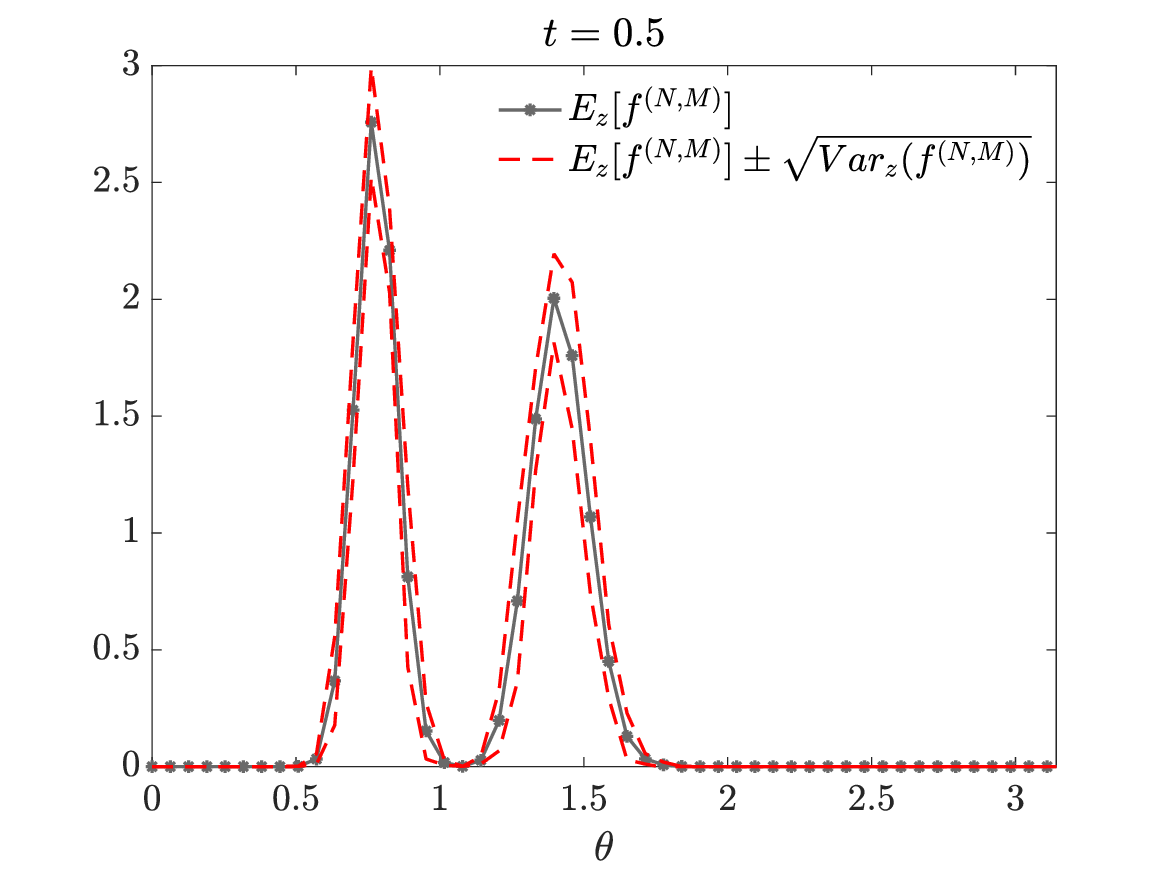}
\caption{\textcolor{black}{We depict $\mathbb E_z[f^{(N,M)}]$ together with confidence bands based on the approximation of the variance $ Var_z(f^{(N,M)})$ at time $t = 0.5$, $N = 10^4$, $M = 2$.  We consider as before $z \sim \mathcal U([1,3])$, $\sigma^2_0 = \frac{1}{10}$ (left) and $\sigma^2_0 = \frac{1}{100}$ (right). }}
\label{fig:bands}
\end{figure}}


\subsection{Regularity propagation}

We evaluate the propagation of $L^2$-regularity in the random field by means of the following error
\begin{align}
\label{eq:error_L2}
\textrm{Error} = \| \langle \varphi,f^{(N)} \rangle -  \langle \varphi,f^{(N,M)}\rangle \|_{L^2([1,+\infty))},
\end{align}
where $\varphi(\cdot):[0,2\pi]\to \mathbb R_+$ a test function. In the following we concentrate on the case $\varphi(\theta) = \theta^2$ which defines the uncertain temperature of the system
\begin{align}
\label{eq:Temp}
\mathcal T(t,z) = \int_{\mathbb R\times[0,2\pi]}(\theta-u)^2 f(t, \theta, \nu, z) g(\nu)d\theta d\nu,
\end{align}
being $g(\cdot)$ and $u$ the distribution of the natural frequencies and the expectation of $\theta$, respectively. In the case of homogeneous frequencies, we consider $N=10^4$ particles and we compute the evolution of \eqref{eq:Temp} up to time $T=1$ with a time step $\Delta t = 10^{-2}$. We considered both the uniform case $z = h+2$, with $h \sim \mathcal U([-1,1])$ and the Gaussian case where $z = h^2 + \frac{3}{2}$, with $h\sim \mathcal N(0,1)$, such that in both cases we get $\mathbb E[z] = 2$ but in the Gaussian case the support of the random variable is not bounded.  It is worth mentioning that in the case of uniformly distributed uncertainty we consider a Legendre polynomial expansion whereas, in the Gaussian case, we consider a Hermite polynomial expansion. In Figure \ref{fig:error}, we report the convergence in the random field based on a reference temperature $\mathcal T^{\textrm{ref}}$ at time $T = 1$ computed with an expansion of order $M = 25$ and $N = 10^4$ particles. We reported the error for several $M=0, \dots, 9$ for two choices of coupling strengths $\kappa = \frac{1}{10}$ and $\kappa = 1$. We may observe how, for increasing $\kappa$, the method has spectral accuracy in the uniform case as discussed in the previous sections. Therefore, for an increasing coupling strength we observe spectral accuracy and therefore, spectral accuracy in terms of the death state. On the other hand, in the case of unbounded support, corresponding to the Gaussian case, we do not expect the same regularity in the random field since the error saturates at $M = 6$ for the same coupling strength.

\begin{figure}
\centering
\includegraphics[width=0.45\textwidth]{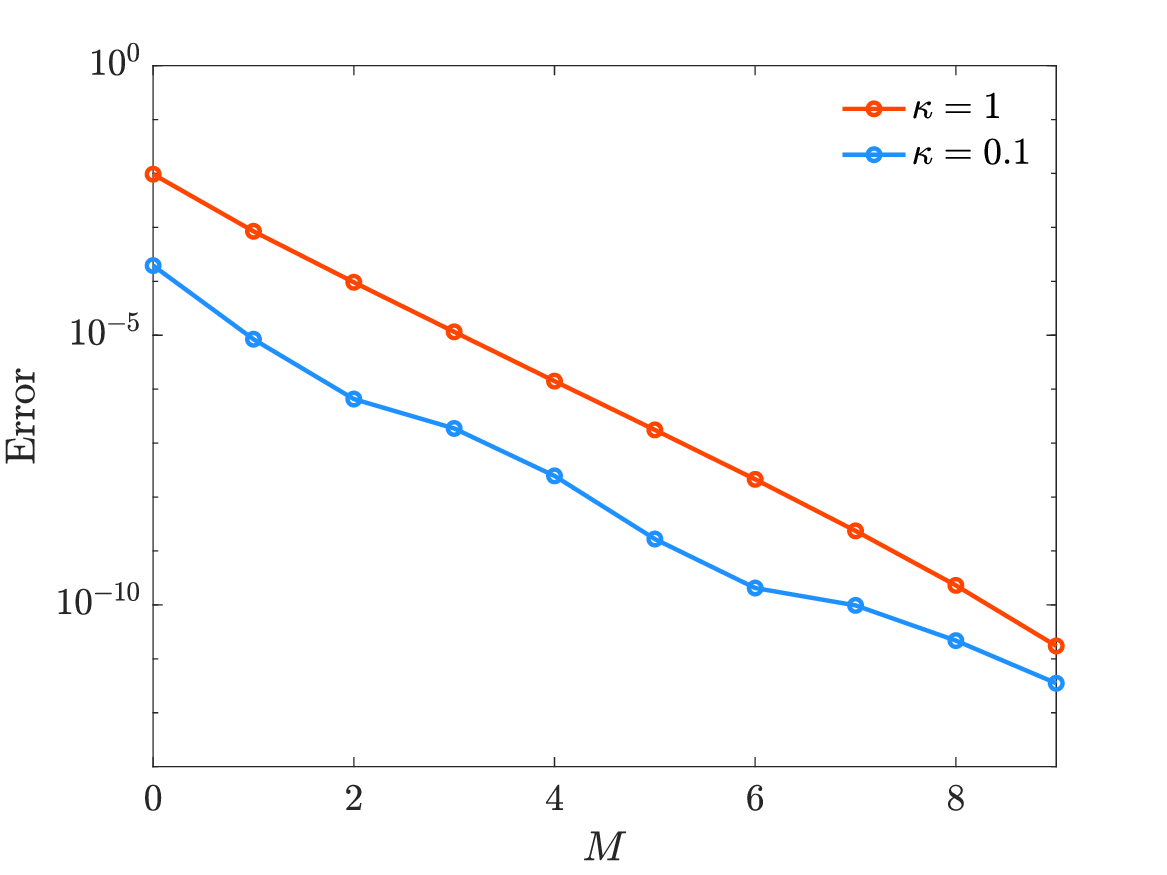}
\includegraphics[width=0.45\textwidth]{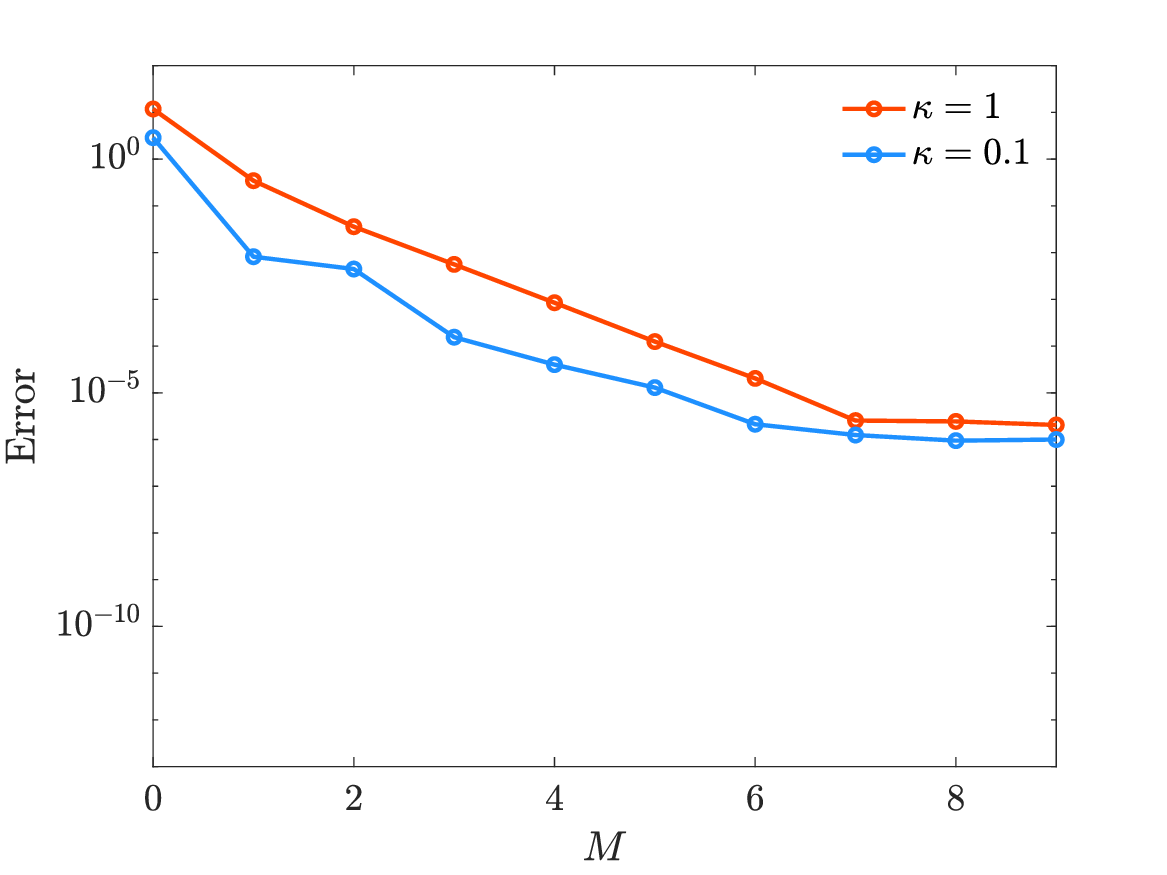}  \\
\caption{Convergence of the uncertain temperature \eqref{eq:Temp} in terms of the error in \eqref{eq:error_L2} based on a reference temperature $\mathcal T^{\textrm{ref}}$ at time $T = 1$ computed at the particle level with $N = 10^4$ particles and $M = 25$. We considered a uniformly distributed uncertain exponent $z = h+2$, $h\sim \mathcal U([-1,1])$ (left) and a Gaussian exponent $z = h^2 + \frac{3}{2}$, $h \sim \mathcal N(0,1)$. In both cases we considered two coupling strengths $\kappa = \frac{1}{10}$ and $\kappa = 1$.  }
\label{fig:error}
\end{figure}

\section{Conclusion} \label{sec:7}
\setcounter{equation}{0}
In this paper, we have studied measure-valued death state and local sensitivity analysis for the random Winfree and its kinetic mean-field models with uncertain high-order couplings.
For the proposed models, we study several sufficient frameworks leading to the oscillator death in which rotation numbers tend to zero asymptotically. In particular, for the random Winfree model, we show the existence of bounded trapping set so that the rotation number of each oscillator tends to zero asymptotically. In contrast, for the random kinetic Winfree model, we also present sufficient framework leading to the oscillator death. For the propagation of regularity in one-dimensional random space, we show that the $H_z^k$-regularity of the solution process can be bounded by that of initial data in any finite-time interval. We also provide several numerical tests and compare them with the analytical results. As noted in this paper, we only focus on the oscillator death which emerges as one of possible asymptotic patterns. However, there are other possible asymptotic patterns such as incoherent state, phase-locked state and chimera state etc. Thus, it will be very interesting to investigate the basin of attractions to these asymptotic patterns. Moreover, extension of results to the pulse-coupled models will be also a challenging problem. We leave these issues for a future work.


\newpage

\appendix

\section{Proof of Proposition \ref{P5.1}}\label{Proof_P5.1}
\noindent We use the method of induction.  For an integer $l \geq 1$, note that
\[ \partial_\theta^l L[f](t, \theta, z) = -\kappa \big(\partial_\theta^l \sin\theta\big) \int_{\bbt\times\bbr} I(\theta_*, z) f(t, \theta_*, \nu_*, z) g(\nu_*)  d\theta_* d\nu_*. \]
This leads to
	\begin{align}\label{D-3-0}
		\|\partial_\theta^l L[f](t)\|_{L^\infty_{\theta,z}} \leq C\kappa,\quad \forall~t \geq 0.
	\end{align}

\noindent $\bullet$~(Initial step):~First we use \eqref{D-0} to see
\begin{align}
\begin{aligned} \label{D-3-1}
\partial_t (|f|^p) &= \partial_t \Big ( (|f|^2)^{\frac{p}{2}} \Big) =   \partial_t \Big ( (f^2)^{\frac{p}{2}} \Big) = \frac{p}{2}  |f|^{p-2} (2f \partial_t f)  = p |f|^{p-2} f \partial_t f   \\
&= -p |f|^{p-2} f \partial_\theta(f L[f]).
\end{aligned}
\end{align}
Similarly, we also have
\begin{align} \label{D-3-2}
\partial_\theta (|f|^p) = p |f|^{p-2} f \partial_\theta f.
\end{align}
We differentiate $\| f(t)\|_{L^p_{\theta, \nu, z}}^p$ with respect to $t$ using \eqref{D-3-1}, \eqref{D-3-2} and the integration by parts to find
	\begin{align}
	\begin{aligned} \label{K-1-1}
			\partial_t \| f(t)\|_{L^p_{\theta, \nu,z}}^p &=  \int_{\bbt\times\bbr\times[1,\infty)} \partial_t ( |f|^p) \rho(z)g(\nu) d\theta d\nu dz \\
			&= -\int_{\bbt\times\bbr\times[1,\infty)} p|f|^{p-2} f \partial_\theta \left( fL[f] \right)\rho(z) g(\nu) d\theta d\nu dz    \\
			&=  -\int_{\bbt\times\bbr\times[1,\infty)} p|f|^{p-2} f \Big( L[f]\partial_\theta f + f\partial_\theta L[f] \Big)\rho(z) g(\nu)  d\theta d\nu dz \\
			&=  -\int_{\bbt\times\bbr\times[1,\infty)}\left( L[f] \partial_\theta |f|^p +p|f|^p \partial_\theta L[f] \right)\rho(z) g(\nu)  d\theta d\nu dz \\
			&=(1-p) \int_{\bbt\times[1,\infty)} \rho(z) \left( \int_\bbr g(\nu) |f|^p d\nu \right) \partial_\theta L[f] dzd\theta \\
			&\leq p\|\partial_\theta L[f](\cdot,t)\|_{L^\infty_{\theta, z}} \cdot \left\| \int_\bbr g |f|^p  d\nu \right\|_{L^1_{\theta,z}} \leq C p\kappa \| f(t)\|_{L^p_{\theta, \nu, z}}^p,
		\end{aligned}
	\end{align}
i.e., one has
\[ \partial_t \| f(t)\|_{L^p_{\theta, \nu,z}} \leq C \kappa  \| f(t)\|_{L^p_{\theta, \nu, z}}, \quad \forall~t > 0. \]
This yields the estimate \eqref{D-2} for $k = 0$:
\[ \| f(t)\|_{L^p_{\theta, \nu,z}} \leq e^{C  \kappa t} \| f^{in} \|_{L^p_{\theta, \nu,z}}, \quad  \forall~t \geq 0. \]

\noindent $\bullet$~(Induction step):~Suppose that the estimate \eqref{D-2} holds for $k=0, \cdots, l-1$, and note that $ (\partial_\theta^l f)$ satisfies
\begin{align} \label{D-3-3}
\partial_t  (\partial_\theta^l f) + \partial_\theta^{l +1} (f L[f])) = 0.
\end{align}
Then, we use \eqref{D-3-3} and the same argument in \eqref{D-3-1} to find
\begin{align} \label{D-3-4}
\partial_t ( |\partial_\theta^l f|^p ) = -p | \partial_\theta^l f|^{p-2} (\partial_\theta^l f) \partial_\theta^{l +1} (f L[f])).
\end{align}
Now we use \eqref{D-3-4}, the same argument as in \eqref{K-1-1}  and the binomial theorem to find
\begin{align}
\begin{aligned} \label{K-1-5}
&\partial_t \|\partial_\theta^l f(t)\|_{L^p_{\theta, \nu, z}}^p  \\
& \hspace{1cm} = \int_{\bbt \times \bbr \times [1,\infty)} \partial_t |\partial_\theta^l f(t)|^{p} \rho gdz d\nu d\theta \\
& \hspace{1cm} = -\int_{\bbt\times\bbr\times[1,\infty)} p|\partial_\theta^l f|^{p-2} (\partial_\theta^l f)  \partial_\theta^{l +1} \left( fL[f] \right)\rho gdz d\nu d\theta \\
& \hspace{1cm} = -\int_{\bbt\times\bbr\times[1,\infty)} p|\partial_\theta^l f|^{p-2} (\partial_\theta^l f) \sum_{i=0}^{l +1} \binom{l+1}{i} \left( \partial_\theta^i f\cdot\partial_\theta^{l +1-i} L[f] \right)\rho gdz d\nu d\theta \\
& \hspace{1cm}= -\int_{\bbt\times\bbr\times[1,\infty)} p|\partial_\theta^l f|^{p-2} (\partial_\theta^l f) \left( (\partial_\theta^{l+1} f) L[f] + (l+ 1) (\partial_\theta^l f)(\partial_\theta L[f]) \right) \rho gdz d\nu d\theta \\
&  \hspace{1.4cm} -\int_{\bbt\times\bbr\times[1,\infty)} p|\partial_\theta^l f|^{p-2} (\partial_\theta^l f ) \sum_{i=0}^{l-1} \binom{l +1}{i} \left( (\partial_\theta^i f) (\partial_\theta^{l +1-i} L[f] ) \right)\rho gdz d\nu d\theta \\
&  \hspace{1cm}=: \mathcal J_{11} +\mathcal J_{12},
\end{aligned}
\end{align}
where we use the following relation:
\begin{align*}
\begin{aligned}
&\sum_{i=0}^{l +1} \binom{l+1}{i} (\partial_\theta^i f ) \cdot (\partial_\theta^{l +1-i} L[f]) \\
&\hspace{1cm} = ( \partial_\theta^{l+1} f)( L[f]) + (l+1)( \partial_\theta^l f) (\partial_\theta L[f] ) + \sum_{i=0}^{l-1} \binom{l +1}{i} \left( \partial_\theta^i f \right) \left(\partial_\theta^{l +1-i} L[f] \right).
\end{aligned}
\end{align*}
Below, we estimate the term ${\mathcal J}_{1i}$ one by one. \newline

	\vspace{0.2cm}

	\noindent$\diamond$ (Estimate on $\mathcal J_{11}$):~Similar to \eqref{K-1-1}, we use integration by parts and \eqref{D-3-0} to obtain
	\begin{align}
	\begin{aligned} \label{D-4}
		\mathcal J_{11} &= -\int_{\bbt\times\bbr\times[1,\infty)}\Big(\partial_\theta\big( |\partial_\theta^l f|^p\big)\cdot L[f]+p(l+1)|\partial_\theta^l f|^p\partial_\theta L[f]\Big)\rho gdz d\nu d\theta\\
		&= (1-p(l+1)) \int_{\bbt\times[1,\infty)}\bigg( \int_\bbr |\partial_\theta^l f|^pgd\nu \bigg) \partial_\theta L[f]\rho dz d\theta \\
		&\leq p(l +1) \big\| \partial_\theta L[f](t)\big\|_{L^\infty_{\theta, z}} \cdot \bigg\| \int_\bbr |\partial_\theta^l f|^pgd\nu \bigg\|_{L^1_{\theta,z}} \\
		& \leq p(l +1) C\kappa \|\partial_\theta^l f(t)\|_{L^p_{\theta, z,\nu}}^p.
	\end{aligned}
	\end{align}
	\noindent$\diamond$ (Estimate on $\mathcal J_{12}$): We use the relation
	\begin{align} \label{K-1-3}
		\binom{l+1}{i} \leq 2^l, \quad\mbox{for all}~~i=0, \cdots, l+1,
	\end{align}
	\eqref{D-3-0}, and H\"older's inequality to find
	\begin{align}
	\begin{aligned} \label{D-5}
		\mathcal J_{12} &\leq 2^l p \sum_{i=0}^{l-1} \int_{\bbt\times\bbr\times[1,\infty)} \big|\partial_\theta^l f\big|^{p-1} \cdot \big| \partial_\theta^i f \big| \cdot \big|\partial_\theta^{l+1-i} L[f] \big|\rho gdz d\nu d\theta \\
		&= 2^l p \sum_{i=0}^{l-1} \int_{\bbt\times[1,\infty)}\bigg( \int_\bbr \big|\partial_\theta^l f\big|^{p-1} \big| \partial_\theta^i f \big|gd\nu \bigg) \big|\partial_\theta^{l+1-i} L[f] \big|\rho dz d\theta \\
		&\leq 2^l pC\kappa \sum_{i=0}^{l-1} \left\| \big|\partial_\theta^l f\big|^{p-1} \big| \partial_\theta^i f \big| \right\|_{L^1_{\theta, \nu, z}} \\
		& \leq 2^l pC\kappa \big\|\partial_\theta^l f(t)\big\|_{L^p_{\theta, \nu, z}}^{p-1} \sum_{i=0}^{l-1} \big\| \partial_\theta^i f(t) \big\|_{L^p_{\theta, \nu, z}}.
	\end{aligned}
	\end{align}
	We substitute all the estimates \eqref{D-4}, \eqref{D-5} on $\mathcal J_{11}$ and $\mathcal J_{12}$ into \eqref{K-1-5} and use the induction hypothesis to obtain
	\begin{align*}
		\partial_t \|\partial_\theta^l f(t)\|_{L^p_{\theta, \nu, z}} &\leq (l +1) C\kappa \|\partial_\theta^l f(t)\|_{L^p_{\theta, \nu, z}} +2^lC\kappa \sum_{i=0}^{l -1} \big\| \partial_\theta^i f(t) \big\|_{L^p_{\theta, \nu, z}} \\
		&\leq (l +1) C\kappa \|\partial_\theta^l f(t)\|_{L^p_{\theta, \nu, z}} +2^l C\kappa \sum_{i=0}^{l-1} \Lambda_{0,i} e^{(i+1)C\kappa t} \sum_{j=0}^i \big\| \partial_\theta^j f^{\mathrm{\mathrm{in}}} \big\|_{L^p_{\theta, \nu, z}} \\
		&\leq (l +1) C\kappa \|\partial_\theta^l f(t)\|_{L^p_{\theta, \nu, z}} +2^lC\kappa e^{lC\kappa t} \bigg(\sum_{i=0}^{l-1} \Lambda_{0,i} \bigg) \bigg( \sum_{j=0}^{l-1}  \big\| \partial_\theta^j f^{\mathrm{\mathrm{in}}} \big\|_{L^p_{\theta, \nu, z}} \bigg) \\
		&= (l +1) C\kappa \|\partial_\theta^l f(t)\|_{L^p_{\theta, \nu, z}} +C\kappa \Lambda_{0,l} e^{lC\kappa t} \bigg( \sum_{j=0}^{l-1}  \big\| \partial_\theta^j f^{\mathrm{\mathrm{in}}} \big\|_{L^p_{\theta, \nu, z}} \bigg).
	\end{align*}
By  Gr\"onwall's inequality, we have
	\begin{align*}
		\|\partial_\theta^l f(t)\|_{L^p_{\theta, \nu, z}} \leq \bigg( \|\partial_\theta^l f^{\mathrm{\mathrm{in}}}\|_{L^p_{\theta, \nu,z}} +\Lambda_{0,l}  \sum_{j=0}^{l-1}  \big\| \partial_\theta^j f^{\mathrm{\mathrm{in}}} \big\|_{L^p_{\theta,\nu,z}} \bigg) e^{(l +1) C\kappa t} -\bigg( \Lambda_{0,l}  \sum_{j=0}^{l-1}  \big\| \partial_\theta^j f^{\mathrm{\mathrm{in}}} \big\|_{L^p_{\theta,\nu,z}} \bigg) e^{lC\kappa t}.
	\end{align*}
	This shows that the case $k=l $ holds. Therefore, we have the desired estimate \eqref{D-2} for all finite $p>1$. Since there is no $p$-dependency in the coefficient in the right hand side of \eqref{D-2}, we can derive the case of $p=\infty$ by taking the limit of $p\to\infty$ on \eqref{D-2}.

\section{Proof of Lemma \ref{L5.2}}\label{Proof_L5.2}
\noindent We split its proof into two cases:
\[ l = 0 \quad \mbox{and} \quad  l \geq 1. \]

\noindent $\bullet$~Case A ($l = 0$):  We use \eqref{D-0} to see
\begin{align} \label{K-1-6}
\begin{aligned}
& \partial_t \| \partial_z^k f\|_{L^p_{\theta, \nu, z}}^p \\
&= -\int_{\mathbb{T}\times\mathbb{R}\times[1,\infty)} p|\partial_z^k f|^{p-2} \partial_z^k f\cdot\partial_z^k \partial_\theta \left( fL[f] \right)\rho gdzd\nu d\theta \\
&= -\int_{\mathbb{T}\times\mathbb{R}\times[1,\infty)} p|\partial_z^k f|^{p-2} \partial_z^k f \left[ \sum_{i=0}^k \binom{k}{i} \left( \partial_z^i f\cdot \partial_z^{k-i} \partial_\theta L[f] +\partial_z^i \partial_\theta f\cdot\partial_z^{k-i} L[f] \right) \right]\rho gdz d\nu d\theta \\
&= -\int_{\mathbb{T}\times\mathbb{R}\times[1,\infty)} p|\partial_z^k f|^{p-2} \partial_z^k f \left( \partial_z^k f\cdot\partial_\theta L[f] +\partial_z^k \partial_\theta f\cdot L[f]\right)\rho gdz d\nu d\theta \\
&\hspace{0.3cm} -\int_{\mathbb{T}\times\mathbb{R}\times[1,\infty)} p|\partial_z^k f|^{p-2} \partial_z^k f \left[ \sum_{i=0}^{k-1} \binom{k}{i} \left( \partial_z^i f\cdot \partial_z^{k-i} \partial_\theta L[f] +\partial_z^i \partial_\theta f\cdot\partial_z^{k-i} L[f] \right) \right]\rho gdz d\nu d\theta \\
&=: \mathcal J_{21} +\mathcal J_{22}.
\end{aligned}
\end{align}
Below, we estimate the term ${\mathcal J}_{2i}$ one by one. \newline

\vspace{0.2cm}

\noindent$\diamond$ (Estimate of $\mathcal J_{21}$): We use the integration by parts and the relation \eqref{D-3-0} to get
	\begin{align} \label{K-1-7}
		\begin{aligned}
			\mathcal J_{21} &= -\int_{\mathbb{T}\times\mathbb{R}\times[1,\infty)} p |\partial_z^k f|^p\partial_\theta L[f]\cdot\rho gdz d\nu d\theta \\
			&\hspace{.4cm}-\int_{\mathbb{T}\times\mathbb{R}\times[1,\infty)}\partial_\theta\big(|\partial_z^k f|^p\big)\cdot L[f]\cdot\rho g dzd\nu d\theta \\
			&=(1-p)\int_{\mathbb T\times\mathbb R\times[1,\infty)}|\partial_z^kf|^p\partial_\theta L[f]\cdot\rho gdzd\nu d\theta \\
			& = (1-p) \int_{\mathbb{T}\times[1,\infty)}\left( \int_\bbr|\partial_z^k f|^pgd\nu \right) \partial_\theta L[f]\cdot\rho dzd\theta \\
			&\le p \left\|\partial_\theta L[f] \right\|_{L^\infty_{\theta,z}} \cdot \left\| \int_\bbr |\partial_z^k f|^pgd\nu \right\|_{L^1_{\theta,z}} \\
			&\le p C\kappa \| \partial_z^k f\|_{L^p_{\theta, \nu, z}}^p.
		\end{aligned}
	\end{align}

	\vspace{0.2cm}

	\noindent$\diamond$ (Estimate of $\mathcal J_{22}$): It follows from \eqref{K-1-3} and Lemma \ref{L5.1} that
	{\allowdisplaybreaks
	\begin{align}
			\mathcal J_{22} &\leq 2^{k-1} p \sum_{i=0}^{k-1} \int_{\mathbb{T}\times\mathbb{R}\times[1,\infty)}\big| \partial_z^k f \big|^{p-1} \big| \partial_z^i f\cdot\partial_z^{k-i} \partial_\theta L[f] +\partial_z^i\partial_\theta f\cdot \partial_z^{k-i} L[f] \big|\rho gdz d\nu d\theta \nonumber\\
			&\leq 2^{k-1} p \big\| \partial_z^k f \big\|_{L^p_{\theta, \nu, z}}^{p-1} \sum_{i=0}^{k-1} \left( \big\| \partial_z^i f \big\|_{L^\infty_{\theta, \nu, z}} \big\| \partial_z^{k-i} \partial_\theta L[f] \big\|_{L^p_{\theta,\nu,z}} +\big\|\partial_z^i \partial_\theta f \big\|_{L^\infty_{\theta, \nu, z}} \big\| \partial_z^{k-i} L[f] \big\|_{L^p_{\theta, \nu, z}} \right) \nonumber\\
			&\leq 2^k \pi p C\kappa M_{k,0,p} \big\| \partial_z^k f \big\|_{L^p_{\theta,\nu,z}}^{p-1} \sum_{i=0}^{k-1}  2^{k-i} \big\| f \big\|_{W^{k-i, p}_{\theta, \nu,z}} \label{K-1-8}\\
			&\leq 2^{2k+1} \pi p C\kappa M_{k,0,p}\big\| \partial_z^k f \big\|_{L^p_{\theta, \nu,z}}^{p-1} \big\| f \big\|_{W^{k, p}_{\theta, \nu,z}}\nonumber\\
			&\leq 2^{2k+1} \pi p C\kappa M_{k,0,p}\big\| \partial_z^k f \big\|_{L^p_{\theta, \nu,z}}^{p-1} \left( \big\| \partial_z^k f \big\|_{L^p_{\theta, \nu,z}} +2k\pi M_{k,0,p}\right),\nonumber
	\end{align}}
	where we used the following relation in the last inequality:
	\begin{align}
	\begin{aligned} \label{K-1-9}
			\big\| f \big\|_{W^{k, p}_{\theta, \nu, z}}^p &= \big\| \partial_z^k f \big\|_{L^p_{\theta, \nu, z}}^p +\sum_{i=0}^{k-1} \big\| \partial_z^i f \big\|_{L^p_{\theta, \nu, z}}^p\\
			&\le\big\|\partial_z^kf\big\|_{L^p_{\theta,\nu,z}}^p+\sum_{i=0}^{k-1}\big\|\partial_z^if\big\|_{L^\infty_{\theta,\nu,z}}^p\int_{\mathbb T\times\mathbb R\times[1,\infty)}\rho(z)g(\nu)d\theta d\nu dz\\
			&=\big\| \partial_z^k f \big\|_{L^p_{\theta, \nu, z}}^p +2\pi \sum_{i=0}^{k-1} \big\| \partial_z^i f \big\|_{L^\infty_{\theta, \nu, z}}^p\leq \big\| \partial_z^k f \big\|_{L^p_{\theta, \nu, z}}^p +2k\pi M_{k,0,p}^p.
	\end{aligned}
	\end{align}
	Now, we combine \eqref{K-1-6}, \eqref{K-1-7}, and \eqref{K-1-8} to obtain the desired estimate:
	\begin{align*}
		\partial_t \| \partial_z^k f\|_{L^p_{\theta, \nu, z}} \leq C\kappa \big( 1+2^{2k+1} \pi M_{k,0,p} \big) \| \partial_z^k f\|_{L^p_{\theta, \nu, z}}  +2^{2k+2}k\pi^2 C\kappa M_{k,0,p}^2.
	\end{align*}

\vspace{0.2cm}

\noindent $\bullet$~Case B ($l \geq 1$): ~It follows from \eqref{D-0} that
	\begin{align*}
		&\partial_t \|\partial_z^k \partial_\theta^l f\|_{L^p_{\theta, \nu, z}}^p  \\
		&\hspace{0.5cm} = -\int_{\bbt\times\bbr\times[1,\infty)} p|\partial_z^k \partial_\theta^l f|^{p-2} \partial_z^k\partial_\theta^l f\cdot\partial_z^k\partial_\theta^{l+1} \left( fL[f] \right)\cdot\rho gdz d\nu d\theta \\
		& \hspace{0.5cm} = -\int_{\bbt\times\bbr\times[1,\infty)} p|\partial_z^k \partial_\theta^l f|^{p-2} \partial_z^k\partial_\theta^l f \\
		&\hspace{2cm} \times\left[ \sum_{i=0}^k \sum_{j=0}^{l+1} \binom{k}{i} \binom{l+1}{j} \left( \partial_z^i \partial_\theta^j f\cdot\partial_z^{k-i} \partial_\theta^{l+1-j}L[f] \right) \right]\rho g dz d\nu d\theta \\
		&\hspace{0.5cm} = -\int_{\bbt\times\bbr\times[1,\infty)} p|\partial_z^k \partial_\theta^l f|^{p-2} \partial_z^k\partial_\theta^l f \left(\partial_z^k \partial_\theta^{l+1} f\cdot L[f] +(l+1)\partial_z^k \partial_\theta^l f\cdot\partial_\theta L[f] \right)\rho gdz d\nu d\theta \\
		& \hspace{0.5cm} \quad -\int_{\bbt\times\bbr\times[1,\infty)} p|\partial_z^k \partial_\theta^l f|^{p-2} \partial_z^k\partial_\theta^l f \sum_{j=0}^{l-1} \binom{l+1}{j} \left( \partial_z^k \partial_\theta^j f\cdot\partial_\theta^{l+1-j}L[f] \right)\rho gdz d\nu d\theta \\
		&\hspace{0.5cm} \quad -\int_{\bbt\times\bbr\times[1,\infty)} p|\partial_z^k \partial_\theta^l f|^{p-2} \partial_z^k\partial_\theta^l f \\
		&\hspace{2cm} \times\left[ \sum_{i=0}^{k-1} \sum_{j=0}^{l+1} \binom{k}{i} \binom{l+1}{j} \left( \partial_z^i \partial_\theta^j f\cdot\partial_z^{k-i} \partial_\theta^{l+1-j}L[f] \right) \right]\rho gdz d\nu d\theta \\
		&\hspace{0.5cm} =: \mathcal J_{31} +\mathcal J_{32} +\mathcal J_{33}.
	\end{align*}
	In the sequel, we estimate the term ${\mathcal J}_{3i}$ one by one. \newline

	\noindent$\diamond$ Case B.1 (Estimate of $\mathcal J_{31}$): Similar to \eqref{K-1-7}, one has
	\begin{align*}
		\mathcal J_{31} &= \int_{\bbt\times\bbr\times[1,\infty)}\big(1-p(l+1)\big) |\partial_z^k \partial_\theta^l f|^p \partial_\theta L[f]\cdot\rho gdz d\nu d\theta \leq p(l+1) C\kappa \|\partial_z^k \partial_\theta^l f\|_{L^p_{\theta, \nu,z}}^p.
	\end{align*}

	\noindent$\diamond$ Case B.2 (Estimate of $\mathcal J_{32}$): Again we use \eqref{K-1-3}, H\"older's equation and Lemma \ref{L5.1} to find
	\begin{align*}
		\mathcal J_{32} &\leq 2^l p \sum_{j=0}^{l-1} \int_{\bbt\times\bbr\times[1,\infty)} \big| \partial_z^k \partial_\theta^l f \big|^{p-1}   \big| \partial_z^k \partial_\theta^j f\cdot \partial_\theta^{l+1-j}L[f] \big|\rho gdz d\nu d\theta \\
		&= 2^l p \sum_{j=0}^{l-1} \int_{\bbt\times[1,\infty)}\bigg( \int_{\bbr} \big| \partial_z^k \partial_\theta^l f \big|^{p-1}   \big| \partial_z^k \partial_\theta^j f \big|gd\nu \bigg) \big| \partial_\theta^{l+1-j}L[f] \big|\rho dz d\theta \\
		&\leq 2^l pC\kappa \sum_{j=0}^{l-1} \bigg\|\int_{\bbr} \big| \partial_z^k \partial_\theta^l f \big|^{p-1}   \big| \partial_z^k \partial_\theta^j f \big|gd\nu \bigg\|_{L^1_{\theta, z}} = 2^l pC\kappa \sum_{j=0}^{l-1} \bigg\| \big| \partial_z^k \partial_\theta^l f \big|^{p-1}   \big| \partial_z^k \partial_\theta^j f \big| \bigg\|_{L^1_{\theta, \nu, z}} \\
		&\leq 2^l pC\kappa \big\| \partial_z^k \partial_\theta^l f \big\|^{p-1}_{L^p_{\theta, \nu, z}}  \sum_{j=0}^{l-1} \big\| \partial_z^k \partial_\theta^j f \big\|_{L^p_{\theta, \nu, z}} \leq l2^l pC\kappa M_{k,l,p}\big\| \partial_z^k \partial_\theta^l f \big\|^{p-1}_{L^p_{\theta, \nu, z}}.
	\end{align*}

	\noindent$\diamond$ Case B.3 (Estimate of $\mathcal J_{33}$): Note that
	\begin{align*}
		\mathcal J_{33} &\leq 2^{k+l-1} p \sum_{i=0}^{k-1} \sum_{j=0}^{l+1} \int_{\bbt\times\bbr\times[1,\infty)}\big| \partial_z^k \partial_\theta^lf \big|^{p-1}  \big| \partial_z^i \partial_\theta^j f\cdot\partial_z^{k-i} \partial_\theta^{l+1-j}L[f] \big|\rho gdz d\nu d\theta \\
		&\leq 2^{k+l-1} p \big\| \partial_z^k \partial_\theta^l f \big\|_{L^p_{\theta, \nu, z}}^{p-1} \sum_{i=0}^{k-1} \sum_{j=0}^{l+1}\big\| \partial_z^i \partial_\theta^j f \big\|_{L^\infty_{\theta, \nu,z}}\big\| \partial_z^{k-i} \partial_\theta^{l+1-j}L[f] \big\|_{L^p_{\theta, \nu, z}}\\
		&\leq 2^{k+l-1} \pi pC\kappa M_{k,l,p} \big\| \partial_z^k \partial_\theta^l f \big\|_{L^p_{\theta, \nu, z}}^{p-1} \sum_{i=0}^{k-1} \sum_{j=0}^{l+1} 2^{k-i} \big\| f \big\|_{W^{k-i, p}_{\theta, \nu, z}} \\
		&\leq (l+2) (1+2k\pi) 2^{2k+l} \pi pC\kappa M_{k,l,p}^2 \big\| \partial_z^k \partial_\theta^l f \big\|_{L^p_{\theta, \nu, z}}^{p-1},
	\end{align*}
	where we used \eqref{K-1-9} again in the last inequality. \newline

	As a result, we have the desired estimate for finite $p>1$:
	\begin{align}\label{K-1-10}
		\partial_t \|\partial_z^k \partial_\theta^l f\|_{L^p_{\theta, \nu, z}} \leq (l+1) C\kappa \|\partial_z^k \partial_\theta^l f\|_{L^p_{\theta, \nu, z}} +l2^l C\kappa M_{k,l,p} +(l+2) (1+2k\pi) 2^{2k+l} \pi C\kappa M_{k,l,p}^2.
	\end{align}
	By definition, we have
	\[M_{k,l,\infty}=\lim_{p\to\infty}M_{k,l,p},\quad\forall k\ge1,~~l\ge0.\]
	Hence, by taking the limit of $p\to\infty$ on \eqref{K-1-10}, we derive the desired result for $p=\infty$.

\end{document}